\documentclass[11pt]{article}  

\usepackage{amssymb}
\usepackage{amsthm}
\usepackage{amsmath}
\usepackage{graphicx}
\usepackage{subcaption}
\usepackage{setspace}
\usepackage{geometry}
\usepackage{nomencl}
\usepackage{listings}
\usepackage{url}
\usepackage{alphalph}

\newtheorem{defn}{Definition}[section]
\newtheorem{prop}{Proposition}[section]
\newtheorem{cor}{Corollary}[section]
\newtheorem{thm}{Theorem}[section]
\newtheorem{lem}{Lemma}[section]
\newtheorem{eg}{Example}[section]
\newtheorem{rmk}{Remark}[section]

\allowdisplaybreaks

\title{Insertion algorithm for inverting    
        the signature of a path}   
\author{Jiawei Chang\footnote{jiawei.chang@maths.ox.ac.uk, University of Oxford}, Terry Lyons\footnote{tlyons@maths.ox.ac.uk, University of Oxford}}
\begin{document}
\maketitle
\begin{abstract}
In this article we introduce the insertion method for reconstructing the path from its signature, i.e. \emph{inverting the signature of a path}. For this purpose, we prove that a converging upper bound exists for the difference between the inserted $n$-th term and the $(n+1)$-th term of the normalised signature of a smooth path, and we also show that there exists a constant lower bound for a subsequence of the terms in the normalised signature of a piecewise linear path. We demonstrate our results with numerical examples.
\end{abstract}
\section{Motivation}    
The signature of a path was first studied by K.T. Chen (\cite{chen1957integration}, \cite{chen1977iterated}). It can be understood as a collection of non-commutative iterated integrals, and has always been an interesting and essential topic in rough paths theory.\newline
\newline
The signature provides a characteristic description of a path. Chen \cite{chen1958integration} first showed that the non-commutative iterated integrals of a piecewise regular continuous path give a unqiue representation of the path up to some null modifications. Hambly and Lyons \cite{hambly2010uniqueness} furthered the result and showed that this non-commutative transform is faithful for paths of bounded variation up to tree-like pieces.\newline
\newline
Given the fact that the signature of a path is unique up to tree-like pieces \cite{hambly2010uniqueness}, it is an important and natural topic to reconstruct the path from its signature for the completeness of the theory. Lyons and Xu (\cite{lyons2017hyperbolic} and \cite{lyons2014inverting}) developed theories about inverting the signature of a $C^1$ path. Geng investigated more complicated cases and developed a method of inverting the signature of a rough path \cite{geng2017reconstruction}. Pfeffer, Seigal and Sturmfels \cite{pfeffer2018learning} demonstrated a method of computing the shortest path with a given signature level.\newline
\newline
The main aim of this article is therefore to provide practical algorithms for signature inversion for some classes of paths, and hopefully shed light on signature inversion in more complicated cases. We develop a new method of inverting the signature of the path by trying to approximate a level of the signature by a lower level of signature. We justify the motivation of the method by considering the signatures of simple paths, and then we illustrate how we can approximate a path by solving an optimisation problem and demonstrate the method for a particular set of paths.
\section{Introductory examples}
We first introduce the definition of the signature of a path.
\begin{defn}[Signature of a path]\label{signaturedefn}
Let $J$ denote a compact interval, $E$ a Banach space. Let $X:J\to E$ be a continuous path of finite $p$-variation for some $p<2$. The \emph{signature of $X$} is 
\begin{align*}
\mathbf{S}_J(X)=(1,S^1_J(X),S^2_J(X),\cdots),
\end{align*} 
where for each $n\geq 1$, $S^n_J(X)=\int_{u_1<\cdots<u_n\\u_1,\cdots,u_n\in J}\mathrm{d}X_{u_1}\otimes\cdots\otimes\mathrm{d}X_{u_n}$.
\end{defn}
Consider a path $\gamma:[0,T]\rightarrow\mathbb{R}^d$. If $\gamma$ is linear, the signature of $\gamma$ at level $n$ is $S^n_{0,T}(\gamma)=\frac{{(\gamma_T-\gamma_0)}^{\otimes n}}{n!}$ for $n\in\mathbb{N}$, so we have the linear relation
\begin{equation}
S^n_{0,T}(\gamma)\otimes(\gamma_{T}-\gamma_0)=(n+1)S^{n+1}_{0,T}(\gamma).
\end{equation}
Assume instead that $\gamma$ is a piecewise linear path, and is linear on $[0,u]$ and $[u,T]$ respectively for $u\in (0,T)$. Then by Chen's identity, $S^n_{0,T}(\gamma)=\sum_{k=0}^n\frac{{(\gamma_u-\gamma_0)}^{\otimes k}}{k!}\otimes\frac{{(\gamma_T-\gamma_u)}^{\otimes (n-k)}}{(n-k)!}$ for $n\in\mathbb{N}$. We note the following lemma for such a path.
\begin{lem}
Consider a piecewise path $\gamma:[0,T]\rightarrow\mathbb{R}^d$ that is linear on $[0,u]$ and $[u,T]$ respectively for $u\in (0,T)$. Then
\begin{equation}\label{2pieces}
(\gamma_u-\gamma_0)\otimes S_{0,T}^n(\gamma)+S^n_{0,T}(\gamma)\otimes(\gamma_T-\gamma_u)=(n+1)S_{0,T}^{n+1}(\gamma).
\end{equation}
\end{lem}
\begin{proof}
\begin{align*}
&(\gamma_u-\gamma_0)\otimes S_{0,T}^n(\gamma)+S^n_{0,T}\otimes(\gamma_T-\gamma_u)\\
&=\sum_{k=0}^n\frac{{(\gamma_u-\gamma_0)}^{\otimes(k+1)}}{k!}\otimes\frac{{(\gamma_T-\gamma_u)}^{\otimes (n-k)}}{(n-k)!}
+\sum_{k=0}^n\frac{{(\gamma_u-\gamma_0)}^{\otimes k}}{k!}\otimes\frac{{(\gamma_T-\gamma_u)}^{\otimes (n-k+1)}}{(n-k)!}.
\end{align*}
Note for $k=1,\cdots,n$,
\begin{align*}
&\frac{{(\gamma_u-\gamma_0)}^{\otimes k}}{(k-1)!}\otimes\frac{{(\gamma_T-\gamma_u)}^{\otimes (n-k+1)}}{(n-k+1)!}+\frac{{(\gamma_u-\gamma_0)}^{\otimes k}}{k!}\otimes\frac{{(\gamma_T-\gamma_u)}^{\otimes (n-k+1)}}{(n-k)!}\\
=&(n+1)\frac{{(\gamma_u-\gamma_0)}^{\otimes k}\otimes{(\gamma_T-\gamma_u)}^{\otimes (n-k+1)}}{k!(n-k+1)!}, 
\end{align*}
and $$\frac{{(\gamma_u-\gamma_0)}^{\otimes (n+1)}}{n!}=(n+1)\frac{{(\gamma_u-\gamma_0)}^{\otimes (n+1)}}{(n+1)!},$$ $$\frac{{(\gamma_T-\gamma_u)}^{\otimes (n+1)}}{n!}=(n+1)\frac{{(\gamma_T-\gamma_u)}^{\otimes (n+1)}}{(n+1)!}.$$ Then
\begin{align*}
&(\gamma_u-\gamma_0)\otimes S_{0,T}^n(\gamma)+S^n_{0,T}(\gamma)\otimes(\gamma_T-\gamma_u)
=(n+1)S^{n+1}_{0,T}(\gamma).
\end{align*}
\end{proof}
We note that by solving the linear equation (\ref{2pieces}) for $\gamma_u-\gamma_0$ and $\gamma_T-\gamma_u$, we are able to reconstruct exactly the underlying path.\newline
We can now computationally reconstruct a path consisting of two linear pieces. If $\gamma:[0,t]\to\mathbb{R}^d$ is a path consisting of linear pieces, Let the $2d\times d^{n+1}$ matrix $A$ represent the linear mapping $\cdot\otimes S^{n}_{0,T}(\gamma)+S^n_{0,T}(\gamma)\otimes \cdot:(\mathbb{R}^d,\mathbb{R}^d)\to\mathbb{R}^{d^{n+1}}$, and $b\in\mathbb{R}^{d^{n+1}}$ represent $S^{n+1}_{0,T}(\gamma)$. Then Equation (\ref{2pieces}) can be written as, for a vector $X\in(\mathbb{R}^d,\mathbb{R}^d)$, 
\begin{align*}
AX=(n+1)b.
\end{align*}
By using \emph{singular value decomposition(SVD)} on $A$, we can obtain a simple computational algorithm that recovers $\gamma$, as shown in Example \ref{2pieceseg}. Note the computation of the signature in the example is by the C++ package \emph{Libalgebra} \cite{libalgebra}, and the matrix computation is done via \emph{LAPACK} \cite{Anderson:1999:LUG:323215}, the version of LAPACK used is provided by \emph{Intel Math Kernel Library}.
\begin{eg}\label{2pieceseg}
For a two-dimensional path $\gamma:[0,T]\rightarrow\mathbb{R}^2$, $\gamma_t=(x_t, y_t)$ where
\begin{align*}
y=\begin{cases}
  2x & \forall x\in[0,1)\\
  -\frac{2}{3}x+\frac{8}{3} & \forall x\in [1,4].
  \end{cases}
\end{align*}
By using two adjacent levels of the signature of $\gamma$, for instance, the third and fourth levels, we can fully reconstruct the underlying path $\gamma$ by solving Equation (\ref{2pieces}), as shown in Figure \ref{2piecesfig}. 
\end{eg}
\begin{figure}
\includegraphics[trim={4cm 8cm 3cm 10cm},clip,width=0.8\textwidth]{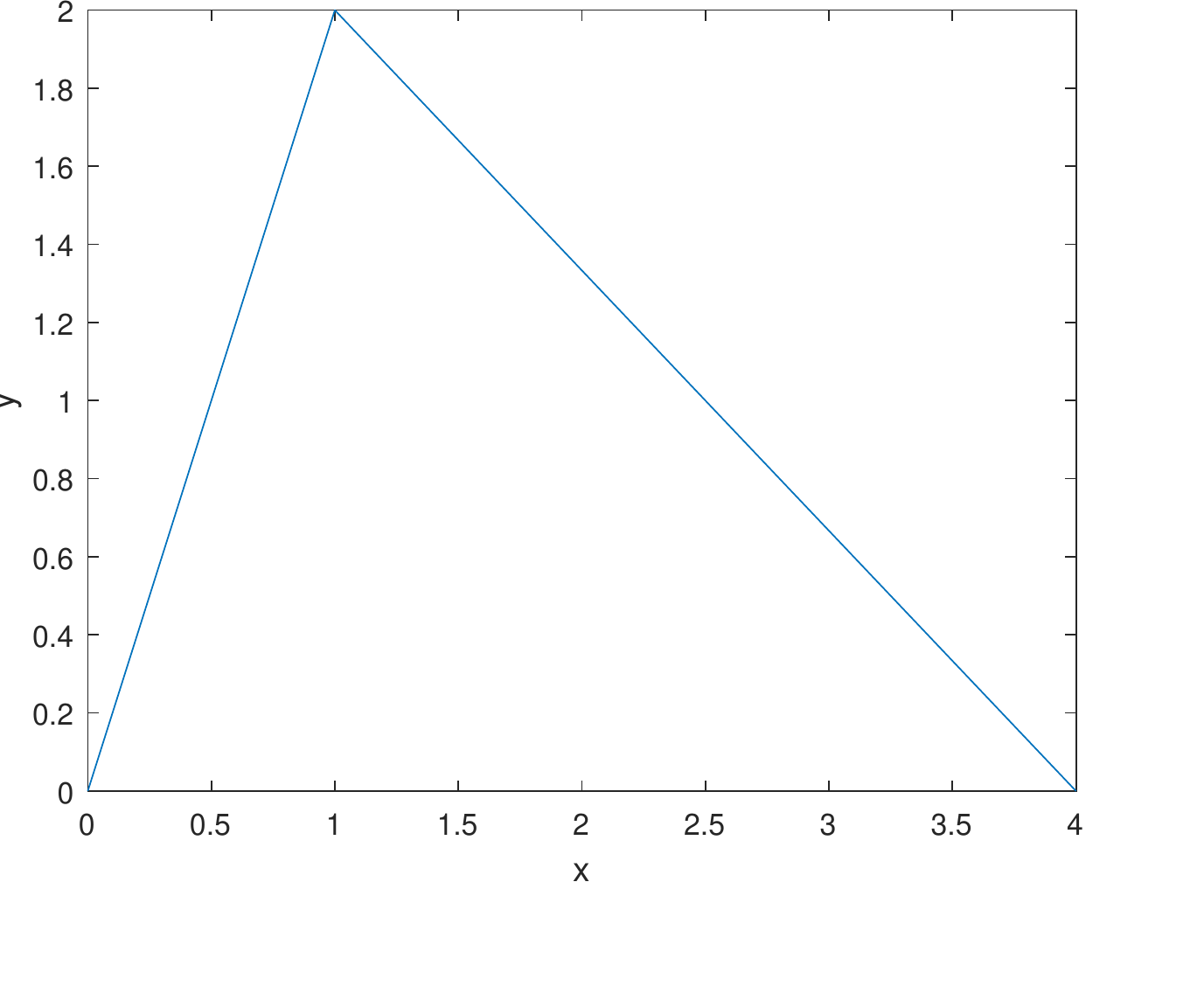}
\centering
\caption{Reconstruction of $y$ as a function of $x$ for the path in Example \ref{2pieceseg}}
\label{2piecesfig}
\end{figure}
Note from above we have shown that for a linear path or a piecewise linear path composed of two linear pieces, we are able to recover the path exactly by comparing two adjacent levels of the signature. This leads to the idea that we may as well get some information about the underlying path if we compare two adjacent levels of the signature of a more complicated path.
\section{A converging upper bound for $\left\lVert I_{p,n}\left(f\left(\theta\right)\right)-\bar{S}_{n+1}\right\rVert$}
From now on we may omit the symbol $\otimes$ in tensor multiplication for simplicity. First we define the properties of admissible norms we assume to be true for the rest of this article.
\begin{defn}\label{norm}
Let $V$ be a Banach space. Suppose the tensor powers are endowed with a tensor norm such that\newline
\begin{enumerate}
\item For all $n\geq1$, the norm of a tensor is invariant under permutation, i.e.
\begin{align*}
\left\lVert\sigma(v)\right\rVert=\left\lVert v\right\rVert\quad\forall v\in V^{\otimes n}, \forall\sigma
\in S(n),
\end{align*}
where $S(n)$ is the symmetric group on $n$ letters;\newline
\item For all $n,m\geq 1$, 
\begin{align*}
\left\lVert v\otimes\omega\right\rVert=\left\lVert v\right\rVert\left\lVert\omega\right\rVert\quad\forall v\in V^{\otimes n}, \omega\in V^{\otimes m}.
\end{align*}
\end{enumerate}
\end{defn}
In the following lemma, we give a collection of norms which satisfy the properties stated in Definition \ref{norm}.
\begin{lem}\label{satifiednorm}
Let $V=\mathbb{R}^d$ with a basis $\left\{e_1,\cdots,e_d\right\}$. Then for any element $u\in V^{\otimes n}$, we can recognise $u$ as a vector in $\mathbb{R}^{d^{n}}$, and in this case, for any $l>0$, $\ell^l$ norm satisfies the properties in Definition \ref{norm}.
\end{lem}
The proof of Lemma \ref{satifiednorm} is straightforward.\newline
\begin{defn}[Normalised signature]
Assume $\gamma$ is a continuous path with bounded variation over the interval $[s,t]$ parametrised at unit speed. Define the \emph{normalised signature} of $\gamma$ over $[s,t]$ as
\begin{align*}
\left(1,\bar{S}^1_{s,t}(\gamma),\bar{S}^2_{s,t}(\gamma),\cdots,\bar{S}^m_{s,t}(\gamma),\cdots\right),
\end{align*}
where for all $m\geq 1$, 
\begin{equation}\label{normalisedsig}
\bar{S}_{s,t}^m(\gamma):=\frac{m!\int_{s<u_1<\cdots<u_m<t}f(u_1)\cdots f(u_m)\mathrm{d}u_1\cdots\mathrm{d}u_m}{{(t-s)}^m}.
\end{equation}
\end{defn}
For simplicity, we write $\bar{S}_m:=\bar{S}_{0,1}^m(\gamma)$.
\begin{defn}[Insertion map]
Assume $\gamma:[0,1]\to\mathbb{R}^d$ is a continuous path of bounded variation parametrised at unit speed. For $p=1,\cdots,n+1$, define the mapping function $I_{p,n}:\mathbb{R}^d\rightarrow{(\mathbb{R}^d)}^{\otimes (n+1)}$ by
\begin{equation}\label{insertionmap}
I_{p,n}(x):=\int_{0<u_1<\cdots<u_n<1}f(u_1)\cdots f(u_{p-1})xf(u_p)\cdots f(u_n)n!\mathrm{d}u_1\cdots\mathrm{d}u_n,
\end{equation}
i.e. $I_{p,n}(x)$ is the function that inserts $x$ into the $p$-th position of the $n$-th normalised signature. Note that the operation of inserting $x\in\mathbb{R}^d$ into a homogeneous tensor $t\in{\mathbb{R}^d}^{\otimes n}$ at $p$-th position is well-defined.
\end{defn}
Similarly, for $p=1,\cdots,n+1$, define the mapping $R_{p,n+1}:\mathbb{R}^d\rightarrow{(\mathbb{R}^d)}^{\otimes (n+1)}$ by
\begin{equation}\label{replacemap}
R_{p,n+1}(x):=\int_{0<u_1<\cdots<u_{n+1}<1}f(u_1)\cdots f(u_{p-1})xf(u_{p+1})\cdots f(u_{n+1})(n+1)!\mathrm{d}u_1\cdots\mathrm{d}u_{n+1},
\end{equation}
i.e. $R_{p,n+1}$ replaces the $p$-th element of the $(n+1)$-th normalised signature by $x$.\newline
A simple observation is that the function $I_{p,n}$ is linear, as stated in the following lemma.
\begin{lem}
Assume $\gamma:[0,1]\rightarrow\mathbb{R}^d$ is differentiable almost surely, and the derivative $f:(0,1)\to\mathbb{R}^d$ satisfies $\left\lVert f(t)\right\rVert=1$ if defined. For all $a,b\in\mathbb{R}$, $x,y\in\mathbb{R}^d$, $n\geq1$, $p\in\{1,\cdots,n+1\}$, $I_{p,n}\left(ax+by\right)=aI_{p,n}(x)+bI_{p,n}(y)$. 
\end{lem}
\begin{proof}
\begin{align*}
&I_{p,n}\left(ax+by\right)\\
=&\int_{0<u_1<\cdots<u_n<1}f(u_1)\cdots f(u_{p-1})(ax+by)f(u_p)\cdots f(u_n)n!\mathrm{d}u_1\cdots\mathrm{d}u_n\\
=&a\int_{0<u_1<\cdots<u_n<1}f(u_1)\cdots f(u_{p-1})xf(u_p)\cdots f(u_n)n!\mathrm{d}u_1\cdots\mathrm{d}u_n\\
&+b\int_{0<u_1<\cdots<u_n<1}f(u_1)\cdots f(u_{p-1})yf(u_p)\cdots f(u_n)n!\mathrm{d}u_1\cdots \mathrm{d}u_n\\
=&aI_{p,n}(x)+bI_{p,n}(y).
\end{align*}
\end{proof}
Because of the properties of the norm stated in Definition \ref{norm}, we are able to state the following property of the distances between images of the map $I_{p,n}$ which we will use later.
\begin{lem}\label{normrelation}
Assume $\gamma:[0,1]\rightarrow\mathbb{R}^d$ is differentiable almost everywhere with derivative $f:(0,1)\rightarrow\mathbb{R}^d$ such that $\left\lVert f(t)\right\rVert=1$ if defined. Then following the properties of the norm defined in Definition \ref{norm}, for any $x,y\in\mathbb{R}^d$, $n\geq 1$, $p\in\{1,\cdots,n+1\}$, 
\begin{equation}\label{normequality}
\left\lVert I_{p,n}(x)-I_{p,n}(y)\right\rVert=\left\lVert\bar{S}_n\right\rVert\left\lVert x-y\right\rVert.
\end{equation}
\end{lem}
\begin{proof}
\begin{align*}
&\left\lVert I_{p,n}(x)-I_{p,n}(y)\right\rVert\\
=&\left\lVert\int_{0<u_1<\cdots<u_n<1}f(u_1)\cdots(x-y)\cdots f(u_n)n!\mathrm{d}u_1\cdots\mathrm{d}u_n\right\rVert\\
=&\left\lVert\int_{0<u_1<\cdots<u_n<1}f(u_1)\cdots f(u_n)n!\mathrm{d}u_1\cdots\mathrm{d}u_n(x-y)\right\rVert\\
=&\left\lVert\int_{0<u_1<\cdots<u_n<1}f(u_1)\cdots f(u_n)n!\mathrm{d}u_1\cdots\mathrm{d}u_n\right\rVert\left\lVert x-y\right\rVert\\
=&\left\lVert\bar{S}_n\right\rVert\left\lVert x-y\right\rVert.
\end{align*}
\end{proof}
\begin{cor}
The function $I_{p,n}$ is Lipschitz continuous.
\end{cor}
\begin{proof}
It is a direct consequence of Equation (\ref{normequality}) that $I_{p,n}$ is Lipschitz.
\end{proof}
Intuitively it is reasonable to expect that if the derivative of the underlying path is inserted at the `correct' position into the $n$-th term in the normalised signature, the resulting tensor shall be a well-behaved approximation of the $(n+1)$-th term in the normalised signature, and as we have a finer partition of the interval, the approximation should be more accurate. We first note the following theorem by Hoeffding \cite{hoeffding1963probability}.
\begin{thm}[Hoeffding's inequality \cite{hoeffding1963probability}]Let $X_1,\cdots, X_n$ be independent random variables strictly bounded by the intervals $[a_i,b_i]$ respectively, define $S_n=\sum_{i=1}^nX_i$. Then for any $t>0$,
\begin{align*}
\mathbb{P}\left(S_n-\mathbb{E}[S_n]\geq t\right)\leq\exp\left(-\frac{2t^2}{\sum_{i=1}^n(b_i-a_i)^2}\right),
\end{align*}
\begin{align*}
\mathbb{P}\left(\left|S_n-\mathbb{E}[S_n]\right|\geq t\right)\leq 2\exp\left(-\frac{2t^2}{\sum_{i=1}^n(b_i-a_i)^2}\right).
\end{align*}
\end{thm}
Notice since a binomial variable is a sum of independent Bernoulli variables, Hoeffding's inequality applies to binomial variables. We note the following example.
\begin{eg}\label{convergeuperboundanaeg}
Assume $\tilde{\gamma}:[0,1]\to\mathbb{R}^d$ is a linear path with derivative $g:(0,1)\to\mathbb{R}^d$ such that $\left\lVert g(t)\right\rVert=1$ for all $t\in(0,1)$. Then for any $\theta\in(0,1)$ and $q=\left\lfloor\theta(n+2)\right\rfloor$, 
\begin{align*}
\left\lVert I_{q,n}(g(\theta))-\bar{S}_{n+1}\right\rVert=\left\lVert n!\frac{\left(g(\theta)\right)^{\otimes (n+1)}}{n!}-(n+1)!\frac{\left(g(\theta)\right)^{\otimes (n+1)}}{(n+1)!}\right\rVert=0,
\end{align*}
therefore
\begin{align*}
\left\lVert I_{q,n}(g(\theta))-\bar{S}_{n+1}\right\rVert\to 0\quad\text{as}\quad n\to\infty.
\end{align*}
Now let us consider a slightly more complicated case. Assume $\{e_1,e_2\}$ is a basis of $\mathbb{R}^2$ and $\gamma:[0,1]\to\mathbb{R}^2$ is a piecewise linear path such that 
\begin{align*}
\gamma(t)= \begin{cases} 
      te_2 & t\in[0,\frac{2}{3}] \\
      (t-\frac{2}{3})e_1+\frac{2}{3}e_2 & t\in(\frac{2}{3},1].
   \end{cases}
\end{align*}
Note that the derivative $f:(0,1)\to\mathbb{R}^2$ of $\gamma$ satisfies $\lVert f(t)\rVert=1$ for all $t\in (0,1)$ where $f$ is defined. Note in this case,
\begin{align*}
\bar{S}_n=n!\sum_{k=0}^n\left(\frac{2}{3}\right)^k\frac{e_2^{\otimes k}}{k!}\otimes\left(\frac{1}{3}\right)^{n-k}\frac{e_1^{\otimes (n-k)}}{(n-k)!}.
\end{align*}
Note that if we choose $\theta=1/2$ and $p=\left\lfloor\theta(n+2)\right\rfloor$, then $f(\frac{1}{2})=e_2$, and we can write 
\begin{align*}
I_{p,n}\left(f\left(\frac{1}{2}\right)\right)=&n!\sum_{k=p-1}^n\left(\frac{2}{3}\right)^k\frac{e_2^{\otimes (k+1)}}{k!}\left(\frac{1}{3}\right)^{n-k}\frac{e_1^{\otimes (n-k)}}{(n-k)!}\\
&+n!\sum_{k=0}^{p-2}\left(\frac{2}{3}\right)^k\frac{e_2^{\otimes k}}{k!}\left(\frac{1}{3}\right)^{n-k}\frac{e_1^{\otimes (p-1-k)}\otimes e_2\otimes e_1^{\otimes (n-p+1)}}{(n-k)!},
\end{align*}
then
\begin{align*}
&I_{p,n}\left(f\left(\frac{1}{2}\right)\right)-\bar{S}_{n+1}\\
=&\sum_{k=p-1}^n\left(\frac{2}{3}\right)^{k+1}\left(\frac{1}{3}\right)^{n-k}\frac{(n+1)!}{(k+1)!(n-k)!}\left(\frac{3}{2}\frac{k+1}{n+1}-1\right)e_2^{\otimes (k+1)}e_1^{\otimes (n-k)}\\
&+\sum_{k=0}^{p-2}\left(\frac{2}{3}\right)^k\left(\frac{1}{3}\right)^{n-k}\frac{n!}{k!(n-k)!}e_2^{\otimes k}e_1^{\otimes (p-1-k)}\otimes e_2\otimes e_1^{\otimes (n-p+1)}\\
&-\sum_{k=0}^{p-1}\left(\frac{2}{3}\right)^k\left(\frac{1}{3}\right)^{n+1-k}\frac{(n+1)!}{k!(n+1-k)!}e_2^{\otimes k}e_1^{\otimes (n+1-k)}.
\end{align*}
Hence
\begin{align}\label{exampleconvergenceupperboundinsertionmiddle}
\left\lVert I_{p,n}\left(f\left(\frac{1}{2}\right)\right)-\bar{S}_{n+1}\right\rVert
\leq&\sum_{k=p}^{n+1}\left(\frac{2}{3}\right)^k\left(\frac{1}{3}\right)^{n+1-k}\frac{(n+1)!}{k!(n+1-k)!}\left|\frac{3}{2}\frac{k}{n+1}-1\right|\nonumber\\
&+\sum_{k=0}^{p-2}\left(\frac{2}{3}\right)^k\left(\frac{1}{3}\right)^{n-k}\frac{n!}{k!(n-k)!}\nonumber\\
&+\sum_{k=0}^{p-1}\left(\frac{2}{3}\right)^k\left(\frac{1}{3}\right)^{n+1-k}\frac{(n+1)!}{k!(n+1-k)!}.
\end{align}
Let us investigate the binomial sums on the right-hand side of (\ref{exampleconvergenceupperboundinsertionmiddle}) respectively. Note 
\begin{align}\label{binomialeg}
&\sum_{k=p}^{n+1}\left(\frac{2}{3}\right)^k\left(\frac{1}{3}\right)^{n+1-k}\frac{(n+1)!}{k!(n+1-k)!}\left|\frac{3}{2}\frac{k}{n+1}-1\right|\nonumber\\
\leq&\sum_{k=0}^{n+1}\left(\frac{2}{3}\right)^k\left(\frac{1}{3}\right)^{n+1-k}\frac{(n+1)!}{k!(n+1-k)!}\left|\frac{3}{2}\frac{k}{n+1}-1\right|\nonumber\\
=&\sum_{\substack{t=k/(n+1)\\k=0,\cdots,n+1}}\left(\frac{2}{3}\right)^{(n+1)t}\left(\frac{1}{3}\right)^{(n+1)(1-t)}\frac{(n+1)!}{\left((n+1)t\right)!\left((n+1)(1-t)\right)!}\left|\frac{3}{2}t-1\right|.
\end{align}
Assume $X\sim$Binomial$(n+1,2/3)$. Note that (\ref{binomialeg}) is the expectation of a function of the random variable $Y:=X/(n+1)$, and
\begin{align*}
\mathbb{E}[Y]=\frac{1}{n+1}\mathbb{E}[X]=\frac{2}{3},\quad \text{Var}[Y]=\frac{1}{(n+1)^2}\text{Var}[X]=\frac{2}{9}\frac{1}{n+1},
\end{align*}
hence we can see that as $n$ increases, the distribution of $Y$ will be mostly concentrated around $\mathbb{E}[Y]$, and since $\left|\frac{3}{2}Y-1\right|=0$ at $Y=\mathbb{E}[Y]$, the value of the sum of (\ref{binomialeg}) would converge to zero as $n$ increases. For a more formal argument, we have for any $\lambda>0$, 
\begin{align*}
&\mathbb{P}\left(\left|Y-\frac{2}{3}\right|\geq\frac{\lambda}{3}\sqrt{\frac{2}{n+1}}\right)\\
=&\mathbb{P}\left(\left|X-\frac{2}{3}(n+1)\right|\geq\frac{\lambda}{3}\sqrt{2(n+1)}\right)\\
\leq&2\exp\left(-\frac{4}{9}\lambda^2\right),
\end{align*}
where the last inequality comes from Hoeffding's inequality.\newline
Then for any $\lambda>0$,
\begin{align*}
&\sum_{k=p}^{n+1}\left(\frac{2}{3}\right)^k\left(\frac{1}{3}\right)^{n+1-k}\frac{(n+1)!}{k!(n+1-k)!}\left|\frac{3}{2}\frac{k}{n+1}-1\right|\\
\leq&\sum_{\substack{t=k/(n+1)\\k=0,\cdots,n+1\\\left|t-\frac{2}{3}\right|<\frac{\lambda}{3}\sqrt{\frac{2}{n+1}}}}\left(\frac{2}{3}\right)^{(n+1)t}\left(\frac{1}{3}\right)^{(n+1)(1-t)}\frac{(n+1)!}{((n+1)t)!((n+1)(1-t))!}\left|\frac{3}{2}t-1\right|\\
&+\sum_{\substack{t=k/(n+1)\\k=0,\cdots,n+1\\\left|t-\frac{2}{3}\right|\geq\frac{\lambda}{3}\sqrt{\frac{2}{n+1}}}}\left(\frac{2}{3}\right)^{(n+1)t}\left(\frac{1}{3}\right)^{(n+1)(1-t)}\frac{(n+1)!}{((n+1)t)!((n+1)(1-t))!}\left|\frac{3}{2}t-1\right|\\
\leq&\frac{\lambda}{\sqrt{2(n+1)}}+2\mathbb{P}\left(\left|Y-\frac{2}{3}\right|\geq\frac{\lambda}{3}\sqrt{\frac{2}{n+1}}\right)\\
\leq&\frac{\lambda}{\sqrt{2(n+1)}}+4\exp\left(-\frac{4}{9}\lambda^2\right).
\end{align*}
Note that for a given $n>0$, $\lambda/\sqrt{2(n+1)}$ is a strictly increasing linear function of $\lambda$ from $(0,\infty)$ to $(0,\infty)$, and $4\exp\left(-\frac{4}{9}\lambda^2\right)$ is a strictly decreasing function of $\lambda$ from $(0,\infty)$ to $(0,4)$. So for each $n$, there exists $\lambda_n>0$ such that 
\begin{align}\label{satisfyinglambda}
\frac{\lambda_n}{\sqrt{2(n+1)}}=4\exp\left(-\frac{4}{9}\lambda_n^2\right).
\end{align}
Differentiating (\ref{satisfyinglambda}) with respect to $n$ gives
\begin{align*}
\frac{\partial \lambda_n}{\partial n}(2(n+1))^{\frac{1}{2}}-\lambda_n(2(n+1))^{-\frac{3}{2}}=-\frac{32}{9}\frac{\partial \lambda_n}{\partial n}\lambda_n\exp\left(-\frac{4}{9}\lambda_n^2\right),
\end{align*}
\begin{align*}
\frac{\partial\lambda_n}{\partial n}\left(\left(2(n+1)\right)^{\frac{1}{2}}+\frac{32}{9}\lambda_n\exp\left(-\frac{4}{9}\lambda_n^2\right)\right)=\lambda_n\left(2(n+1)\right)^{-\frac{3}{2}},
\end{align*}
therefore $\frac{\partial \lambda_n}{\partial n}>0$, and $\lambda_n$ is a strictly increasing function in $n$ which tends to infinity, so $4\exp\left(-4/9\lambda^2_n\right)$ is a strictly decreasing function in $n$, and
\begin{align*}
\sum_{k=p}^{n+1}\left(\frac{2}{3}\right)^k\left(\frac{1}{3}\right)^{n+1-k}\frac{(n+1)!}{k!(n+1-k)!}\left|\frac{3}{2}\frac{k}{n+1}-1\right|\leq&\frac{2\lambda_n}{\sqrt{2(n+1)}}\\
=&8\exp\left(-\frac{4}{9}\lambda^2_n\right)\\
\to& 0
\end{align*}
as $n$ goes to infinity.\newline
For the other two binomial sums in (\ref{exampleconvergenceupperboundinsertionmiddle}), by Hoeffding's inequality, 
\begin{align*}
&\sum_{k=0}^{p-2}\left(\frac{2}{3}\right)^k\left(\frac{1}{3}\right)^{n-k}\frac{n!}{k!(n+1-k)!}+\sum_{k=0}^{p-1}\left(\frac{2}{3}\right)^{k}\left(\frac{1}{3}\right)^{n+1-k}\frac{(n+1)!}{k!(n+1-k)!}\\
\leq&\exp\left(-C_1n+B_1\right)+\exp\left(-C_2n+B_2\right),
\end{align*}
where $C_1>0$, $C_2>0$, $B_1$ and $B_2$ are some constants. Therefore
\begin{align*}
\left\lVert I_{p,n}\left(\frac{1}{2}\right)-\bar{S}_{n+1}\right\rVert\leq\frac{2\lambda_n}{\sqrt{2(n+1)}}+\exp\left(-C_1n+B_1\right)+\exp\left(-C_2n+B_2\right)\to 0
\end{align*}
as $n\to\infty$. Note since $2\lambda_n/\sqrt{2(n+1)}$ is decreasing in $n$, and $\lambda_n$ is increasing in $n$, the rate of increasing of $\lambda_n$ must be smaller than the increasing rate of $\sqrt{2(n+1)}$, therefore the upper bound obtained for $\left\lVert I_{p,n}\left(\frac{1}{2}\right)-\bar{S}_{n+1}\right\rVert$ decreases at a rate slower than $O(1/\sqrt{n})$.
\end{eg}
Example \ref{convergeuperboundanaeg} has inspired us that we can use the tail behaviour of the binomial distribution to prove that $\left\lVert I_{p,n}\left(f(\theta)\right)-\bar{S}_{n+1}\right\rVert$ converges to $0$.
\begin{thm}\label{thmconvergetozeroupperbound}
Suppose $\gamma:[0,1]\to\mathbb{R}^d$ is a tree-reduced continuous bounded-variation path, and the derivative of $\gamma$ $f:(0,1)\to\mathbb{R}^d$ is defined almost everywhere, and $\lVert f(t)\rVert=1$ for all $t\in(0,1)$ if defined. Assume $\gamma$ is linear on $[s,t]$ for $0\leq s<t\leq 1$, and let $\theta\in(s,t)$. If we choose $p=\left\lfloor\theta(n+2)\right\rfloor$, then
\begin{align*}
\left\lVert I_{p,n}\left(f(\theta)\right)-\bar{S}_{n+1}\right\rVert\to 0\quad\text{as}\quad n\to\infty,
\end{align*}
and the rate of convergence of the upper bound obtained for $\left\lVert I_{p,n}\left(f(\theta)\right)-\bar{S}_{n+1}\right\rVert$ is slower than $O\left(1/\sqrt{n+1}\right)$.
\end{thm}
\begin{proof}
Since $\gamma$ is tree-reduced, by Boedihardjo and Geng \cite{boedihardjo2018non}, there exists $N\in\mathbb{N}$ such that for all $n\geq N$, $S^n_{0,1}(\gamma)\neq 0$. We now only consider the case when $n\geq N$.\newline
By Chen's identity,
\begin{align*}
\bar{S}_n&=n!\sum_{k_1+k_2+k_3=n}S^{k_1}_{0,s}(\gamma)\otimes S^{k_2}_{s,t}(\gamma)\otimes S^{k_3}_{t,1}(\gamma)\\
&=n!\sum_{k_1+k_2+k_3=n}S^{k_1}_{0,s}(\gamma)\otimes\frac{(t-s)^{k_2}{f(\theta)}^{\otimes k_2}}{k_2!}\otimes S^{k_3}_{t,1}(\gamma).
\end{align*}
Define the sets $$I:=\left\{(k_1,k_2,k_3):k_1+k_2+k_3=n,k_i=0,\cdots, n,\, i=1,2,3\right\}$$ and $$J:=\left\{(k_1,k_2,k_3):(k_1,k_2,k_3)\in I, k_1\leq p-1,k_3\leq n+1-p\right\}.$$ For $x\in\mathbb{R}^d$, let $x\uparrow S^k_{u,v}(\gamma)$ denote the resulting tensor of inserting $x$ into any position of $S^k_{u,v}(\gamma)$ for any $0\leq u<v\leq 1$. Then
\begin{align*}
I_{p,n}\left(f(\theta)\right)=&\sum_{(k_1,k_2,k_3)\in J}n!S^{k_1}_{0,s}(\gamma)\otimes\frac{(t-s)^{k_2}{f(\theta)}^{\otimes k_2+1}}{k_2!}\otimes S^{k_3}_{t,1}(\gamma)\\
&+\sum_{(k_1,k_2,k_3)\in I, k_1\geq p}n!\left(f(\theta)\uparrow S^{k_1}_{0,s}(\gamma)\right)\otimes\frac{(t-s)^{k_2}{f(\theta)}^{\otimes k_2}}{k_2!}\otimes S^{k_3}_{t,1}(\gamma)\\
&+\sum_{(k_1,k_2,k_3)\in I, k_3\geq n-p+2}n!S^{k_1}_{s,t}(\gamma)\otimes\frac{(t-s)^{k_2}{f(\theta)}^{\otimes k_2}}{k_2!}\otimes\left(f(\theta)\uparrow S^{k_3}_{t,1}(\gamma)\right).
\end{align*}
Note also
\begin{align*}
\bar{S}_{n+1}=&\sum_{k_1+k_2+k_3=n+1}(n+1)!S^{k_1}_{0,s}(\gamma)\otimes\frac{(t-s)^{k_2}{f(\theta)}^{\otimes k_2}}{k_2!}\otimes S^{k_3}_{t,1}(\gamma)\\
=&\sum_{(k_1,k_2,k_3)\in J}(n+1)!S^{k_1}_{0,s}(\gamma)\frac{(t-s)^{k_2+1}f(\theta)^{\otimes k_2+1}}{(k_2+1)!}\otimes S^{k_3}_{t,1}(\gamma)\\
&+\sum_{\substack{k_1+k_2+k_3=n+1\\k_1\geq p}}(n+1)!S^{k_1}_{0,s}(\gamma)\otimes\frac{(t-s)^{k_2}{f(\theta)}^{\otimes k_2}}{k_2!}\otimes S^{k_3}_{t,1}(\gamma)\\
&+\sum_{\substack{k_1+k_2+k_3=n+1\\k_3\geq n-p+2}}(n+1)!S^{k_1}_{0,s}(\gamma)\otimes\frac{(t-s)^{k_2}{f(\theta)}^{\otimes k_2}}{k_2!}\otimes S^{k_3}_{t,1}(\gamma).
\end{align*}
Then since for any $0\leq u<v\leq 1$ and any $k\geq 1$, $\left\lVert S^k_{u,v}(\gamma)\right\rVert\leq (v-u)^k/k!$, we have
\begin{align}\label{upperboundforbinomialsums}
&\left\lVert I_{p,n}\left(f(\theta)\right)-\bar{S}_{n+1}\right\rVert\nonumber\\
\leq &\sum_{0\leq k\leq n+1}\frac{(n+1)!}{k!(n+1-k)!}(t-s)^{k}(1-t+s)^{n+1-k}\left|\frac{k}{n+1}\frac{1}{t-s}-1\right|\nonumber\\
&+\sum_{p\leq k\leq n}\frac{n!}{k!(n-k)!}s^k(1-s)^{n-k}\nonumber\\
&+\sum_{n-p+2\leq k\leq n}\frac{n!}{k!(n-k)!}(1-t)^kt^{n-k}\nonumber\\
&+\sum_{p\leq k\leq n+1}\frac{(n+1)!}{k!(n+1-k)!}s^k(1-s)^{n+1-k}\nonumber\\
&+\sum_{n-p+2\leq k\leq n+1}\frac{(n+1)!}{k!(n+1-k)!}(1-t)^kt^{n+1-k}.
\end{align}
Note
\begin{align}\label{functionindepofn}
&\sum_{0\leq k\leq n+1}\frac{(n+1)!}{k!(n+1-k)!}(t-s)^{k}(1-t+s)^{n+1-k}\left|\frac{k}{n+1}\frac{1}{t-s}-1\right|\nonumber\\
=&\sum_{\substack{r=k/(n+1)\\k=0,\cdots, n+1}}\frac{(n+1)!}{\left((n+1)r\right)!\left((n+1)(1-r)\right)!}(t-s)^{(n+1)r}(1-t+s)^{(n+1)(1-r)}\left|\frac{1}{t-s}r-1\right|,
\end{align}
then if we assume $X\sim$Binomial$(n+1,t-s)$ and define $Y:=X/(n+1)$, notice that (\ref{functionindepofn}) is the expectation of the function $\left |Y/(t-s)-1\right|$, and
\begin{align*}
\mathbb{E}[Y]=t-s,\quad \text{Var}[Y]=\frac{(t-s)(1-t+s)}{n+1}.
\end{align*}
By Hoeffding's inequality, for any $\lambda>0$, 
\begin{align*}
&\mathbb{P}\left(\left|Y-(t-s)\right|\geq\lambda\sqrt{\frac{(t-s)(1-t+s)}{n+1}}\right)\\
=&\mathbb{P}\left(\left|X-(t-s)(n+1)\right|\geq\lambda\sqrt{(t-s)(1-t+s)(n+1)}\right)\\
\leq&2\exp\left(-2(t-s)(1-t+s)\lambda^2\right),
\end{align*}
therefore by considering the cases when $\left|Y-(t-s)\right|<\lambda\sqrt{(t-s)(1-t+s)/(n+1)}$ and $\left|Y-(t-s)\right|\geq\lambda\sqrt{(t-s)(1-t+s)/(n+1)}$ respectively, we have
\begin{align*}
&\sum_{0\leq k\leq n+1}\frac{(n+1)!}{k!(n+1-k)!}(t-s)^{k}(1-t+s)^{n+1-k}\left|\frac{k}{n+1}\frac{1}{t-s}-1\right|\\
\leq&\lambda\sqrt{\frac{1-t+s}{(t-s)(n+1)}}\\
+&\max\left(1,\left(\frac{1}{t-s}-1\right)\right)\mathbb{P}\left(\left| Y-(t-s)\right|\geq \lambda\sqrt{\frac{(t-s)(1-t+s)}{n+1}}\right)\\
\leq&\lambda\sqrt{\frac{1-t+s}{(t-s)(n+1)}}+2\max\left(1,\left(\frac{1}{t-s}-1\right)\right)\exp\left(-2(t-s)(1-t+s)\lambda^2\right).
\end{align*}
By similar arguments as in Example \ref{convergeuperboundanaeg}, there exists a strictly increasing sequence $\left(\lambda_n\right)_n$ such that for each $n\in\mathbb{N}$, $\lambda_n>0$, and 
\begin{equation}\label{limitofparameter}
\lambda_n\sqrt{\frac{1-t+s}{(t-s)(n+1)}}=2\max\left(1,\left(\frac{1}{t-s}-1\right)\right)\exp\left(-2(t-s)(1-t+s)\lambda_n^2\right).
\end{equation}
Suppose $\lambda_n\to\lambda^*<\infty$ as $n\to\infty$. Then taking limits on both sides of Equation (\ref{limitofparameter}) gives
\begin{align*}
0=2\max\left(1,\left(\frac{1}{t-s}-1\right)\right)\exp\left(-2(t-s)(1-t+s){\lambda^*}^2\right),
\end{align*}
which does not hold if $\lambda^*$ is finite. Hence we must have $\lambda^*=\infty$. Therefore
\begin{align*}
&\sum_{0\leq k\leq n+1}\frac{(n+1)!}{k!(n+1-k)!}(t-s)^{k}(1-t+s)^{n+1-k}\left|\frac{k}{n+1}\frac{1}{t-s}-1\right|\\
\leq &2\lambda_n\sqrt{\frac{1-t+s}{(t-s)(n+1)}}\\
=&4\max\left(1,\left(\frac{1}{t-s}-1\right)\right)\exp\left(-2(t-s)(1-t+s)\lambda_n^2\right)\to 0
\end{align*}
as $n\to\infty$. If we define $X_1\sim$Binomial$(n,s)$, $X_2\sim$Binomial$(n,1-t)$, $X_3\sim$Binomial$(n+1,s)$, and $X_4\sim$Binomial$(n+1,1-t)$, by applying Hoeffding's inequality to the other binomial sums on the right-hand side of (\ref{upperboundforbinomialsums}), there exists $M\in\mathbb{N}$ such that for all $n\geq M$, 
\begin{align*}
&\left\lVert I_{p,n}\left(f(\theta)\right)-\bar{S}_{n+1}\right\rVert\\
\leq & 2\lambda_n\sqrt{\frac{1-t+s}{(t-s)(n+1)}}+\mathbb{P}\left(X_1\geq p\right)\\
&+\mathbb{P}\left(X_2\geq n-p+2\right)+\mathbb{P}\left(X_3\geq p\right)+\mathbb{P}\left(X_4\geq n-p+2\right)\\
\leq &2\lambda_n\sqrt{\frac{1-t+s}{(t-s)(n+1)}}+\exp\left(-C_1n+B_1\right)\\
&+\exp\left(-C_2n+B_2\right)+\exp\left(-C_3n+B_3\right)+\exp\left(-C_4n+B_4\right),
\end{align*}
where we have used Hoeffding's inequality and $C_1>0, C_2>0, C_3>0, C_4>0$, $B_1,B_2,B_3$ and $B_4$ are some constants. Therefore
\begin{align*}
\left\lVert I_{p,n}\left(f(\theta)\right)-\bar{S}_{n+1}\right\rVert\to 0\quad\text{as}\quad n\to\infty.
\end{align*}
Moreover, the rate of convergence of the upper bound obtained for $\left\lVert I_{p,n}\left(f(\theta)\right)-\bar{S}_{n+1}\right\rVert$ is slower than $O\left(1/\sqrt{(t-s)(n+1)}\right)$.
\end{proof}
We also have the following example as a numerical demonstration of our claim. 
\begin{eg}
Assume $\gamma\in\mathbb{R}^2$ is a piecewise linear path which is an approximation to the quadratic path over the unit interval $[0,1]$ parametrised at unit speed, i.e. $\gamma^{(2)}(t)=(\gamma^{(1)}(t))^2$ for all $t\in[0,1]$. Fixing $\theta=0.3$, if we compute the difference $\left\rVert I_{p,n}(f(\theta))-\bar{S}_{n+1}\right\rVert$ under the $\ell^1$ norm and $\ell^2$ norm, we obtain Figure \ref{nextleveldiffl1} and \ref{nextleveldiffl2}. From the figures we can see that under both the $\ell^1$ norm and $\ell^2$ norm, the difference between the $n$-th term in the normalised signature with $f(\theta)$ inserted at the $p$-th position and the $(n+1)$-th term in the normalised signature decreases as $n$ increases, although not monotonically. The reason for the non-monotonicity is that given  $p=\left\lfloor\theta(n+2)\right\rfloor$, $I_{p,n}(f(p/(n+2)))$ is a good approximation of $\bar{S}_{n+1}$, while $I_{p,n}(f(\theta))$ may not be as a good approximation as $I_{p,n}(f(p/(n+2)))$ if $\theta(n+2)$ is not an integer, hence we may observe small increases when $p/(n+2)$ is a bit far from $\theta$.\newline
As a justification, we now choose $\theta=0.5$ and $p=\left\lfloor0.5(n+2)\right\rfloor$ for $n=4,6,8,10,12$, and plot $\left\lVert I_{p,n}(f(\theta))-\bar{S}_{n+1}\right\rVert_1$ in Figure \ref{nextleveldiffl10.5}. In this case since the signature level $n$ used is even, $0.5(n+2)$ is an integer, and from the figure we can see that we get monotone convergence as $n$ increases in this case. 
\end{eg}
\begin{figure}[h!]
\centering
\includegraphics[trim={2cm 8cm 3cm 6cm},clip,width=0.8\textwidth]{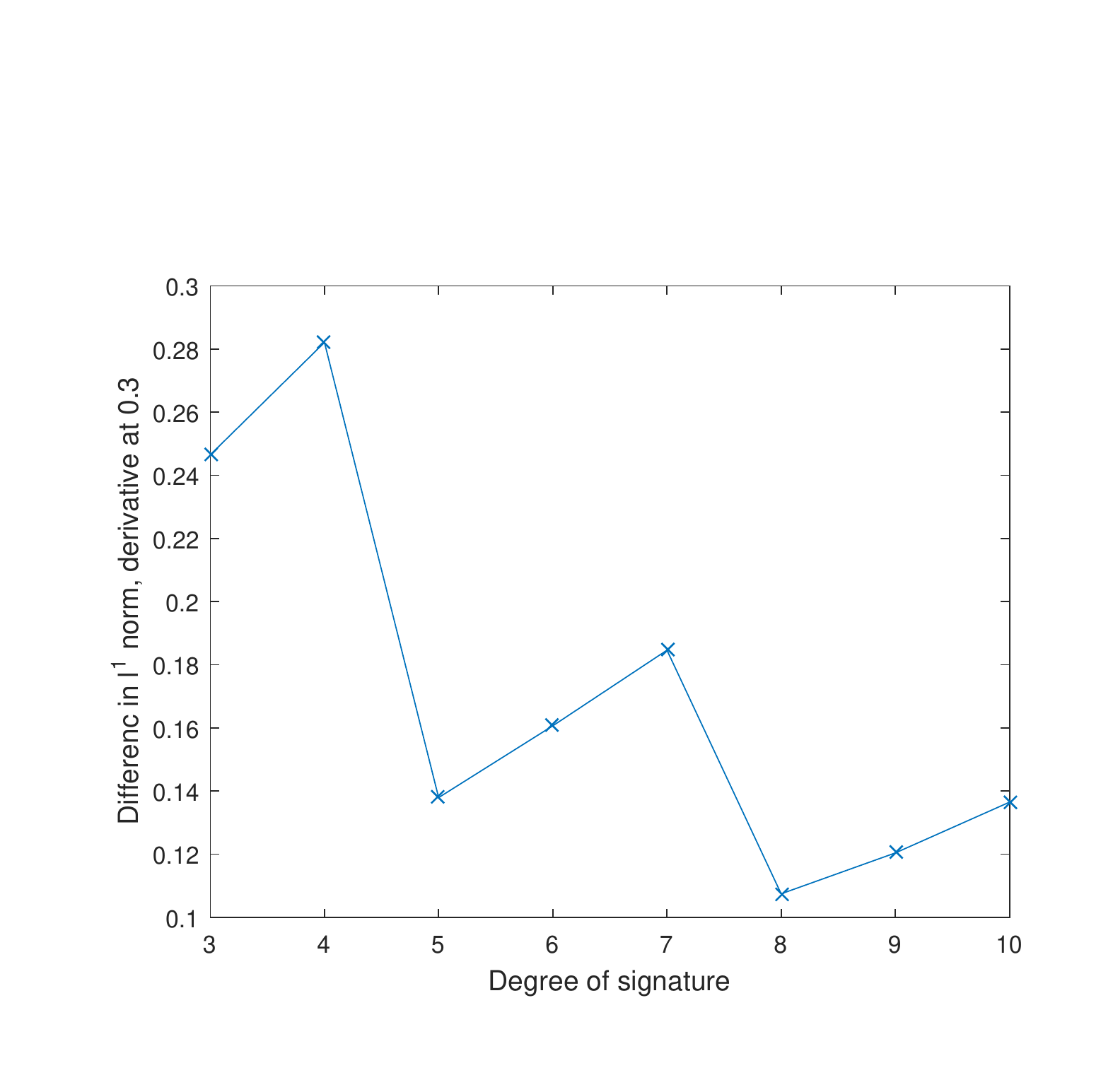}
\caption{$\left\lVert I_{p,n}(f(\theta))-\bar{S}_{n+1}\right\rVert_1$ for $p=\left\lfloor 0.3(n+2)\right\rfloor$, $n=3,\cdots,10$.}
\label{nextleveldiffl1}
\end{figure}
\begin{figure}[h!]
\centering
\includegraphics[trim={2cm 8cm 3cm 6cm},clip,width=0.8\textwidth]{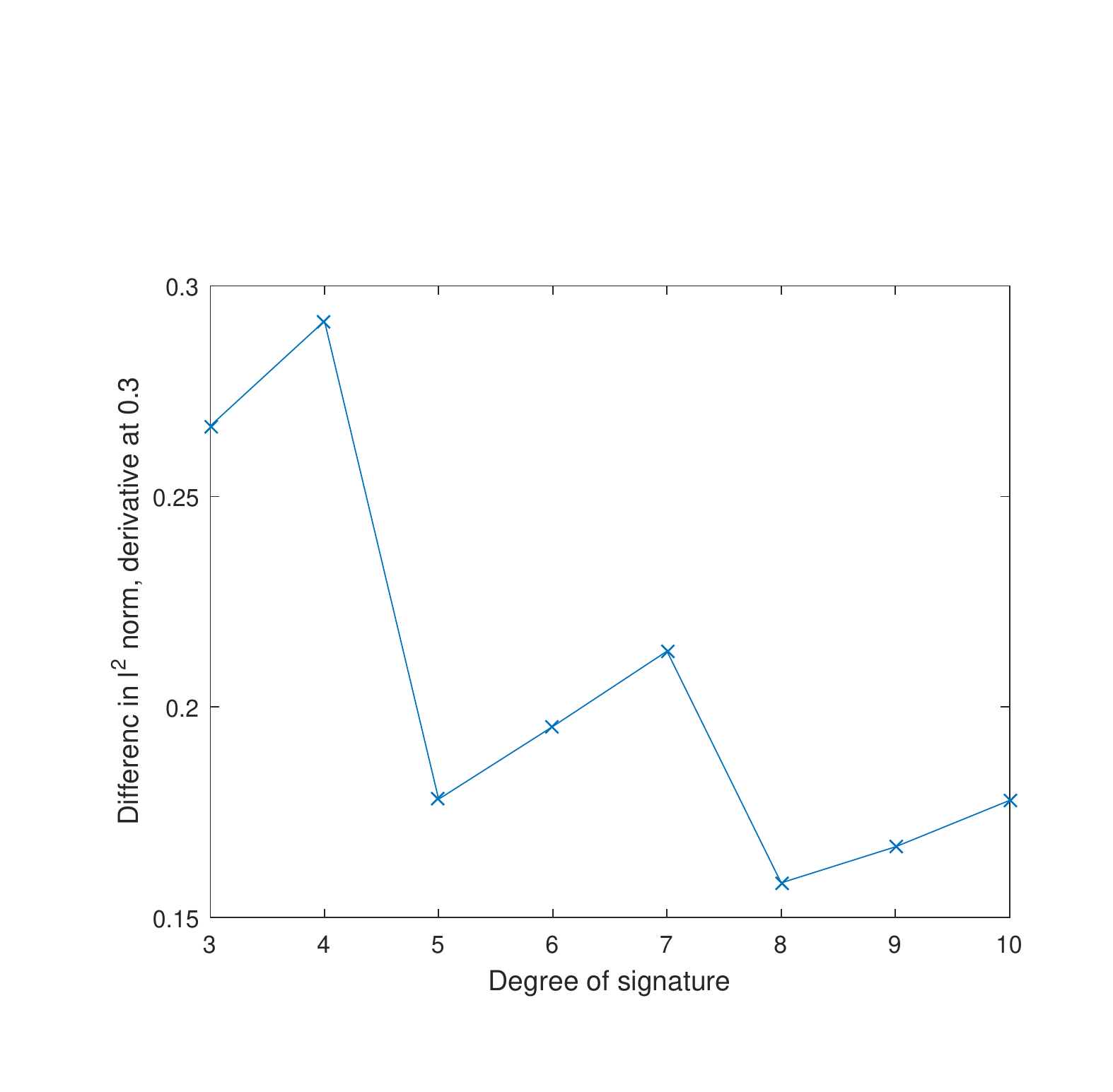}
\caption{$\left\rVert I_{p,n}(f(\theta))-\bar{S}_{n+1}\right\rVert_2$ for $p=\left\lfloor 0.3(n+2)\right\rfloor$, $n=3,\cdots,10$.}
\label{nextleveldiffl2}
\end{figure}
\begin{figure}[h!]
\centering
\includegraphics[trim={2cm 8cm 3cm 6cm},clip,width=0.8\textwidth]{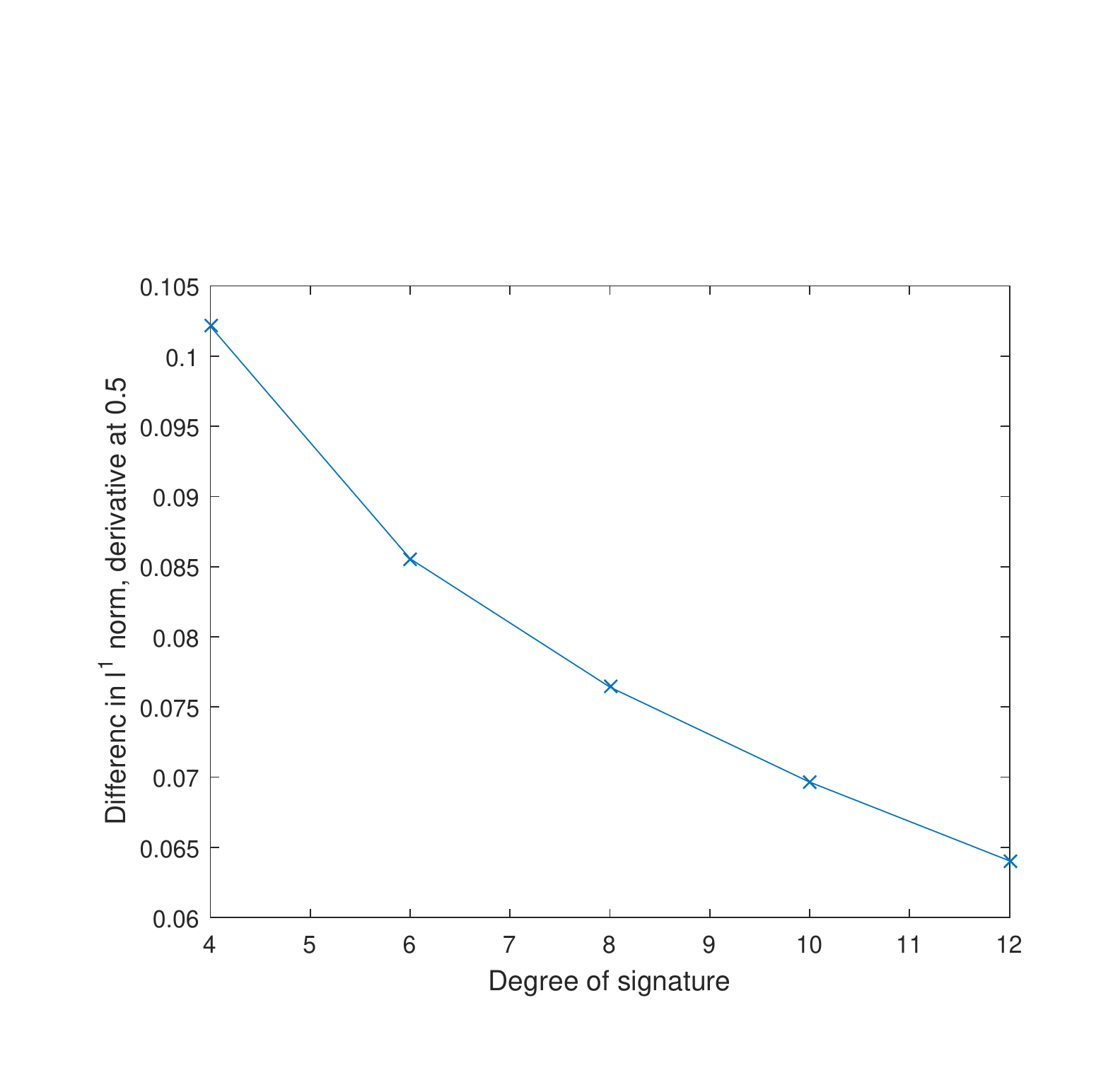}
\caption{$\left\rVert I_{p,n}(f(\theta))-\bar{S}_{n+1}\right\rVert_1$ for $p=\left\lfloor 0.5(n+2)\right\rfloor$, $n=4,6,8,10,12$.}
\label{nextleveldiffl10.5}
\end{figure}
\section{A lower bound for the signature of a path}\label{sectionlowerbound}
We have so far discussed finding an upper bound on $\left\lVert I_{p,n}(f(\theta))-\bar{S}_{n+1}\right\rVert$ for a path which is differentiable at $\theta$. In the light of Lemma \ref{normrelation}, we know that for $x,y\in\mathbb{R}^d$, 
\begin{align*}
\left\lVert\bar{S}_n\right\rVert\left\lVert x-y\right\rVert=\left\lVert I_{p,n}(x)-I_{p,n}(y)\right\rVert\leq\left\lVert I_{p,n}(x)-\bar{S}_{n+1}\right\rVert+\left\lVert I_{p,n}(y)-\bar{S}_{n+1}\right\rVert.
\end{align*}
Given an upper bound on the right-hand side of the inequality, if we can obtain a lower bound on $\left\lVert\bar{S}_n\right\rVert$, we can get an upper bound on $\lVert x-y\rVert$. In fact, finding a lower bound for the signature is itself an interesting topic. We will see in the following example that the rate of decay of the signature depends on the path as well as the norm we choose. 
\begin{eg}\label{strongdecaycounter}
If we consider a monotone lattice path $\gamma$ of consisting of two pieces, and each piece is of length $\frac{1}{2}$, the $\ell^1$ norm of the signature at level $n$ is
\begin{align*}
\left\lVert n!S^n_{0,T}(\gamma)\right\rVert_{1}=\sum_{k=0}^n\binom{n}{k}\left(\frac{1}{2}\right)^k\left(\frac{1}{2}\right)^{n-k}=1,
\end{align*}
hence the signature is clearly bounded below. If we consider the norm of the signature under the Hilbert-Schmidt norm, then
\begin{align*}
\left\lVert n!S^n_{0,T}(\gamma)\right\rVert_{\text{HS}}=\sqrt{\sum_{k=0}^n\binom{n}{k}^2\left(\frac{1}{2}\right)^{2k}\left(\frac{1}{2}\right)^{2(n-k)}}.
\end{align*}
As we can see from Figure \ref{lowerboundeg}, the $\left\lVert S^n_{0,T}(\gamma)\right\rVert_{\text{HS}}$ decreases in such a way that there is no obvious constant non-zero lower bound for the signature.
\end{eg}
\begin{figure}
\centering
\includegraphics[width=\textwidth]{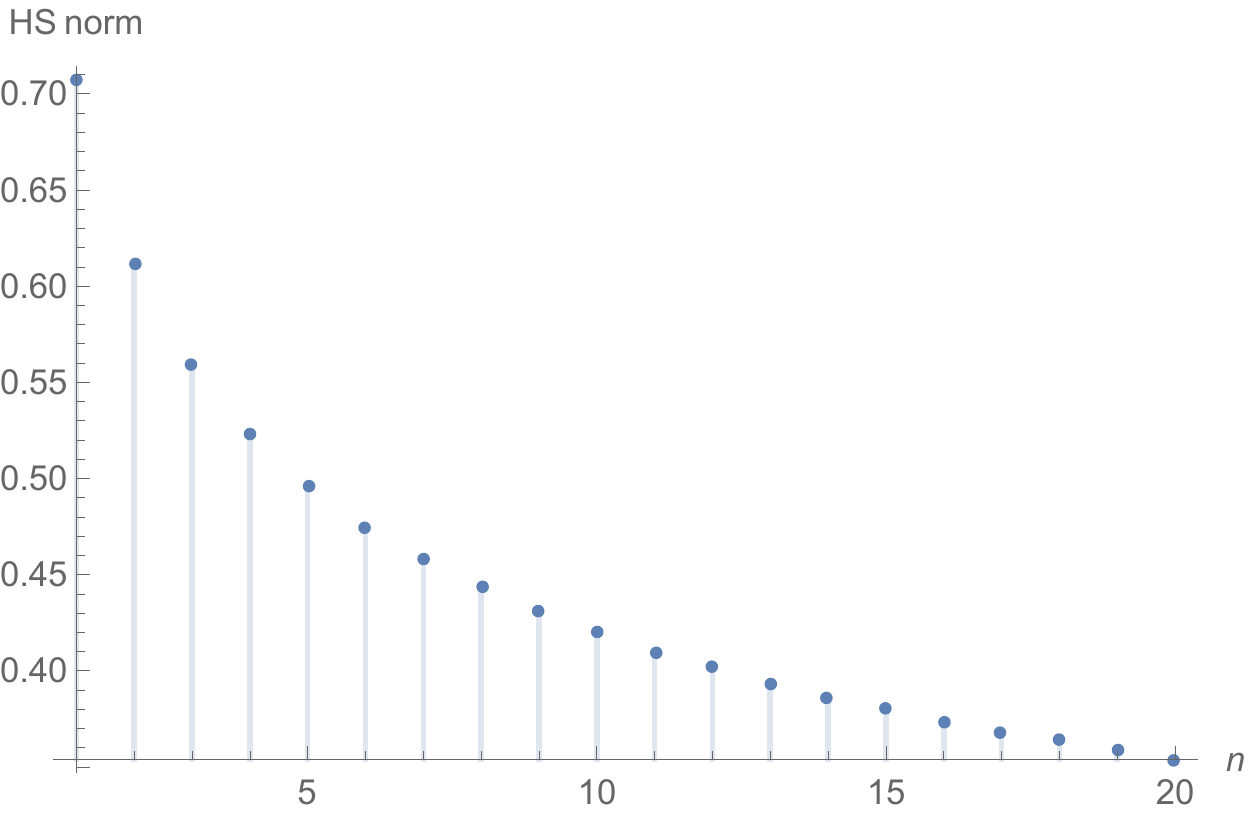}
\caption{The Hilbert-Schmidt norm of the normalised signature of a monotone lattice path at level $n$}
\label{lowerboundeg}
\end{figure}
Therefore it is important to take into account the effects of the norm when we look for a lower bound for the signature.\newline
We first recall the norm in defined by Hambly and Lyons \cite{hambly2010uniqueness}.
\begin{defn}
If $V$ is a Banach space, $A$ is a Banach algebra, and $F_1,\cdots,F_k\in \text{Hom}(V,A)$, then the canonical linear extension $F_1\otimes\cdots\otimes F_k$ from $V^{\otimes k}$ to $A$ is defined as
\begin{align*}
(v_1,\cdots,v_k)\to F_1(v_1)\cdots F_k(v_k).
\end{align*} 
Define the norm 
\begin{align*}
\lVert x\rVert_{\to A}:=\sup_{F_i\in \text{Hom}(V,A),\lVert F_i\rVert_{\text{Hom}(V,A)}=1}\left\lVert F_1\otimes\cdots\otimes F_k(x)\right\rVert_A.
\end{align*}
\end{defn}
As stated by Hambly and Lyons \cite{hambly2010uniqueness}, the norm $\lVert\cdot\rVert_{\to A}$ is smaller than the projective tensor norm. We give a proof of this claim in the following lemma.
\begin{lem}\label{arrowandprojnorm}
For all $x\in V^{\otimes k}$, $\lVert x\rVert_{\to A}\leq \lVert x\rVert_{\pi}$, where $\lVert\cdot\rVert_{\pi}$ is the projective tensor norm.
\end{lem}
\begin{proof}
For all $x\in V^{\otimes k}$, if $x=\sum_{i\in I}v_{i_1}\otimes\cdots\otimes v_{i_k}$ is an representation of $x$ for some indexing set $I$, then for any $F_i\in \text{Hom}(V,A)$ such that $\lVert F_i\rVert_{\text{Hom}(V,A)}=1$, $i=1,\cdots,k$, we have
\begin{align*}
\left\lVert(F_1\otimes\cdots\otimes F_k)(x)\right\rVert_A&=\left\lVert\sum_{i\in I}F_1(v_1)\cdots F_k(v_k)\right\rVert_A\\
&\leq\sum_{i\in I}\left\lVert F(v_{i_1})\rVert_{A}\cdots\right\lVert F(v_{i_k})\rVert_A\\
&\leq\sum_{i\in I}\left\lVert v_{i_1}\right\rVert\cdots\left\lVert v_{i_k}\right\rVert,
\end{align*}
for an abitrary representation of $x$. Then by the definition of the projective tensor norm, for any $F_i\in \text{Hom}(V,A)$ such that $\lVert F_i\rVert_{\text{Hom}(V,A)}=1$, $i=1,\cdots,k$,
\begin{align*}
\left\lVert(F_1\otimes\cdots\otimes F_k)(x)\right\rVert_A\leq\lVert x\rVert_{\pi},
\end{align*}
hence
\begin{align*}
\lVert x\rVert_{\to A}\leq\lVert x\rVert_{\pi}.
\end{align*}
\end{proof}
We also introduce a mapping function defined by Hambly and Lyons \cite{hambly2010uniqueness}.
\begin{defn}\label{hypermatrix}
Define the mapping $F:\mathbb{R}^{d}\to \text{Hom}\left(\mathbb{R}^{d+1},\mathbb{R}^{d+1}\right)$ such that for $x=(x_1,\cdots,x_d)\in\mathbb{R}^d$,
\begin{align*}
F:x\mapsto\begin{pmatrix}
0 & \cdots & 0 & x_1\\
\vdots &\ddots &\vdots &\vdots\\
0 &\cdots & 0 & x_d\\
x_1 &\cdots & x_d & 0
\end{pmatrix},
\end{align*}
where $\text{Hom}(V,A)$ denotes the set of bounded linear operators from $V$ to $A$.
\end{defn}
We also note the following useful lemma which gives a bound on the tail behaviour of the Poisson distribution by Canonne \cite{canonne2017}.
\begin{lem}[Canonne \cite{canonne2017}]\label{poissontail}
Let $X\sim$ Poisson$(\lambda)$ for some parameter $\lambda>0$. Then for any $h>0$, we have
\begin{align*}
\mathbb{P}\left(|X-\lambda|\geq h\right)\leq2\exp\left(-\frac{h^2}{2(\lambda+h)}\right).
\end{align*} 
\end{lem}
We extend the argument by Hambly and Lyons (Theorem 13, \cite{hambly2010uniqueness}) and prove in the following theorem that a non-zero lower bound exists for more than one level of the signature of a piecewise linear path.
\begin{thm}\label{lowerbound}
Let $\gamma:[0,1]\rightarrow\mathbb{R}^d$ be a non-degenerate piecewise linear path consisting of $M>0$ linear pieces. Suppose $2\Omega>0$ is the smallest angle between two adjacent edges. Equip $\mathbb{R}^d$ and $\mathbb{R}^{d+1}$ with the Euclidean norm. Then for any $c\in(0,1)$, there exists at least an increasing subsequence $(n_k)_{k\geq 1}\in\mathbb{N}$ such that
\begin{align*}
\left\lVert\bar{S}_{n_k}\right\rVert_{\to\text{Hom}(\mathbb{R}^{d+1},\mathbb{R}^{d+1})}\geq c\exp\left(-(M-1)K(\Omega)\right)\quad\forall\, k\geq 1,
\end{align*}
where $K(\Omega):=\log\left(\frac{2}{1-\cos|\Omega|}\right)$.
\end{thm}
\begin{proof}
Without loss of generality, we can assume $\gamma$ is of length $1$. Suppose $D>0$ is the length of the shortest edge of $\gamma$. For $\alpha>0$, write the path $\alpha\gamma$ as $\gamma_\alpha$. Then for all $\alpha$ such that $\alpha\geq\frac{K(\Omega)}{D}$, the shortest path of $\gamma_\alpha$ is at least of length $K(\Omega)$. Then by Lemma 3.7 of \cite{hambly2010uniqueness}, the Cartan development $\Gamma_\alpha$ of $\gamma_\alpha$ satisfies
\begin{align*}
d(o, \Gamma_{\alpha}o)\geq\alpha-(M-1)K(\Omega).
\end{align*}
Also by Proposition 3.13 of \cite{hambly2010uniqueness}, we know that 
\begin{align*}
\lVert\Gamma_\alpha\rVert_{\text{Hom}(\mathbb{R}^{d+1},\mathbb{R}^{d+1})}\geq\exp(d(o,\Gamma_{\alpha}o)).
\end{align*}
Then if we recall the definition of the map $F$ as in Definition \ref{hypermatrix}, for all $\alpha$ such that $\alpha>\frac{K(\Omega)}{D}$,  we have
\begin{align}\label{poissonineqderive}
&\exp\left(\alpha-(M-1)K(\Omega)\right)\\
\leq&\left\lVert\Gamma_\alpha\right\rVert_{\text{Hom}(\mathbb{R}^{d+1},\mathbb{R}^{d+1})}\nonumber\\
\leq&\sum_{n=0}^\infty\alpha^n\left\lVert\int_{0<u_1<\cdots<u_n<1}F(\mathrm{d}\gamma_{u_1})\cdots F(\mathrm{d}\gamma_{u_n})\right\rVert_{\text{Hom}(\mathbb{R}^{d+1},\mathbb{R}^{d+1})}\nonumber\\
=&\sum_{n=0}^\infty\alpha^n\left\lVert (F\otimes\cdots\otimes F)\left(\int_{0<u_1<\cdots<u_n<1}\mathrm{d}\gamma_{u_1}\otimes\cdots\otimes\mathrm{d}\gamma_{u_n}\right)\right\rVert_{\text{Hom}(\mathbb{R}^{d+1},\mathbb{R}^{d+1})}\nonumber\\
\leq&\sum_{n=0}^\infty\alpha^n\left\lVert\int_{0<u_1<\cdots<u_n<1}\mathrm{d}\gamma_{u_1}\otimes\cdots\otimes\mathrm{d}\gamma_{u_n}\right\rVert_{\to\text{Hom}(\mathbb{R}^{d+1},\mathbb{R}^{d+1})}\,\nonumber\\
\leq&\sum_{n=0}^\infty\frac{\alpha^n}{n!}\left\lVert\bar{S}_n\right\rVert_{\to\text{Hom}(\mathbb{R}^{d+1},\mathbb{R}^{d+1})},
\end{align}
where the third inequality follows from  the definition of the norm $\lVert\cdot\rVert_{\to \text{Hom}(\mathbb{R}^{d+1},\mathbb{R}^{d+1})}$. Multiplying both sides of (\ref{poissonineqderive}) by $\exp(-\alpha)$ gives
\begin{equation}\label{poisson}
\exp\left(-(M-1)K(\Omega)\right)\leq\exp(-\alpha)\sum_{n=0}^{\infty}\frac{\alpha^n}{n!}\left\lVert\bar{S}_n\right\rVert_{\to\text{Hom}(\mathbb{R}^{d+1},\mathbb{R}^{d+1})}.
\end{equation}
Note the right hand side of (\ref{poisson}) is the expectation of the function $\left\lVert\bar{S}_n\right\rVert_{\to\text{Hom}(\mathbb{R}^{d+1},\mathbb{R}^{d+1})}$ under the Poisson distribution with parameter $\alpha$. Note the distribution has mean $\alpha$, variance $\alpha$. We have the following claim: \newline
For all $c\in(0,1)$, 
\begin{align*}
&\mathbb{P}\left(n\,\text{such that}\,\left\lVert\bar{S}_n
\right\rVert_{\to\text{Hom}(\mathbb{R}^{d+1},\mathbb{R}^{d+1})}\geq c\exp\left(-(M-1)K(\Omega)\right)\right)\\
\geq&(1-c)\exp\left(-(M-1)K(\Omega)\right).
\end{align*}
We prove the above claim by contradiction. Suppose that
\begin{align*}
&\mathbb{P}\left(n\,\text{such that} \left\lVert\bar{S}_n\right\rVert_{\to\text{Hom}(\mathbb{R}^{d+1},\mathbb{R}^{d+1})}\geq c\exp\left(-(M-1)K(\Omega)\right)\right)\\
<&(1-c)\exp\left(-(M-1)K(\Omega)\right).
\end{align*}
We know from \cite{lyons2007differential} that $\left\lVert\bar{S}_n\right\rVert_{\to \text{Hom}(\mathbb{R}^{d+1},\mathbb{R}^{d+1})}\leq 1$. Then if we think about how large the expectation can be, we have
\begin{align*}
&\exp(-\alpha)\sum_{n=0}^{\infty}\frac{\alpha^n}{n!}\left\lVert\bar{S}_n\right\rVert_{\to\text{Hom}(\mathbb{R}^{d+1},\mathbb{R}^{d+1})}\\
<&(1-c)\exp\left(-(M-1)K(\Omega)\right)+c\exp(-(M-1)K(\Omega))\\
=&\exp(-(M-1)K(\Omega))
\end{align*}
which contradicts (\ref{poisson}). So we must have
\begin{align*}
&\mathbb{P}\left(n\,\text{such that}\,\left\lVert\bar{S}_n\right\rVert_{\to\text{Hom}(\mathbb{R}^{d+1},\mathbb{R}^{d+1})}\geq c\exp\left(-(M-1)K(\Omega)\right)\right)\\
\geq&(1-c)\exp\left(-(M-1)K(\Omega)\right).
\end{align*}
Then by Lemma \ref{poissontail}, we have an estimate for $X\sim$Poisson$(\alpha)$ such that for all $h>0$, 
\begin{align*}
\mathbb{P}\left(|X-\alpha|\geq h\right)\leq2\exp\left(-\frac{h^2}{2(\alpha+h)}\right).
\end{align*}
In particular, 
\begin{equation}\label{poissontailspecial}
\mathbb{P}\left(|X-\alpha|\geq\alpha^{3/4}\right)\leq 2\exp\left(-\frac{\alpha^{3/2}}{2(\alpha+\alpha^{3/4})}\right).
\end{equation}
Note the right-hand side of (\ref{poissontailspecial}) is a decreasing function in  $\alpha$, hence there exists $\alpha^*$ such that for all $\alpha>\alpha^*$, 
\begin{align*}
\mathbb{P}(|X-\alpha|\geq\alpha^{3/4})&\leq 2\exp\left(-\frac{\alpha^{3/2}}{2(\alpha+\alpha^{3/4})}\right)\\
&<(1-c)\exp\left(-(M-1)K(\Omega)\right)\\
&\leq\mathbb{P}\left(n\, \text{such that}\, \left\lVert\bar{S}_n\right\rVert_{\to\text{Hom}(\mathbb{R}^{d+1},\mathbb{R}^{d+1})}\geq c\exp(-(M-1)K(\Omega))\right).
\end{align*}
Then there must be some $n$ near the mean $\alpha$ such that
\begin{align*}
\left\lVert\bar{S}_n\right\rVert_{\to\text{Hom}(\mathbb{R}^{d+1},\mathbb{R}^{d+1})}\geq c\exp(-(M-1)K(\Omega)),
\end{align*}
i.e. for $\alpha>\alpha^*$,
\begin{align*}
&\mathbb{P}\left(X=n\,\text{where}\,|n-\alpha|<\alpha^{3/4}\,\text{and}\,\lVert\bar{S}_n\rVert_{\to\text{Hom}(\mathbb{R}^{d+1},\mathbb{R}^{d+1})}\geq c\exp(-(M-1)K(\Omega))\right)\\
\geq&(1-c)\exp\left(-(M-1)K(\Omega)\right)-2\exp\left(-\frac{\alpha^{3/2}}{2(\alpha+\alpha^{3/4})}\right)\\
>&0.
\end{align*}
Hence for large enough $\alpha$, there exists at least one $n\in(\alpha-\alpha^{3/4},\alpha+\alpha^{3/4})$ such that $\left\lVert\bar{S}_n\right\rVert_{\to\text{Hom}(\mathbb{R}^{d+1},\mathbb{R}^{d+1})}\geq c\exp(-(M-1)K(\Omega))$. Note that $\alpha$ grows faster than $\alpha^{3/4}$, so as $\alpha$ increases, the interval $(\alpha-\alpha^{3/4},\alpha+\alpha^{3/4})$ moves rightwards. Hence there exists a strictly increasing subsequence $(n_k)_{k\geq 1}\in\mathbb{N}$ such that 
\begin{align*}
\left\lVert\bar{S}_{n_k}\right\rVert_{\to\text{Hom}(\mathbb{R}^{d+1},\mathbb{R}^{d+1})}\geq c\exp(-(M-1)K(\Omega))\quad\forall\, k\geq 1.
\end{align*}
\end{proof}
\begin{cor}\label{pinormlowerbound}
Let $\gamma:[0,1]\rightarrow\mathbb{R}^d$ be a non-degenerate piecewise linear path consisting of $M>0$ linear pieces. Suppose $2\Omega>0$ is the smallest angle between two adjacent edges. Equip $\mathbb{R}^d$ and $\mathbb{R}^{d+1}$ with the Euclidean norm. Then for any $c\in(0,1)$, there exists at least a subsequence $(n_k)_{k\geq 1}\in\mathbb{N}$ such that
\begin{align*}
\left\lVert\bar{S}_{n_k}\right\rVert_{\pi}\geq c\exp(-(M-1)K(\Omega))\quad\forall k\geq 1,
\end{align*}
where $K(\Omega):=\log\left(\frac{2}{1-\cos|\Omega|}\right)$ and $\lVert\cdot\rVert_{\pi}$ is the projective tensor norm induced from the Euclidean space $\mathbb{R}^d$.
\end{cor}
\begin{proof}
Without loss of generality, we assume $\gamma$ is of length $1$. There are two ways to justify this result.\newline
By Lemma \ref{arrowandprojnorm}, we know that $\lVert\cdot\rVert_{\pi}$ is a bigger norm than $\lVert\cdot\rVert_{\to \text{Hom}(\mathbb{R}^{d+1},\mathbb{R}^{d+1})}$, hence by Theorem \ref{lowerbound}, for any $c\in (0,1)$, there exists a subsequence $(n_k)_{k\geq 1}\in\mathbb{N}$ such that 
\begin{align*}
c\exp(-(M-1)K(\Omega))\leq\left\lVert\bar{S}_{n_k}\right\rVert_{\to \text{Hom}(\mathbb{R}^{d+1},\mathbb{R}^{d+1})}\leq \left\lVert\bar{S}_{n_k}\right\rVert_{\pi}.
\end{align*}
\newline
An alternative proof directly applies the argument in the proof of Theorem \ref{lowerbound} to the norm $\lVert\cdot\rVert_{\pi}$. Note that for any $x\in\mathbb{R}^{d^{\otimes k}}$, and $F:\mathbb{R}^d\to \text{Hom}(\mathbb{R}^{d+1},\mathbb{R}^{d+1})$ as defined in Definition \ref{hypermatrix}, if $x=\sum_{i\in I}v_{i_1}\otimes\cdots\otimes v_{i_k}$ for an indexing set $I$, we have
\begin{align*}
&\left\lVert(F\otimes\cdots\otimes F)(x)\right\rVert_{\text{Hom}(\mathbb{R}^{d+1},\mathbb{R}^{d+1})}\\
=&\left\lVert\sum_{i\in I}F(v_{i_1})\cdots F(v_{i_k})\right\rVert_{\text{Hom}(\mathbb{R}^{d+1},\mathbb{R}^{d+1})}\\
\leq&\sum_{i\in I}\left\lVert F(v_{i_1})\right\rVert_{\text{Hom}(\mathbb{R}^{d+1},\mathbb{R}^{d+1})}\cdots\lVert F(v_{i_k})\rVert_{\text{Hom}(\mathbb{R}^{d+1},\mathbb{R}^{d+1})}\\
\leq&\left\lVert F\right\rVert^k_{\text{Hom}(\mathbb{R}^d,\text{Hom}(\mathbb{R}^{d+1},\mathbb{R}^{d+1}))}\sum_{i\in I}\left\lVert v_{i_1}\right\rVert\cdots\left\lVert v_{i_k}\right\rVert,
\end{align*}
for an abritrary representation of $x$. Hence
\begin{align*}
\left\lVert (F\otimes\cdots\otimes F)(x)\right\rVert_{\text{Hom}(\mathbb{R}^{d+1},\mathbb{R}^{d+1})}\leq\left\lVert F\right\rVert^k_{\text{Hom}(\mathbb{R}^d,\text{Hom}(\mathbb{R}^{d+1},\mathbb{R}^{d+1}))}\lVert x\rVert_{\pi}\leq \lVert x\rVert_{\pi}. 
\end{align*}
Therefore when we equip $(\mathbb{R}^d)^{\otimes k}$ with $\lVert\cdot\rVert_{\pi}$, we have
\begin{align*}
\lVert (F\otimes\cdots\otimes F)\rVert_{\text{Hom}\left(\left(\mathbb{R}^d\right)^{\otimes k}, \text{Hom}(\mathbb{R}^{d+1},\mathbb{R}^{d+1})\right)}\leq 1.
\end{align*}
As in the proof of Theorem \ref{lowerbound}, and if we use $\lVert\cdot\rVert_{\pi}$ instead of $\lVert\cdot\rVert_{\to \text{Hom}(\mathbb{R}^{d+1},\mathbb{R}^{d+1})}$, (\ref{poissonineqderive}) becomes
\begin{align*}
&\exp(\alpha-(M-1)K(\Omega))\\
\leq&\left\lVert\Gamma_{\alpha}\right\rVert_{\text{Hom}(\mathbb{R}^{d+1},\mathbb{R}^{d+1})}\\
\leq&\sum_{n=0}^\infty\alpha^n\left\lVert\int_{0<u_1<\cdots<u_k<1} F(\mathrm{d}\gamma_{u_1})\cdots F(\mathrm{d}\gamma_{u_n})\right\rVert_{\text{Hom}(\mathbb{R}^{d+1},\mathbb{R}^{d+1})}\\
=&\sum_{n=0}^{\infty}\alpha^n\left\lVert(F\otimes\cdots\otimes F)\int_{0<u_1<\cdots<u_n<1}\mathrm{d}\gamma_{u_1}\otimes\cdots\otimes\mathrm{d}\gamma_{u_n}\right\rVert_{\text{Hom}(\mathbb{R}^{d+1},\mathbb{R}^{d+1})}\\
\leq&\sum_{n=0}^\infty\alpha^n\left\lVert\int_{0<u_1<\cdots<u_n<1}\mathrm{d}\gamma_{u_1}\otimes\cdots\otimes\mathrm{d}\gamma_{u_n}\right\rVert_{\pi}\\
\leq&\sum_{n=0}^\infty\frac{\alpha^n}{n!}\left\lVert\bar{S}\right\rVert_{\pi}
\end{align*}
and then the rest of the proof of Theorem \ref{lowerbound} applies.
\end{proof}
\begin{rmk}
We have seen from the conjecture in \cite{chang2018super} that we expect the $n$-th root of the $n$-th term in the signature of a path of finite length multiplied by $n!$ to converge to the length of the path under a reasonable tensor algebra norm. Hambly and Lyons \cite{hambly2010uniqueness} showed that (Theorem \ref{hamblylyonslimit}) a stronger decay result holds in special cases: If $\gamma$ is a path of finite length $L>0$ with the modulus of continuity of its derivative $\delta(\epsilon)=o(\epsilon^{3/4})$, then
\begin{align*}
L^{-k}k!\left\lVert\int_{0<u_1<\cdots<u_k<1}\mathrm{d}\gamma_{u_1}\otimes\cdots\otimes\mathrm{d}\gamma_{u_k}\right\rVert_{\to \text{Hom}(\mathbb{R}^{d+1},\mathbb{R}^{d+1})}\to1
\end{align*} 
as $n\to\infty$. However we have seen from Example \ref{strongdecaycounter} such a strong result does not hold for piecewise linear paths at least under the Hilbert-Schmidt norm. The significance of Theorem \ref{lowerbound} and Corollary \ref{pinormlowerbound} is that we have a stronger result for a piecewise linear path than stated in the conjecture.
\end{rmk}
\section{Inverting the signature of a path}
Assume $\gamma:[0,1]\rightarrow\mathbb{R}^d$ is a continuous bounded-variation path with derivative $f:(0,1)\rightarrow\mathbb{R}^d$ such that $\lVert f(t)\rVert_2=1$ for all $t\in(0,1)$ almost everywhere. Assume $\gamma$ is linear on $[s,t]\subset[0,1]$, and $\theta\in(s,t)$. For $n\geq1$, choose $p\in\{1,\cdots,n+1\}$ such that $p=\left\lfloor\theta(n+2)\right\rfloor$. In this section we use the result of Theorem \ref{thmconvergetozeroupperbound}, i.e. there exists $\epsilon^\gamma_{\theta,n}$ such that $\lVert I_{p,n}(f(\theta))-\bar{S}_{n+1}\rVert_{\pi}\leq\epsilon^\gamma
_{\theta,n}$ and $\epsilon^\gamma_{\theta,n}\rightarrow0$ as $n\rightarrow\infty$.\newline
Define the set
\begin{align*}
A^\gamma_{\theta,n}:=\left\{x\in\mathbb{R}^d:\lVert x\rVert_2=1, \left\lVert I_{p,n}(x)-\bar{S}_{n+1}\right\rVert_{\pi}\leq\epsilon^\gamma_{\theta, n}\right\}.
\end{align*}
Note $f(\theta)\in A^\gamma_{\theta,n}$. We adopt these notations in this section. We first note the following simple lemma.
\begin{lem}
The projective tensor norm $\lVert\cdot\rVert_{\pi}$ satisfies Definition \ref{norm}.
\end{lem}
We now give a strategy to invert the signature of a non-degenerate piecewise linear path. 
\begin{thm}\label{setconverge}
Assume $\gamma:[0,1]\rightarrow\mathbb{R}^d$ is a non-degenerate piecewise linear path with derivative $f:(0,1)\rightarrow\mathbb{R}^d$ such that $\lVert f(t)\rVert_2=1$ for all $t\in(0,1)$ if defined. Assume $\gamma$ is differentiable at $\theta\in(0,1)$. For $n\geq1$, choose $p\in\{1,\cdots,n+1\}$ such that $p=\left\lfloor\theta(n+2)\right\rfloor$. Then there exists a subsequence $(n_k)_{k\geq 1}\in\mathbb{N}$ such that for all $k\geq 1$, for all $x_{\theta,n_k}, y_{\theta,n_k}\in A^\gamma_{\theta, n_k}$, 
\begin{align*}
\left\lVert x_{\theta,n_k}-y_{\theta,n_k}\right\rVert_{2}\rightarrow0\quad\text{as}\quad k\rightarrow\infty. 
\end{align*}
\end{thm}
\begin{proof}
Note for all $n\geq 1$, $x_{\theta,n},y_{\theta,n}\in A^\gamma_{\theta,n}$, 
\begin{align*}
&\left\lVert I_{p,n}(x_{\theta,n})-I_{p,n}(y_{\theta,n})\right\rVert_{\pi}\\
\leq&\left\lVert I_{p,n}(x_{\theta,n})-\bar{S}_{n+1}\right\rVert_{\pi}+\left\lVert I_{p,n}(y_{\theta,n})-\bar{S}_{n+1}\right\rVert_{\pi}\\
\leq&2\epsilon^\gamma_{\theta,n}.
\end{align*}
By Corollary \ref{pinormlowerbound}, there exists a subsequence $(n_k)_{k\geq 1}\in\mathbb{N}$ such that for all $k\geq 1$, $\lVert\bar{S}_{n_k}\rVert_{\pi}\geq \frac{1}{2}\exp(-(M-1)K(\Omega))$, where $K(\Omega)=\log(\frac{2}{1-\text{cos}|\Omega|})$, $2\Omega$ is the smallest angle between two adjacent edges, and $M>0$ is the number of linear pieces of $\gamma$. Then by Lemma \ref{normrelation}, 
\begin{align*}
&\left\lVert x_{\theta,n_k}-y_{\theta,n_k}\right\rVert_{2}\\
=&\frac{\left\lVert I_{p,n_k}(x_{\theta,n_k})-I_{p,n_k}(y_{\theta,n_k})\right\rVert_{\pi}}{\left\lVert\bar{S}_{n_k}\right\rVert_{\pi}}\\
\leq&\frac{4\epsilon^\gamma_{\theta,n_k}}{\exp(-(M-1)K(\Omega))}.
\end{align*}
Since $\epsilon_{\theta,n_k}\rightarrow0$ as $k\rightarrow\infty$, we have $\lVert x_{\theta,n_k}-y_{\theta,n_k}\rVert_{2}\rightarrow0$ as $k\rightarrow\infty$.
\end{proof}
We are then able to derive a corollary which is more useful for computation.
\begin{cor}\label{derapprox}
Assume $\gamma:[0,1]\rightarrow\mathbb{R}^d$ is a non-degenerate piecewise linear path with derivative $f:(0,1)\rightarrow\mathbb{R}^d$ such that $\lVert f(t)\rVert_2=1$ for all $t\in(0,1)$ if defined. Assume $\gamma$ is differentiable at $\theta\in(0,1)$. For $n\geq1$, choose $p\in\{1,\cdots,n+1\}$ such that $p=\left\lfloor\theta(n+2)\right\rfloor$. Define
\begin{equation}\label{optimisationsln}
x^*_{\theta,n}:=argmin_{x\in\mathbb{R}^d,\left\lVert x\right\rVert_2=1}\left\lVert I_{p,n}(x)-\bar{S}_{n+1}\right\rVert_{\pi}.
\end{equation}
Then there exists at least a subsequence $(n_k)_{k\geq 1}\in\mathbb{N}$ such that $x^*_{n_k}$ converges to $f(\theta)$ as $k$ increases. 
\end{cor}
\begin{proof}
By Theorem \ref{setconverge}, we know that there exists a subsequence $(n_k)_{k\geq 1}\in\mathbb{N}$ such that for all $k\geq 1$, for all $x_{\theta,n_k}, y_{\theta,n_k}\in A^\gamma_{\theta, n_k}$, $\lVert x_{\theta,n_k}-y_{\theta,n_k}\rVert_{2}\rightarrow0$ as $k\rightarrow\infty$. We know that $f(\theta)\in A^\gamma_{\theta, n_k}$, and $\left\lVert I_{p,n_k}(x^*_{\theta,n_k})-\bar{S}_{n_k+1}\right\rVert_{\pi}\leq\epsilon^\gamma_{\theta,n_k}$ due to the fact that $x^*_{\theta,n_k}$ gives the shortest distance between $I_{p,n_k}(x)$ and $\bar{S}_{n_k+1}$ among all $x\in\mathbb{R}^d$ such that $\lVert x\rVert_2=1$. Therefore $x^*_{\theta,n_k}\in A^\gamma_{\theta,n_k}$, hence $\left\lVert x^*_{\theta,n_k}-f(\theta)\right\rVert_{2}\rightarrow0$ as $k\rightarrow\infty$.
\end{proof}
We can also develop such an algorithm for another set of paths. First we recall the following theorem by Hambly and Lyons \cite{hambly2010uniqueness}.
\begin{thm}[Hambly and Lyons, Theorem 9 \cite{hambly2010uniqueness}]\label{hamblylyonslimit}
Let $J$ be a closed and bounded interval. Let $\gamma:J\to\mathbb{R}^d$ be a continuous path of finite length $\ell>0$. Recall that the modulus of continuity of the derivative is defined as $\delta(h):=\sup_{|u-v|\leq h}\left\lVert\dot{\gamma}(u)-\dot{\gamma}(v)\right\rVert_2$. If $\delta(h)=o\left(h^{3/4}\right)$, then
\begin{align*}
\ell^kk!\left\lVert\int_{0<u_1<\cdots<u_k<1}\mathrm{d}\gamma_{u_1}\otimes\cdots\otimes\mathrm{d}\gamma_{u_k}\right\rVert_{\to\text{Hom}\left(\mathbb{R}^{d+1},\mathbb{R}^{d+1}\right)}\to 1
\end{align*}
as $k\to\infty$.
\end{thm}
\begin{thm}\label{differentiablepathreconstruction}
Let $\gamma:[0,1]\to\mathbb{R}^d$ be a continuous path with derivative $f:(0,1)\to\mathbb{R}^d$ such that $\left\lVert f(t)\right\rVert_2=1$ for all $t\in(0,1)$. Suppose further that the modulus of continuity of $f$ is $\delta(h)=o\left(h^{3/4}\right)$. Assume $\gamma$ is linear over the interval $[s,t]\subset[0,1]$ and $\theta\in(s,t)$. Then for $n\geq 1$, choose $p\in\{1,...,n+1\}$ such that $p=\left\lfloor\theta(n+2)\right\rfloor$. Define
\begin{align*}
x^*_{\theta,n}:=argmin_{x\in\mathbb{R}^d,\lVert x\rVert_2=1}\left\lVert I_{p,n}(x)-\bar{S}_{n+1}\right\rVert_{\pi}.
\end{align*}
Then $x^*_{\theta,n}$ converges to $f(\theta)$ as $n$ increases.
\end{thm}
\begin{proof}
By Theorem \ref{hamblylyonslimit}, for any $c\in(0,1)$, there exists $N\in\mathbb{N}$ such that for all $n\geq N$, 
\begin{align*}
\left\lVert\bar{S}_n\right\rVert_{\to\text{Hom}\left(\mathbb{R}^{d+1},\mathbb{R}^{d+1}\right)}\geq 1-c.
\end{align*}
By Lemma \ref{arrowandprojnorm}, the projective norm is bigger than the norm $\lVert\cdot\rVert_{\to\text{Hom}\left(\mathbb{R}^{d+1},\mathbb{R}^{d+1}\right)}$, hence for all $n\geq N$, 
\begin{align*}
\left\lVert\bar{S}_n\right\rVert_{\pi}\geq 1-c.
\end{align*}
Then for all $n\geq N$, for all $x_{\theta,n}, y_{\theta,n}\in A^\gamma_{\theta,n}$, we have
\begin{align*}
&\left\lVert x_{\theta,n}-y_{\theta,n}\right\rVert_2\\
=&\frac{\left\lVert I_{\theta,n}(x_{\theta,n})-I_{\theta,n}(y_{\theta,n})\right\rVert_{\pi}}{\left\lVert\bar{S}_n\right\rVert_{\pi}}\\
\leq&\frac{2\epsilon^\gamma_{\theta,n}}{1-c}.
\end{align*}
Since $\epsilon^\gamma_{\theta,n}\to 0$ as $n\to\infty$, we have $\left\lVert x_{\theta,n}-y_{\theta,n}\right\rVert_2\to 0$ as $n\to\infty$. Since $x^*_{\theta,n}, f(\theta)\in A^\gamma_{\theta,n}$, we have 
\begin{align*}
\left\lVert x^*_{\theta,n}-f(\theta)\right\rVert_2\to 0\quad\text{as}\quad n\to\infty.
\end{align*}
\end{proof}
\begin{rmk}
Note that if we take $p=\left\lfloor\theta(n+2)\right\rfloor$, we may get $p=0$ if $n$ is small. But we can always take higher orders of the signature, and this will not affect our result.
\end{rmk}
Note that so far in this section we have assumed that the underlying path is parametrised at unit speed. However, in practice when we only have the information from the signature, it is impossible to know whether the path is parametrised at unit speed. We prove in the following lemma that our algorithm still works with a slight alteration.
\begin{lem}\label{differentscale}
For a non-degenerate piecewise linear path $\gamma:[a,b]\to\mathbb{R}^d$ of length $L>0$ and differentiable at $\theta\in(a,b)$, we can slightly change (\ref{optimisationsln}) and obtain an approximation to the derivative of $\gamma$ when it is parametrised at unit speed when we choose the position of insertion $p$ appropriately, even if the original speed of parametrisation is unknown. Moreover the same changes apply to the result of Theorem \ref{differentiablepathreconstruction}.
\end{lem}
\begin{proof}
Let the function $\phi:[a,b]\to[u,v]$ be such that the path $\tilde{\gamma}:=\gamma\circ\phi$ is parametrised at unit speed.\newline
We first try to determine what value $p$ should be, i.e. the position at which the element shall be inserted into the $n$-th level of the normalised signature of $\tilde{\gamma}$. For any norm which satisfies properties stated in Definition \ref{norm}, if we insert $x\in\mathbb{R}^d$ into the $n$-th level of the signature of $\tilde{\gamma}$, then 
\begin{align*}
&L^{-(n+1)}\bigg\lVert\int_{u<t_1<\cdots<t_{n+1}<v}(n+1)!\dot{\tilde{\gamma}}_{t_1}\otimes\cdots\otimes\dot{\tilde{\gamma}}_{t_{p-1}}\otimes x\otimes\dot{\tilde{\gamma}}_{t_{p+1}}\otimes\cdots\otimes\dot{\tilde{\gamma}}_{t_{n+1}}\mathrm{d}t_1\cdots\mathrm{d}t_{n+1}\\
&-(n+1)!\int_{u<t_1<\cdots<t_{n+1}<v}\dot{\tilde{\gamma}}_{t_1}\otimes\cdots\otimes\dot{\tilde{\gamma}}_{t_{n+1}}\mathrm{d}t_1\cdots\mathrm{d}t_{n+1}\bigg\rVert\\
=&L^{-(n+1)}(n+1)!\bigg\lVert\int_{u<t_1<\cdots<t_{n+1}<v}\dot{\tilde{\gamma}}_{t_1}\otimes\cdots\otimes\dot{\tilde{\gamma}}_{t_{p-1}}\\
&\otimes(x-\dot{\tilde{\gamma}}_{t_p})\otimes\dot{\tilde{\gamma}}_{t_{p+1}}\otimes\cdots\otimes\dot{\tilde{\gamma}}_{t_{n+1}}\mathrm{d}t_1\cdots\mathrm{d}t_{n+1}\bigg\rVert\\
\leq&\int_{u<t<v}\lVert x-\dot{\tilde{\gamma}}_t\rVert L^{-(n+1)}(n+1)!\frac{(t-u)^{p-1}}{(p-1)!}\frac{(v-t)^{n+1-p}}{(n+1-p)!}\mathrm{d}t,
\end{align*} 
which gives rise to the expectation of a function about a \emph{non-standard beta variable} $U\sim$ Beta$(p,n-p+2)$ over the interval $(u,v)$. We can change the variable in the integral to obtain a standard beta variable $U^\prime:=(U-u)/(v-u)$, which is  over the interval $(0,1)$. Therefore we can see that the expectation of $U$ is $p(v-u)/(n+2)+u$. If we want to minimise the the upper bound obtained on the difference, we shall choose
\begin{align*}
p=\left\lfloor\frac{\phi(\theta)-u}{v-u}(n+2)\right\rfloor
\end{align*}
in order to approximate the derivative of $\tilde{\gamma}$ at $\phi(\theta)$.\newline 
With a slight extension of the analysis so far, we see that for $n\geq 1$, the solution to
\begin{equation}\label{optimisationwithlength}
\min_{\left\lVert x\right\rVert_2=1}\left\lVert L^{-n}I^{\tilde{\gamma}}_{p,n}(x)-L^{-(n+1)}(n+1)!S^{n+1}_{u,v}(\tilde{\gamma})\right\rVert_{\pi}
\end{equation}
gives an approximation to the derivative of $\tilde{\gamma}$ at $\phi(\theta)$, where 
\begin{align*}
I^{\tilde{\gamma}}_{p,n}(x):=n!\int_{u<t_1<\cdots<t_n<v}\mathrm{d}\tilde{\gamma}_{t_1}\otimes\cdots\otimes\mathrm{d}\tilde{\gamma}_{t_{p-1}}\otimes x\otimes\mathrm{d}\tilde{\gamma}_{t_p}\otimes\cdots\otimes\mathrm{d}\tilde{\gamma}_{t_n}.
\end{align*}
Note 
\begin{align*}
S^n_{u,v}(\tilde{\gamma})&=\int_{u<t_1<\cdots<t_n<v}\mathrm{d}\tilde{\gamma}_{t_1}\otimes\cdots\otimes\mathrm{d}\tilde{\gamma}_{t_n}\\
&=\int_{a<t_1<\cdots<t_n<b}(\phi^\prime(t_1)\mathrm{d}\gamma_{\phi(t_1)})\otimes\cdots\otimes(\phi^\prime(t_n)\mathrm{d}\gamma_{\phi(t_n)})\\
&=\int_{a<t_1<\cdots<t_n<b}\mathrm{d}\gamma_{t_1}\otimes\cdots\otimes\mathrm{d}\gamma_{t_n}\\
&=S^n_{a,b}(\gamma).
\end{align*}
Hence (\ref{optimisationwithlength}) can be written as a problem about $\gamma$:
\begin{align*}
\min_{\left\lVert x\right\rVert_2=1}\left\lVert L^{-n}I^{\gamma}_{p,n}(x)-L^{-(n+1)}(n+1)!S^{n+1}_{a,b}(\gamma)\right\rVert_{\pi},
\end{align*}
which is equivalent to solving the following optimisation problem
\begin{equation}\label{notunitspeedopt}
\min_{\left\lVert x\right\rVert_2=1}\left\lVert L I^{\gamma}_{p,n}(x)-(n+1)!S^{n+1}_{a,b}(\gamma)\right\rVert_{\pi}.
\end{equation}
Hence if we solve problem (\ref{notunitspeedopt}), we will obtain an approximation to the derivative of $\gamma$ at $\theta$ when it is parametrised at unit speed. Therefore we can still recover the path, but maybe at a different speed of parametrisation from the underlying. The same argument clearly applies to the result of Theorem \ref{differentiablepathreconstruction}.
\end{proof}
\begin{rmk}
The significance of Lemma \ref{differentscale} is that it provides us with a generalised version of the insertion algorithm we have developed, and we will then be able to reconstruct a path even if it is parametrised at an unknown speed. In fact it shows that the insertion algorithm developed for inverting the signature of a path requires the knowledge of the length of the path. A particular example can be found in the next section in Example \ref{digit8egch5}.
\end{rmk}

\section{Computational reconstruction of a path from its signature}
\label{ch:computation}
We have seen in the previous section that a path can be reconstructed by solving an optimisation problem after inserting an element into a level of the signature of the path. In this section we demonstrate computationally how to use this method to recover a path.
Suppose $\gamma:[0,1]\rightarrow\mathbb{R}^d$ is a tree-reduced continuous bounded-variation path with derivative $f:(0,1)\rightarrow\mathbb{R}^d$ such that $\lVert f(t)\rVert_2=1$ for all $t\in(0,1)$ almost everywhere, and $\gamma$ is differentiable at $\theta\in(0,1)$. We have seen from the previous section that given certain assumptions are satisfied, the key to reconstruct the path from the signature is to solve the optimisation problem 
\begin{equation}\label{optimisation}
\min_{\lVert x\rVert_2=1} \left\lVert I_{p,n}(x)-\bar{S}_{n+1}\right\rVert_{\pi},
\end{equation}
where $p=
\left\lfloor\theta(n+2)\right\rfloor$. If we want to computationally reconstruct the path from its signature, it is necessary to consider programmes which solve the non-linear optimisation problem (\ref{optimisation}). Note in practice the projective tensor norm $\lVert\cdot\rVert_{\pi}$ is difficult to compute, we can generalise the problem to a wider set of tensor norms:\newline
\textbf{Problem.} For a norm function $\lVert\cdot\rVert$ which satisfies Definition \ref{norm}, assume a tree-reduced continuous bounded-variation path $\gamma:[0,1]\rightarrow\mathbb{R}^d$ with derivative $f:(0,1)\rightarrow\mathbb{R}^d$ such that $\lVert f(t)\rVert=1$ for all $t\in(0,1)$ almost everywhere, and $\gamma$ is differentiable at $\theta\in(0,1)$. For all $n\geq 1$, define $g:\mathbb{R}^d\to\mathbb{R}$ such that $$g(x):=\left\lVert I_{p,n}(x)-\bar{S}_{n+1}\right\rVert$$ for $p\in\{1,\cdots,n+1\}$. We are interested in the following optimisation problem 
\begin{align}\label{optprobanynorm}
\min_{\lVert x\rVert=1}g(x).
\end{align} 
\begin{lem}\label{existencesol}
There exists at least one solution to (\ref{optprobanynorm}).
\end{lem}
\begin{proof}
We first show that $g$ is a continuous function: for $x,y\in\mathbb{R}^d$, we have
\begin{align*}
\left|g(x)-g(y)\right|&\leq \left\lVert I_{p,n}(x)-I_{p,n}(y)\right\rVert\\
&=\left\lVert I_{p,n}(x-y)\right\rVert\\
&=\left\lVert x-y\right\rVert\left\lVert\bar{S}_n\right\rVert,
\end{align*}
so $g$ is Lipschitz hence continuous. The set $\{x\in\mathbb{R}^d: \lVert x\rVert =1\}$ is closed and bounded in $\mathbb{R}^d$, so there exists $x^*\in \{x\in\mathbb{R}^d: \lVert x\rVert =1\}$ such that $g(x^*)=\min_{\lVert x\rVert=1}g(x)$.
\end{proof}
After proving the existence, a natural question to ask is whether the solution is unique. In the next section we will prove that if we identify $(\mathbb{R}^d)^{\otimes n}$ with $(\mathbb{R}^d)^{n}$, then under $\ell^2$ norm the minimiser is unique using the method of Lagrange multipliers.
\subsection{Application of the method of Lagrange multipliers}
If $\left\lVert I_{p,n}(x)-\bar{S}_{n+1}\right\rVert$ is a smooth function under the norm we choose, then a practical method to find a minimum to the problem is using \emph{Lagrange multipliers}. As an example, let us consider a tree-reduced $d$-dimensional path $\gamma:[0,1]\to\mathbb{R}^d$ parametrised at unit speed, i.e. $\lVert\dot{\gamma}_t\rVert_{2}=1$. For any $n\geq 1$, $p\in\{1,\cdots,n+1\}$, let $A\in\mathbb{R}^{d^{n+1}\times d}$ denote the matrix representing the linear mapping $I_{p,n}$, and $b\in\mathbb{R}^{d^{n+1}}$ be the normalised signature of $\gamma$ at level $n+1$. We now try to find a solution to 
\begin{equation}\label{l2minimisation}
\min_{x\in\mathbb{R}^d, \lVert x\rVert_2=1}\left\lVert Ax-b\right\rVert_2.
\end{equation}
We first note the following property of $A$.
\begin{lem}\label{findsingularvals}
Assume that in $\mathbb{R}^d$, $A$ is same the matrix as in (\ref{l2minimisation}). The singular values of $A$ are the same and equal to $\left\lVert\bar{S}_n\right\rVert_2$.
\end{lem}
\begin{proof}
Let $\{e_1, e_2,\cdots,e_d\}$ be a basis of $\mathbb{R}$, and $\bar{S}_n=\sum_{i\in I(n)}a_{i_1i_2\cdots i_n}e_{i_1}\otimes\cdots\otimes e_{i_n}$ for the set $I(n)$ of all words of length $n$ over the alphabet $\{1,\cdots,d\}$. Note the matrix can be obtained by applying the map $I_{p,n}$ on the basis $\{e_1,\cdots,e_n\}$ of $\mathbb{R}$, which gives elements in $\mathbb{R}^{\otimes (n+1)}$. Therefore we can identify the entries in $A$ by the bases of $\mathbb{R}$ and $\mathbb{R}^{\otimes (n+1)}$ simultaneously, and write entries of $A$ as $A_{i_1\cdots i_{n+1},j}$ for all $i\in I(n+1)$ and $j\in\{1,\cdots,d\}$. Then by the definition of $I_{p,n}$, we have
\begin{align*}
A_{i_1\cdots i_{p-1}i_pi_{p+1}\cdots i_{n+1},j}=\begin{cases} 
      a_{i_1\cdots i_{p-1}i_{p+1}\cdots i_{n+1}} & \text{if}\,j=i_p, \\
      0 & \text{otherwise}. 
   \end{cases}
\end{align*}
Hence
\begin{align*}
A^TA=\begin{pmatrix}
\left\lVert\bar{S}_n\right\rVert_2^2 & 0 & \cdots & 0\\
0 & \left\lVert\bar{S}_n\right\rVert_2^2 & 0 & \vdots \\
\vdots & 0 & \ddots & \vdots \\
0 & \cdots & 0 & \lVert\bar{S}_n\rVert_2^2
\end{pmatrix},
\end{align*}
which is a diagonal matrix with all diagonal entries equal to $\left\lVert\bar{S}_n\right\rVert_2^2$. Then by the definition of singular values, the singular values of $A$ are equal to $\left\lVert\bar{S}_n\right\rVert_2$.
\end{proof}
Because the objective function in (\ref{l2minimisation}) is differentiable, we can use the classical \emph{method of Lagrange multipliers}.\newline
Now we can show that problem (\ref{l2minimisation}) admits a unique solution on the sphere. 
\begin{prop}\label{l2uniqueness}
There exists a unique solution to problem (\ref{l2minimisation}), and we can develop an explicit formula for the minimum using the method of Lagrange multipliers.
\end{prop}
\begin{proof}
Applying \emph{singular value decomposition} on $A$, we can write
\begin{align*}
A=U\Sigma V^T,
\end{align*}
where $U\in \mathbb{R}^{d^{n+1}\times d^{n+1}}$ is an orthogonal matrix, $\Sigma\in\mathbb{R}^{d^{n+1}\times d}$ is a diagonal matrix, and $V\in\mathbb{R}^{d\times d}$ is an orthogonal matrix. Note 
\begin{align*}
\left\lVert Ax-b\right\rVert_2&=\left\lVert U\Sigma V^T x-b\right\rVert_2\\
&=\left\lVert\Sigma V^T x-U^T b\right\rVert_2.
\end{align*}
Recall from Lemma \ref{findsingularvals} that the singular values of $A$ are equal to $\left\lVert\bar{S}_n\right\rVert_{2}$. Define $\lambda:=\left\lVert\bar{S}_n\right\rVert_{2}$, and write 
\begin{align*}
\Sigma=\begin{pmatrix} 
\lambda & 0  & \cdots & 0 \\
0 & \ddots &  & \vdots \\
\vdots &  & \ddots & 0 \\
0& \cdots & 0 & \lambda\\
0 &\cdots &\cdots & 0\\
\vdots & \vdots & \vdots & \vdots
\end{pmatrix}, 
\quad
V^Tx=\begin{pmatrix}
q_1\\
q_2\\
\vdots\\
q_d
\end{pmatrix},
\quad
U^Tb=\begin{pmatrix}
y_1\\
y_2\\
\vdots\\
y_{d^{n+1}}
\end{pmatrix}.
\end{align*} 
Then 
\begin{align*}
\left\lVert\Sigma V^T x-U^Tb\right\rVert_2&=
\left\lVert \begin{pmatrix} 
\lambda & 0  & \cdots & 0 \\
0 & \ddots &  & \vdots \\
\vdots &  & \ddots & 0 \\
0& \cdots & 0 & \lambda\\
0 &\cdots &\cdots & 0\\
\vdots & \vdots & \vdots & \vdots
\end{pmatrix}
\begin{pmatrix}
q_1\\
q_2\\
\vdots\\
q_d
\end{pmatrix}
-
\begin{pmatrix}
y_1\\
y_2\\
\vdots\\
y_{d^{n+1}}
\end{pmatrix}\right\rVert_2\\
&=\left\lVert\begin{pmatrix}
\lambda q_1-y_1\\
\lambda q_2-y_2\\
\vdots\\
\lambda q_d-y_d\\
-y_{d+1}\\
\vdots\\
-y_{d^{n+1}}
\end{pmatrix}\right\rVert_2.
\end{align*}
Note that $\lVert V^Tx\rVert_2=\lVert x\rVert_2$ since $V$ is orthogonal. Note also that $q_i$ for $i=1,\cdots,d$ only appear in the first $d$ entries, and therefore (\ref{l2minimisation}) is equivalent to 
\begin{align*}
\min\sum_{i=1}^d(\lambda q_i-y_i)^2\quad\text{subject to}\quad\sum_{i=1}^dq_i^2=1.
\end{align*}
By assuming $\lVert\bar{S}_n\rVert_2> 0$, we have $\lambda> 0$, then from above we can see that $\left(\frac{y_i}{\sqrt{\sum_{j=1}^d y_j^2}}\right)_{i=1,\cdots,d}$ is the global minimum. Hence the solution to (\ref{l2minimisation}) is unique by the way our optimisation problem is proposed. Finally we get the minimum $x^*$ by $x=VV^Tx$
\end{proof}
\begin{cor}\label{l2uniquenessnotunitspeed}
Assume a tree-reduced continuous bounded-variation path $\gamma:[a,b]\to\mathbb{R}^d$ is of length $L>0$ and differentiable at $\theta\in(a,b)$, and suppose $\gamma$ is parametrised at an unknown speed. Then there exists a unique solution to problem
\begin{equation}\label{notunitspeedmatrix}
\min_{x\in\mathbb{R}^d, \lVert x\rVert_2=1}\left\lVert LAx-b\right\rVert_2,
\end{equation}
where $A$ and $b$ are as described in (\ref{l2minimisation}).
\end{cor}
\begin{proof}
We have seen from Lemma \ref{differentscale} that if we do not know the speed of parametrisation of the path, we can solve the optimisation problem (\ref{notunitspeedopt}) to get a approximation of the derivative of the path. If we use $\ell^2$ norm, then (\ref{notunitspeedopt}) becomes (\ref{notunitspeedmatrix}). Note the only difference between (\ref{notunitspeedmatrix}) and (\ref{l2minimisation}) is the constant $L$ in front of the matrix $A$, therefore a similar analysis as in Proposition \ref{l2uniqueness} applies, and (\ref{notunitspeedmatrix}) admits a unique solution.
\end{proof}
We now demonstrate some examples of inverting the signature of a path by solving (\ref{l2minimisation}). All of the following computation is done in C++, and the graphs are plotted in MATLAB. The computation of signatures used is done via the C++ library \emph{Libalgebra} \cite{libalgebra}. The matrix computation algorithms used are from \emph{LAPACK} \cite{Anderson:1999:LUG:323215}, and the version used is provided by Intel Math Kernel Library.
\begin{eg}[Semicircle]
Let $\gamma:[0,1]\rightarrow\mathbb{R}^2$ be the path of a semicircle, i.e. $\gamma^{(1)}_t=\frac{1}{\pi}\cos(\pi t)$, $\gamma^{(2)}_t=\frac{1}{\pi}\sin(\pi t)$ for $t\in [0,1]$. If we use $\ell^2$ norm, we can use the formulae obtained in Proposition \ref{l2uniqueness} to get an approximation to the derivative of the path at different time points. Thus we are able to approximate the increments over subintervals by Mean Value Theorem, as shown in Figure \ref{cpp_lagrange_semicircle}. We can see that using higher levels of signature gives better approximations to the true path.
\end{eg}
\begin{figure}
\centering
\includegraphics[trim={4cm 10cm 3cm 10cm},clip,width=\textwidth]{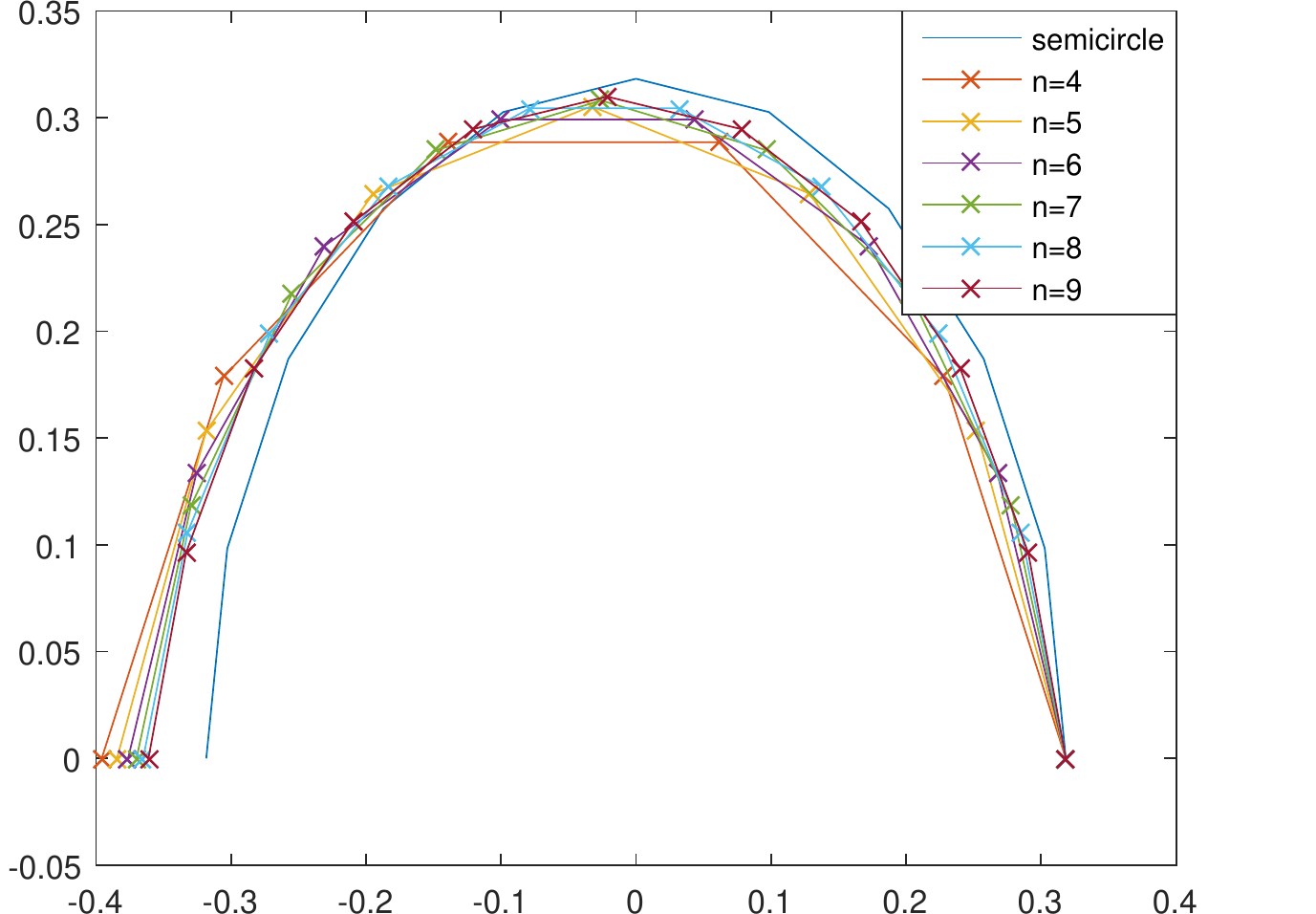}
\caption{Reconstruction of a semicircle under $\ell^2$ norm, where $n$ is the level of signature used}\label{cpp_lagrange_semicircle}
\end{figure}
\begin{eg}[Circle]
Assume in this example $\gamma:[0,1]\rightarrow\mathbb{R}^2$ is the path of a circle such that $\gamma_t=(\frac{1}{2\pi}\cos(2\pi t),\frac{1}{2\pi}\sin(2\pi t))$ for $t\in[0,1]$. Again we used the formulae obtained in Proposition \ref{l2uniqueness} to get an approximation to the derivative at different times in $\ell^2$ norm, therefore an approximation of the increments over the sub-intervals, as shown in Figure \ref{cpp_lagrange_circle}.
\end{eg}
\begin{figure}
\centering 
\includegraphics[trim={4cm 10cm 3cm 10cm},clip,width=\textwidth]{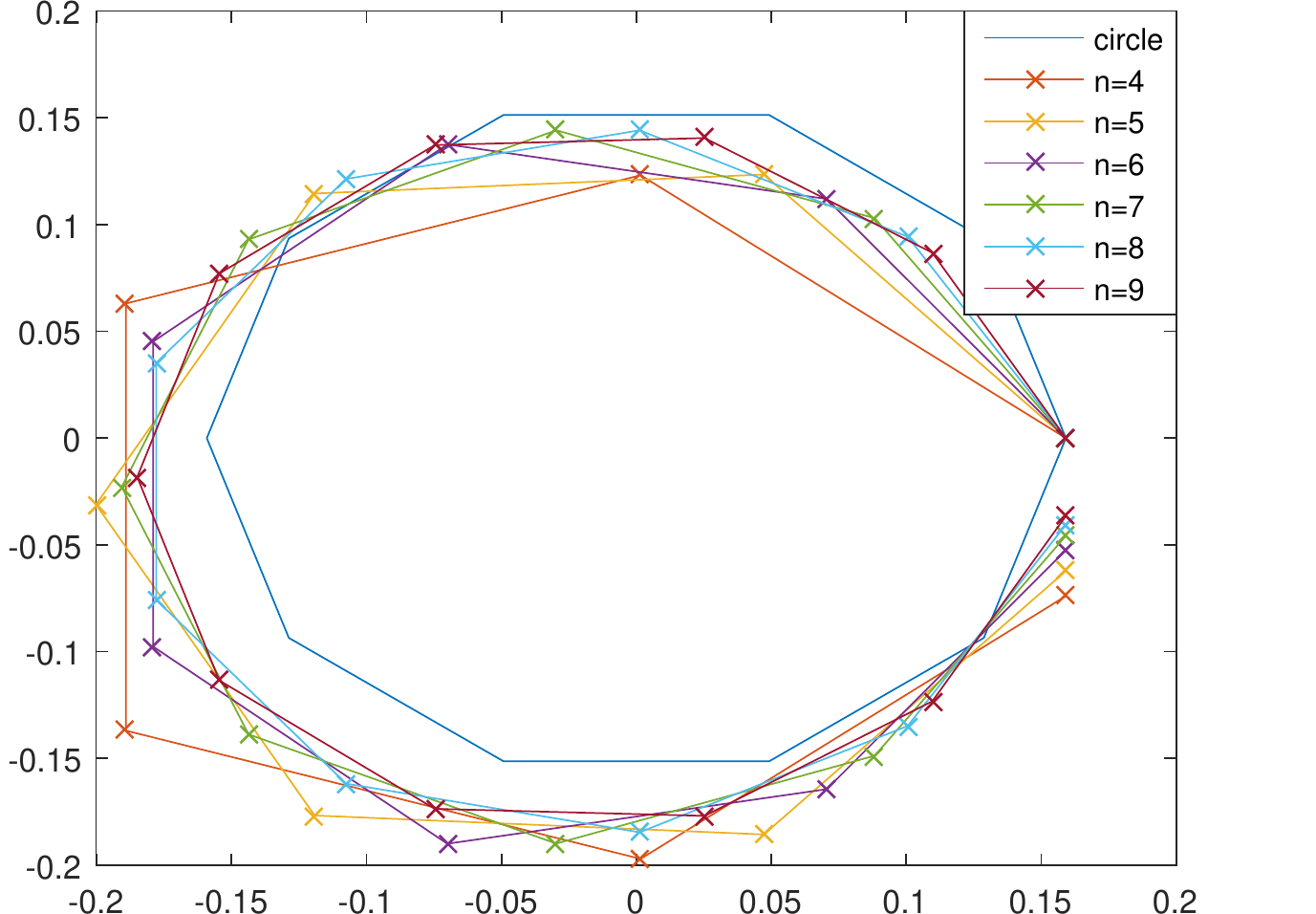}
\caption{Reconstruction of a circle under $\ell^2$ norm, where $n$ is the level of signature used}\label{cpp_lagrange_circle}
\end{figure}
\begin{eg}[Digit `8']\label{digit8egch5}
One interesting case to consider is when the path is self-crossing. A good example of this kind is digits.\newline
The dataset we use is \emph{Pen-Based Recognition of Handwritten  Digits Data Set} from UC Irvine Machine Learning Repository \cite{Dua:2017}, which is a digit database by collecting 250 handwritten digit samples from 44 writers. The dataset records the $(x,y)$ coordinates on the $2$-dimensional plane as the participants write. The raw data captured consists of integer values between $0$ and $500$, and then a resampling algorithm is applied so that the points are regularly spaced in arc length.\newline
We have taken one sample of the digit `8' from the training data, and normalise the input vectors so that they consist of values in $[0,1]$. Note that the path now is not necessarily parametrised at unit speed. In this case, we can solve a slightly altered optimisation problem by the result of Lemma \ref{differentscale}, therefore we need an approximation of the length of the path. Due to the conjecture in \cite{chang2018super}, we can approximate the length of the path by taking the $n$-th root of the $n$-th level of the signature multiplied by $n!$.\newline
We then reconstruct the underlying path using the method of Lagrange multipliers to approximate the derivative of the path at different points by the results of Proposition \ref{l2uniqueness} and Corollary \ref{l2uniquenessnotunitspeed}, and use splines to smooth the derivatives, and then integrate over $[1/(n+2),(n+1)/(n+2)]$ in MATLAB to approximate the underlying path, where $n$ is the level of the lower level signature used. Compared to the underlying path in Figure \ref{eightrue}, we can see from Figure \ref{eight4}, \ref{eight5}, \ref{eight6}, \ref{eight7}, \ref{eight8}, \ref{eight9}, and \ref{eight10} that overall we get better approximations when we use higher levels of the signature of the path. Note that the paths reconstructed are at different scales from the underlying path. This is because we have reconstructed the path parametrised at unit speed, as shown in Lemma \ref{differentscale}. Also note that if $\gamma:[u,u+L]\to\mathbb{R}^d$ is a path parametrised at unit speed and of length $L$, then
\begin{align*}
&\int_{u}^{u+L}\dot{\gamma}(t)\mathrm{d}t\\
=&L\int_0^1\dot{\gamma}(Ls+u)\mathrm{d}s,
\end{align*}
therefore the path obtained is the underlying path parametrised at unit speed and scaled by $1/L$. The shapes of the reconstructed paths are not affected even though we have a different speed of parametrisation.
\end{eg}
\begin{figure}
\begin{subfigure}{.5\textwidth}
\centering
\includegraphics[trim={4cm 10cm 3cm 10cm},clip,width=0.9\linewidth]{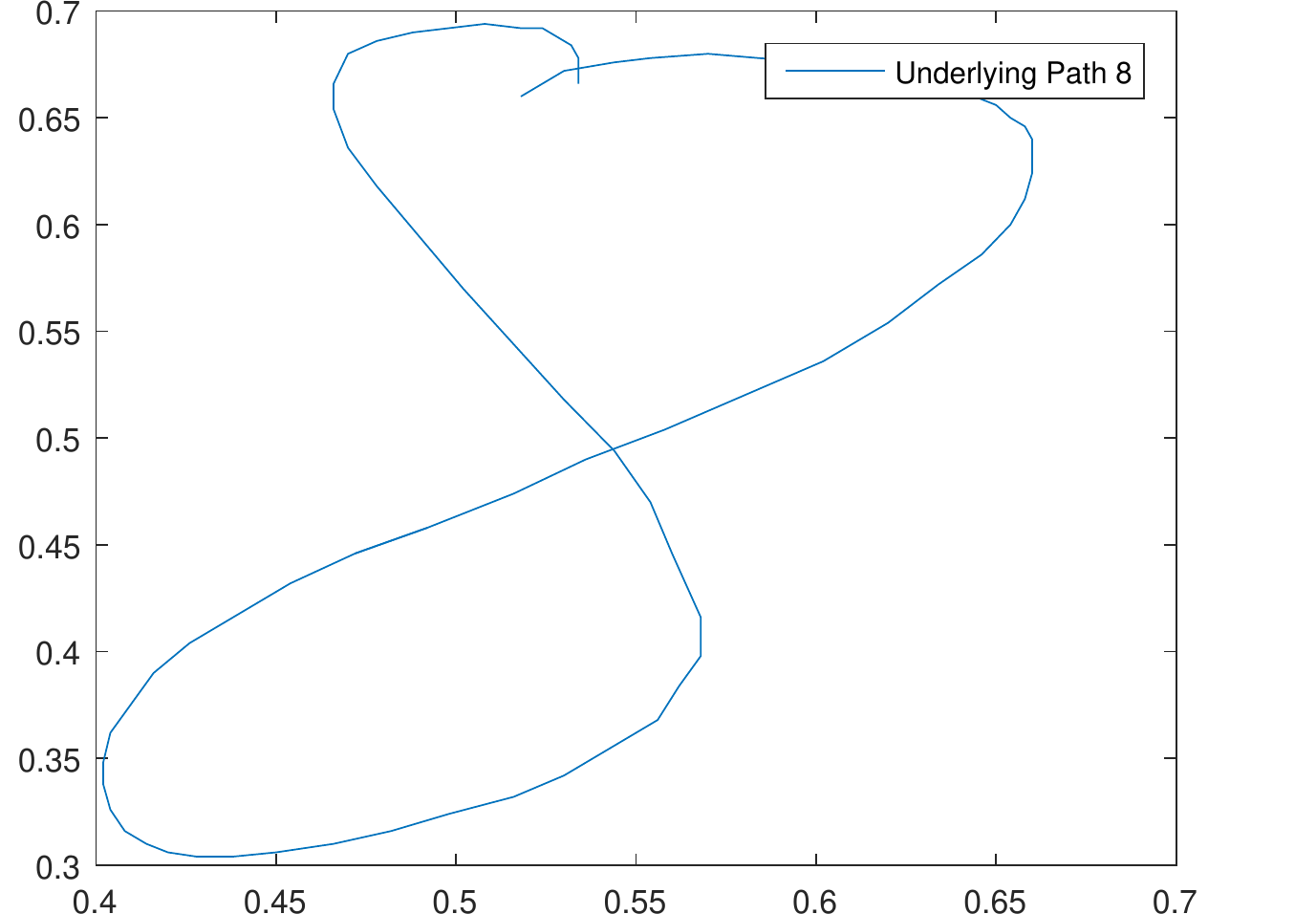}
\caption{Underlying path of the digit `8'}
\label{eightrue}
\end{subfigure}%
\begin{subfigure}{.5\textwidth}
\centering
\includegraphics[trim={4cm 10cm 3cm 10cm},clip,width=0.9\linewidth]{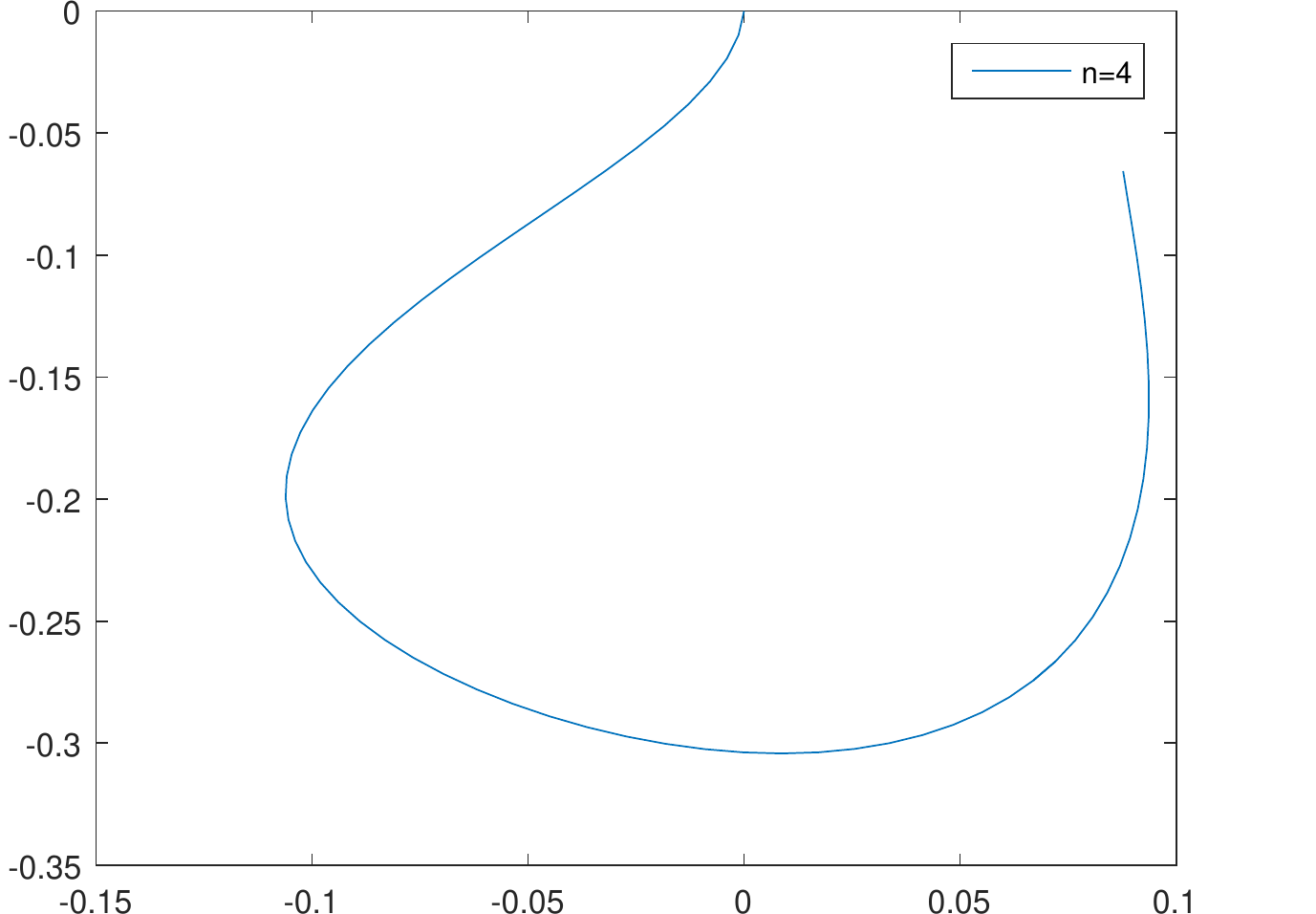}
\caption{Reconstruction using signature level 4}
\label{eight4}
\end{subfigure}

\begin{subfigure}{.5\textwidth}
\centering
\includegraphics[trim={4cm 10cm 3cm 10cm},clip,width=0.9\linewidth]{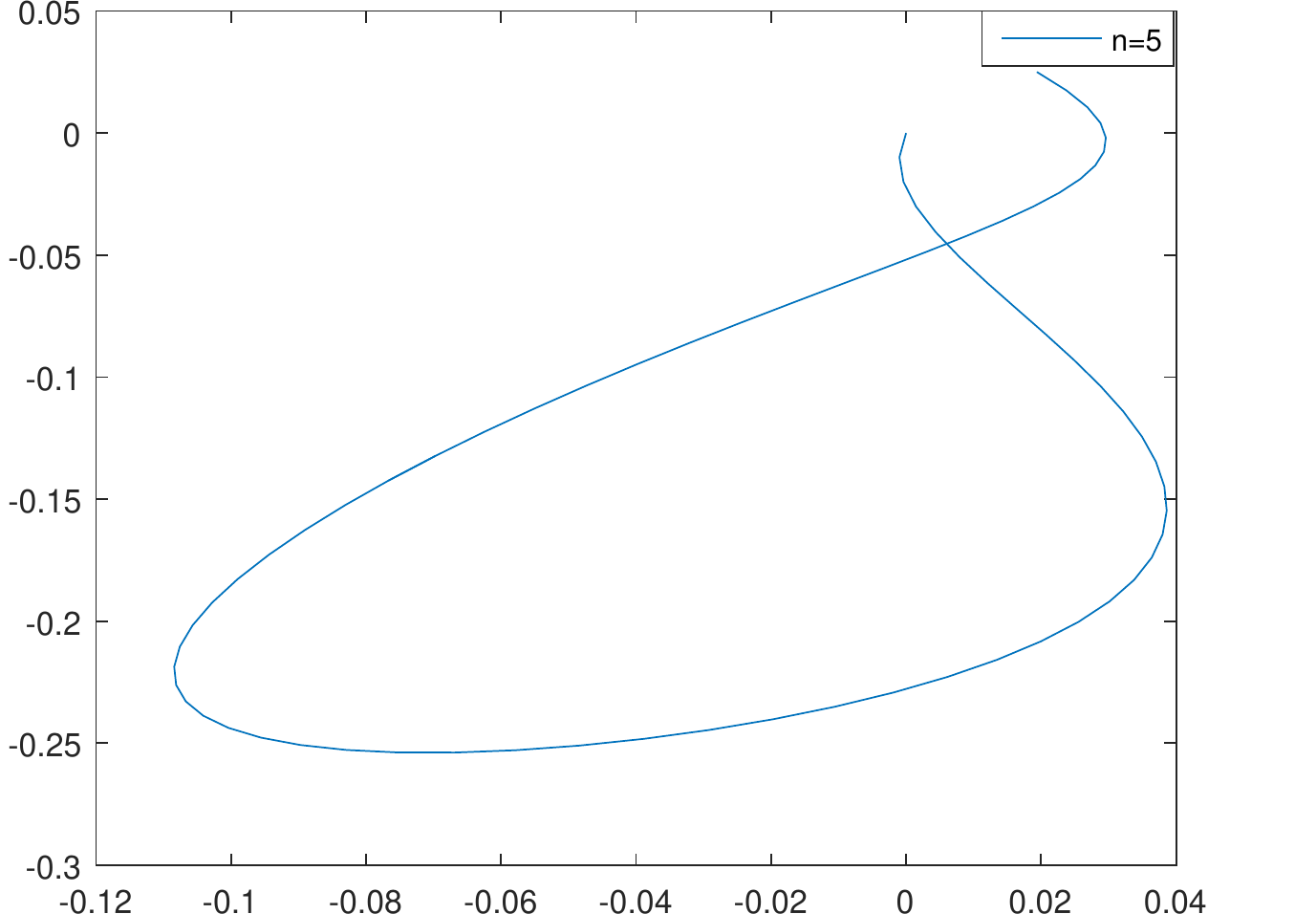}
\caption{Reconstruction using signature level 5}
\label{eight5}
\end{subfigure}%
\begin{subfigure}{.5\textwidth}
\centering
\includegraphics[trim={4cm 10cm 3cm 10cm},clip,width=0.9\linewidth]{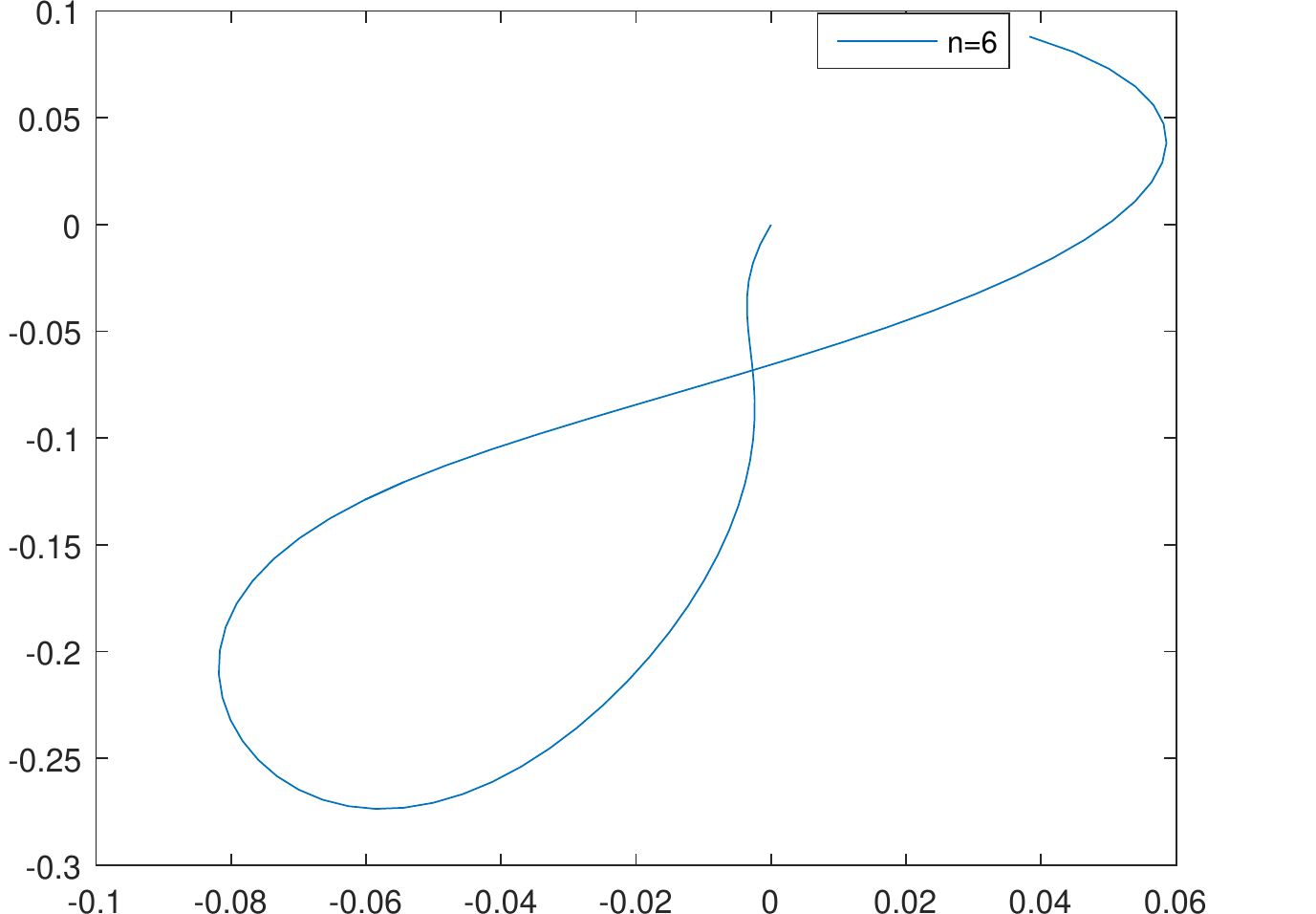}
\caption{Reconstruction using signature level 6}
\label{eight6}
\end{subfigure}

\begin{subfigure}{.5\textwidth}
\centering
\includegraphics[trim={4cm 10cm 3cm 10cm},clip,width=0.9\linewidth]{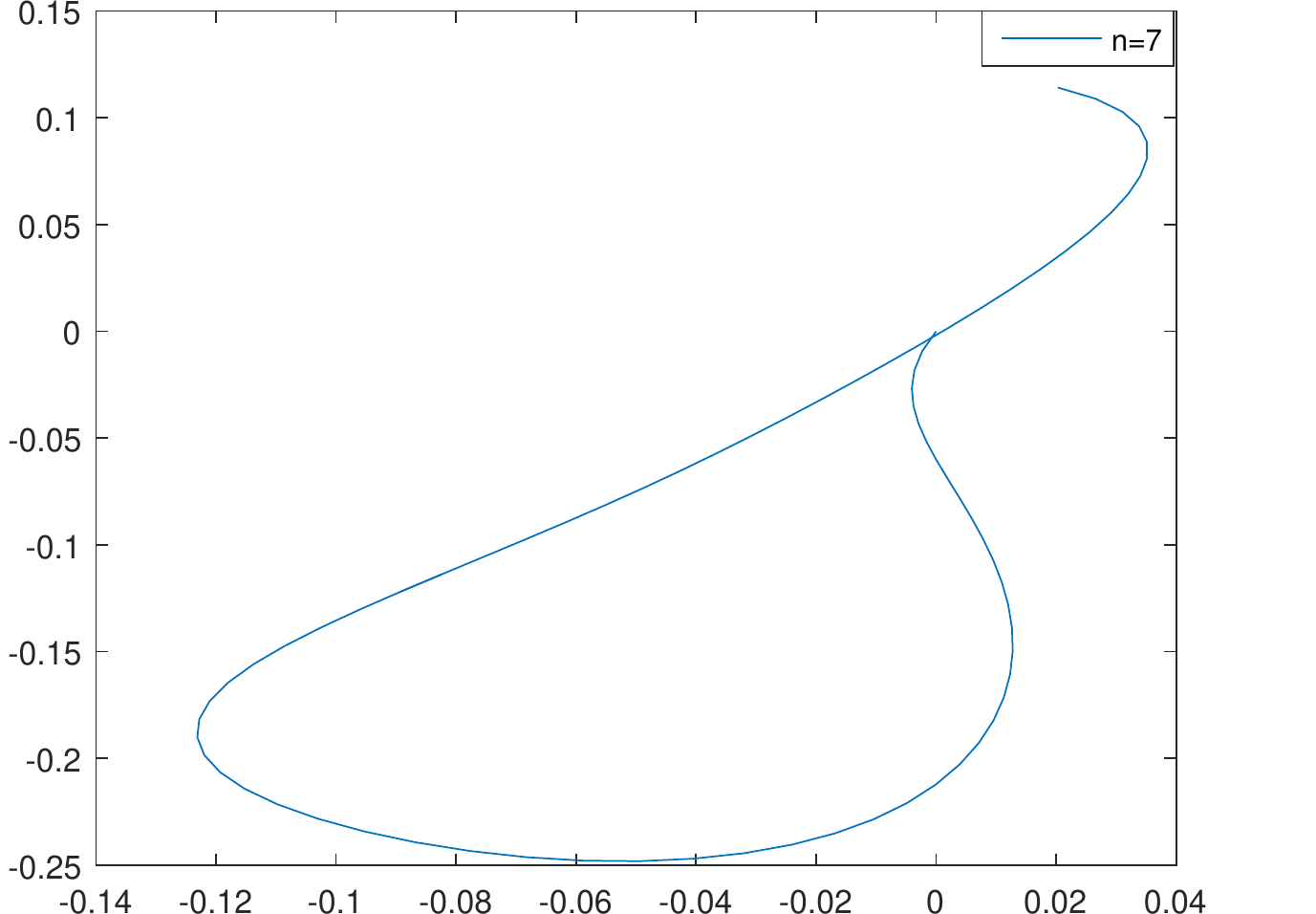}
\caption{Reconstruction using signature level 7}
\label{eight7}
\end{subfigure}%
\begin{subfigure}{.5\textwidth}
\centering
\includegraphics[trim={4cm 10cm 3cm 10cm},clip,width=0.9\linewidth]{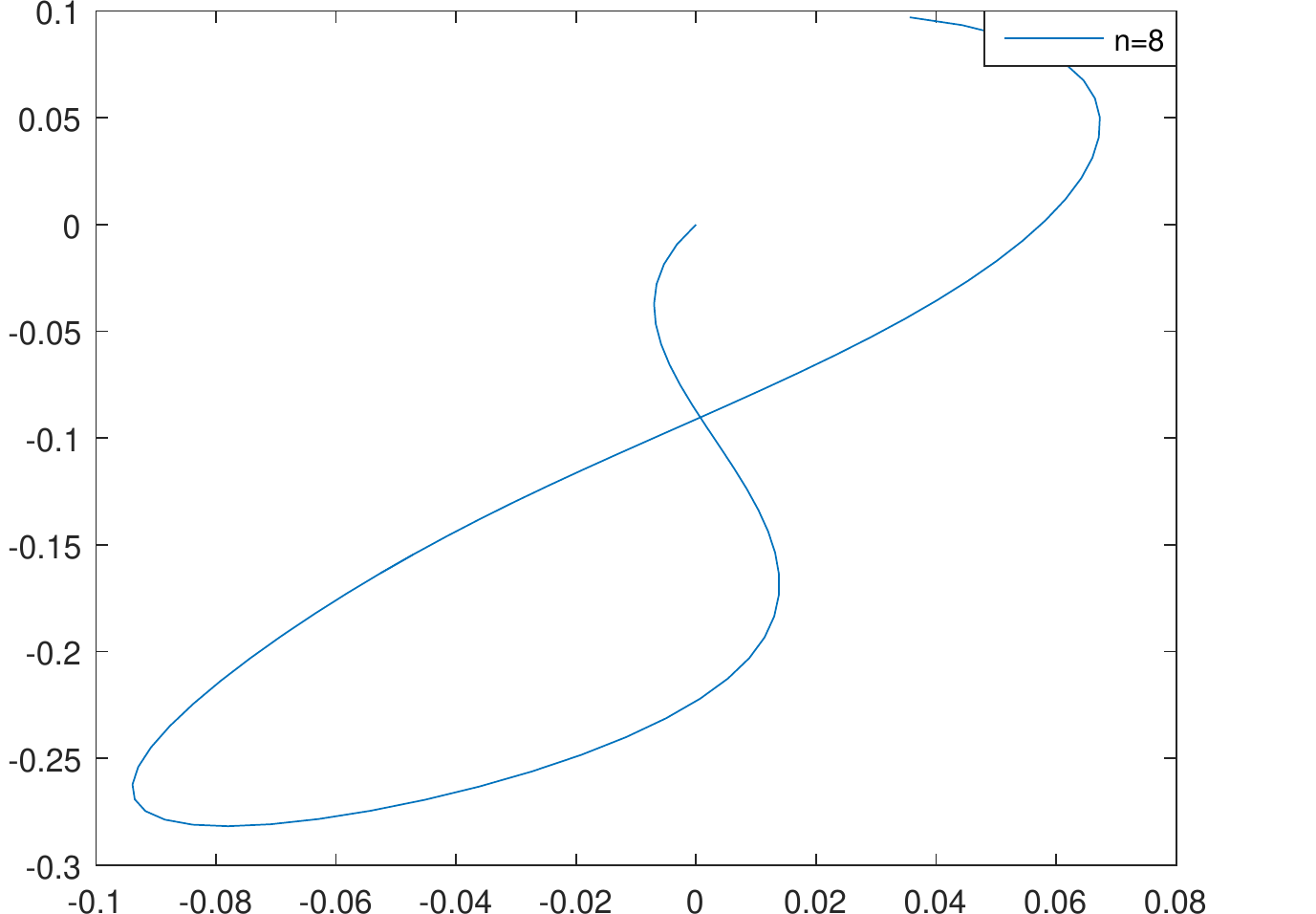}
\caption{Reconstruction using signature level 8}
\label{eight8}
\end{subfigure}

\begin{subfigure}{.5\textwidth}
\centering
\includegraphics[trim={4cm 10cm 3cm 10cm},clip,width=0.9\linewidth]{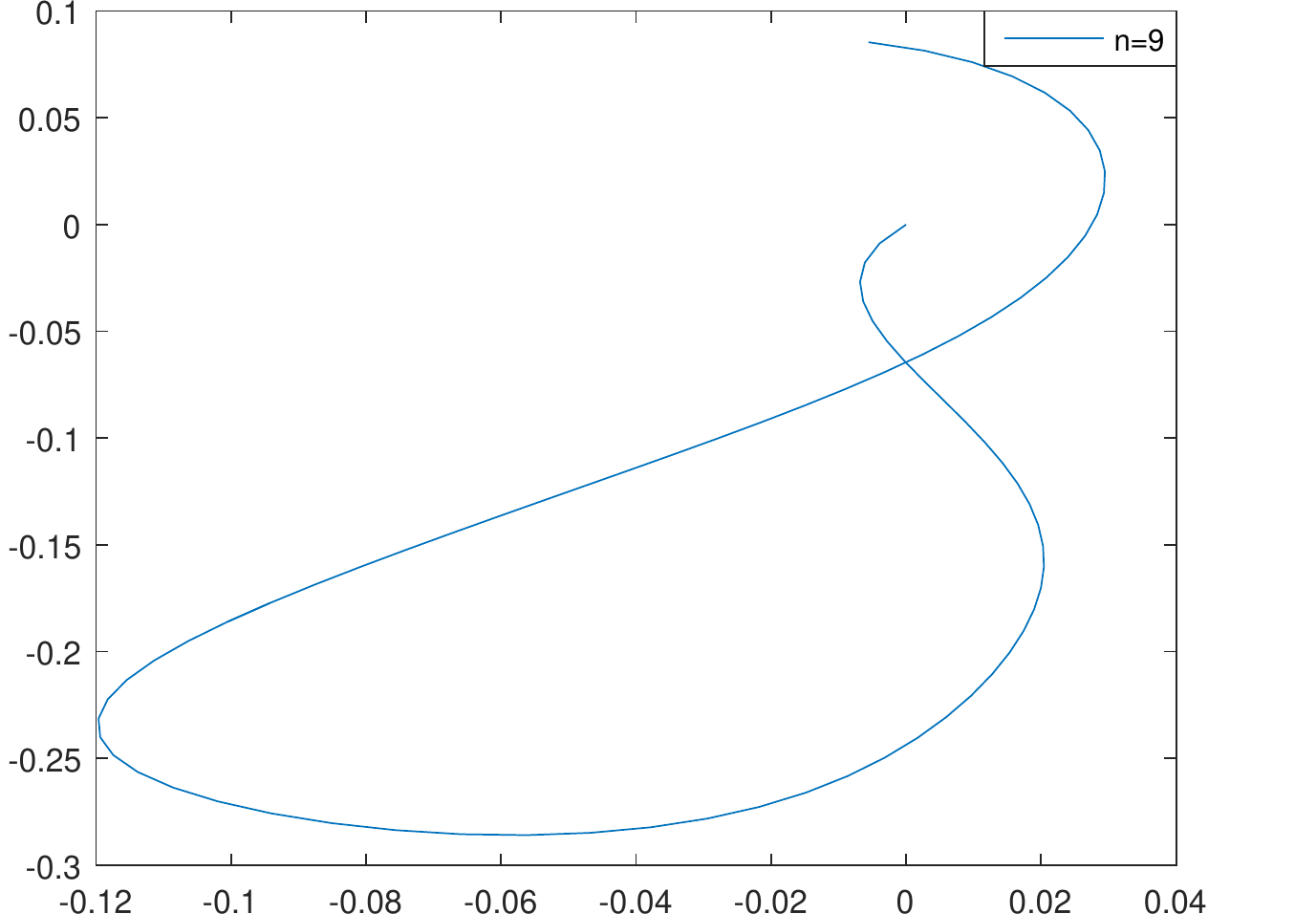}
\caption{Reconstruction using signature level 9}
\label{eight9}
\end{subfigure}%
\begin{subfigure}{.5\textwidth}
\centering
\includegraphics[trim={4cm 10cm 3cm 10cm},clip,width=0.9\linewidth]{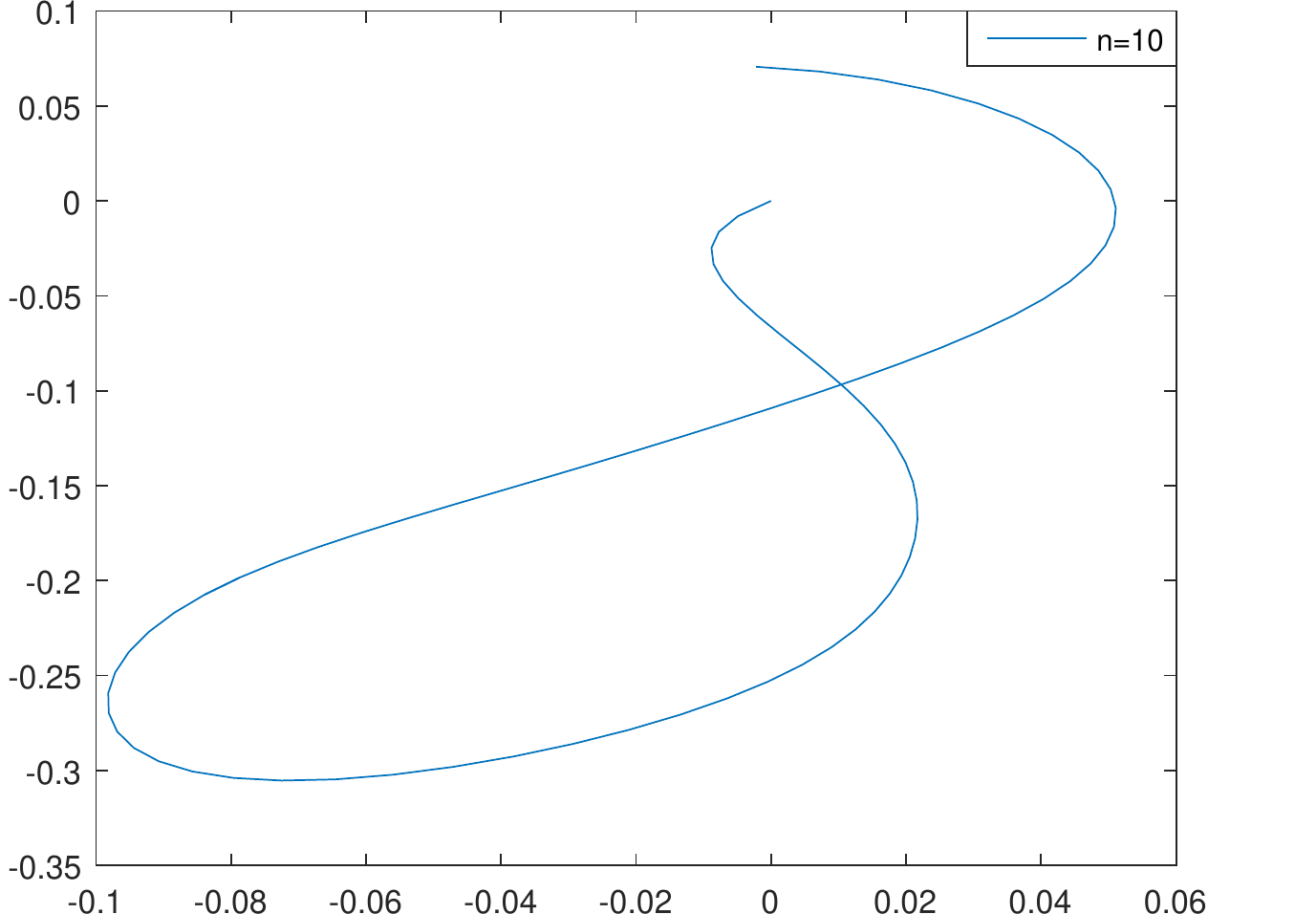}
\caption{Reconstruction using signature level 10}
\label{eight10}
\end{subfigure}
\caption{Reconstruction of the digit `8' using the insertion method}
\end{figure}

\begin{eg}[Robustness of the insertion method]
In this example, we show that we can build a pipeline to invert the signatures of paths by the insertion method. We arbitrarily choose 20 samples from the training set (consisting of handwritten digits by 30 writers) of the Pen-Based Recognition of Handwritten Digits Data Set \cite{Dua:2017} and normalise the data as described in Example \ref{digit8egch5}. Then we reconstruct the underlying path using signature level $9$ and $10$ using the method of Lagrange multipliers as described in Proposition \ref{l2uniqueness} and Corollary \ref{l2uniquenessnotunitspeed}, and obtain Figure \ref{robustdataset-1}, \ref{robustdataset-2}, \ref{robustdataset-3}, \ref{robustdataset-4} and \ref{robustdataset-5}. Note we export the derivatives computed in C++ into MATLAB, and use the splines to approximate the derivatives, and then unlike Example \ref{digit8egch5}, we integrate the splines over $[0,1]$. This is because signature level $9$ is relatively higher than most of the signature levels used in Example \ref{digit8egch5}, so the splines are supposed to behave better at extrapolation. We can see that the insertion method is in general quite robust, however it may not be able to give an accurate approximation at the corner of the path.
\end{eg}

\begin{figure}[t]
\begin{subfigure}{.5\textwidth}
\centering
\includegraphics[trim={4cm 10cm 3cm 10cm},clip,width=0.9\linewidth]{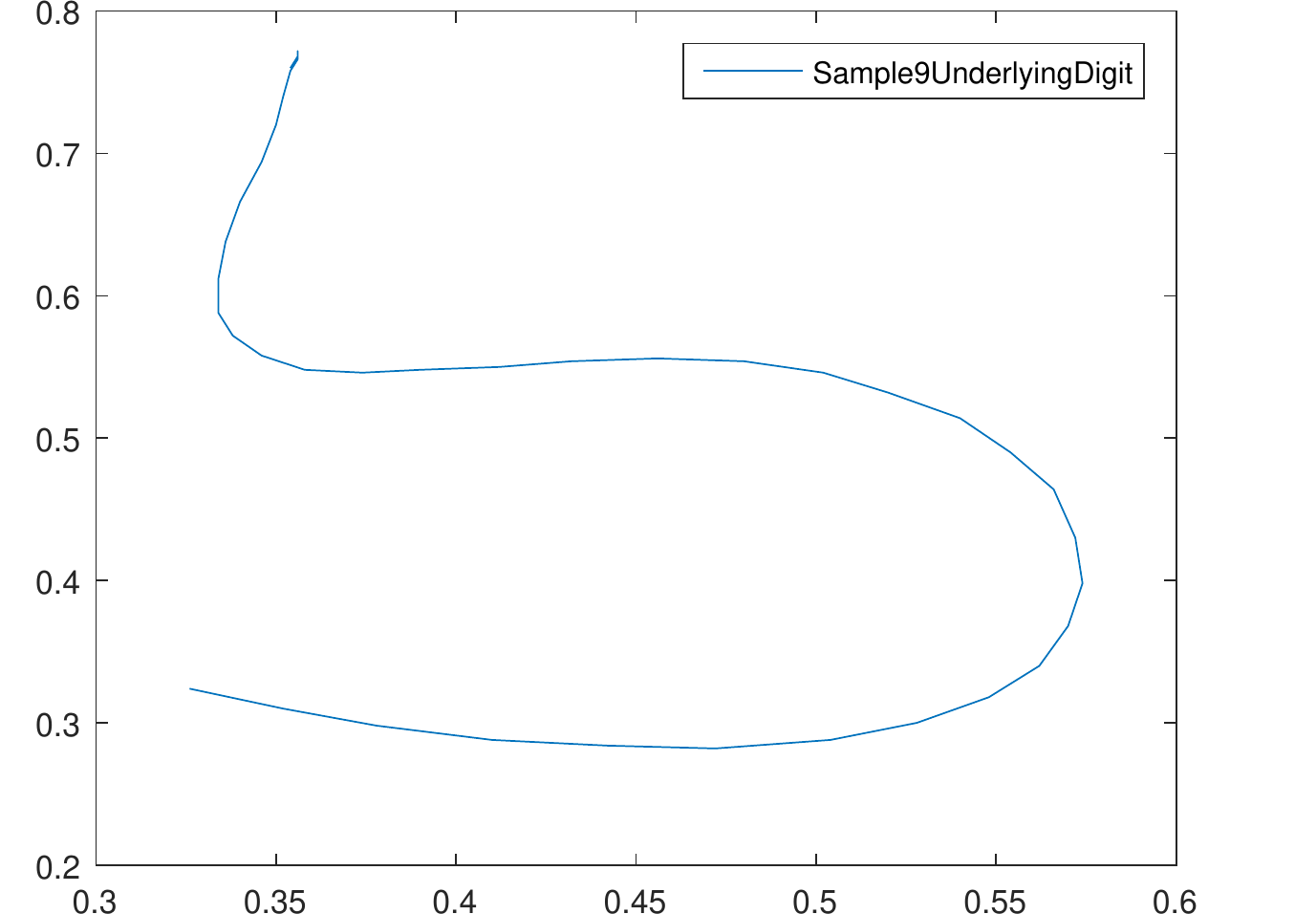}
\caption{Sample 9, underlying digit}
\end{subfigure}%
\begin{subfigure}{.5\textwidth}
\centering
\includegraphics[trim={4cm 10cm 3cm 10cm},clip,width=0.9\linewidth]{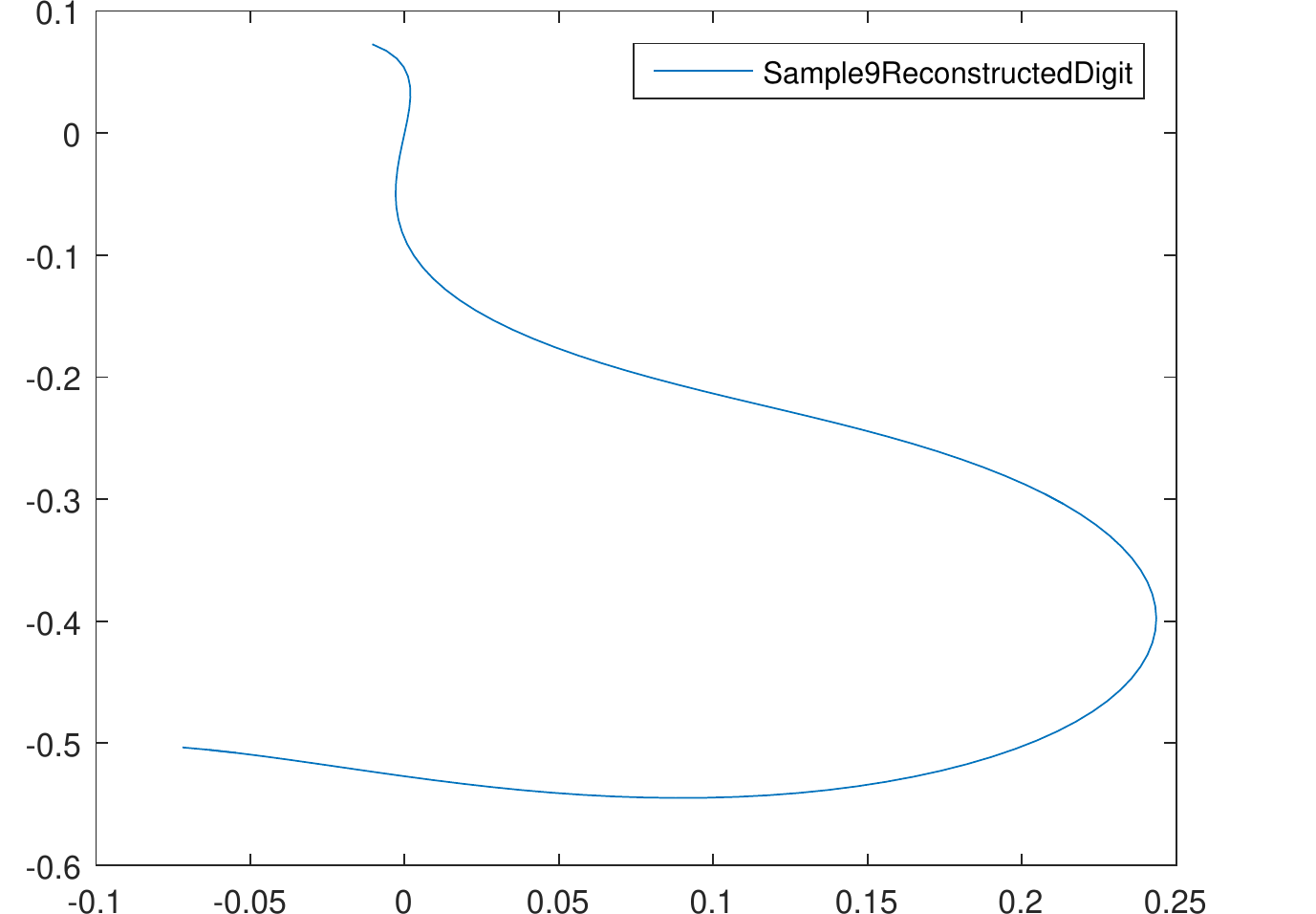}
\caption{Sample 9, reconstructed digit}
\end{subfigure}

\begin{subfigure}{.5\textwidth}
\centering
\includegraphics[trim={4cm 10cm 3cm 10cm},clip,width=0.9\linewidth]{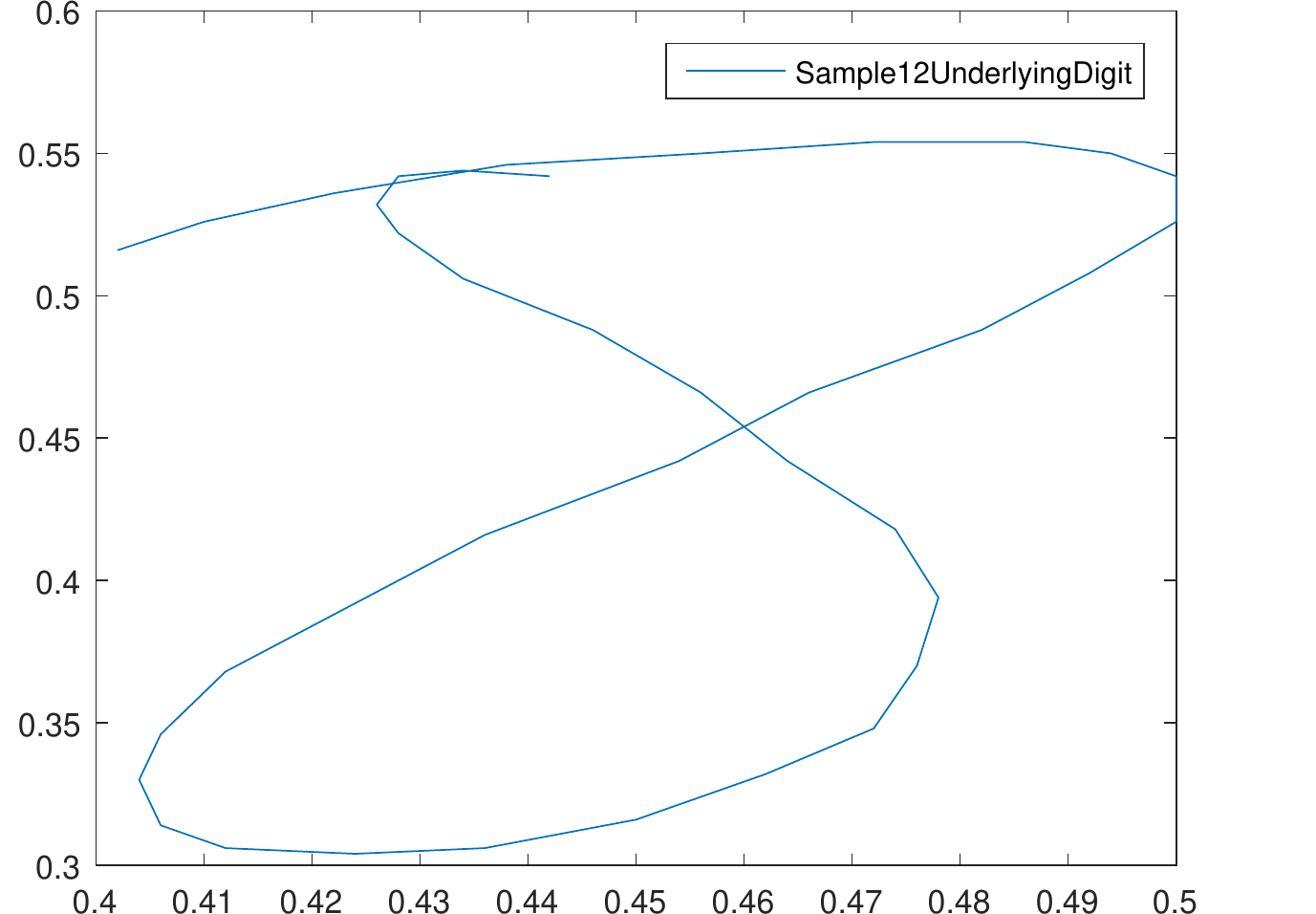}
\caption{Sample 12, underlying digit}
\end{subfigure}%
\begin{subfigure}{.5\textwidth}
\centering
\includegraphics[trim={4cm 10cm 3cm 10cm},clip,width=0.9\linewidth]{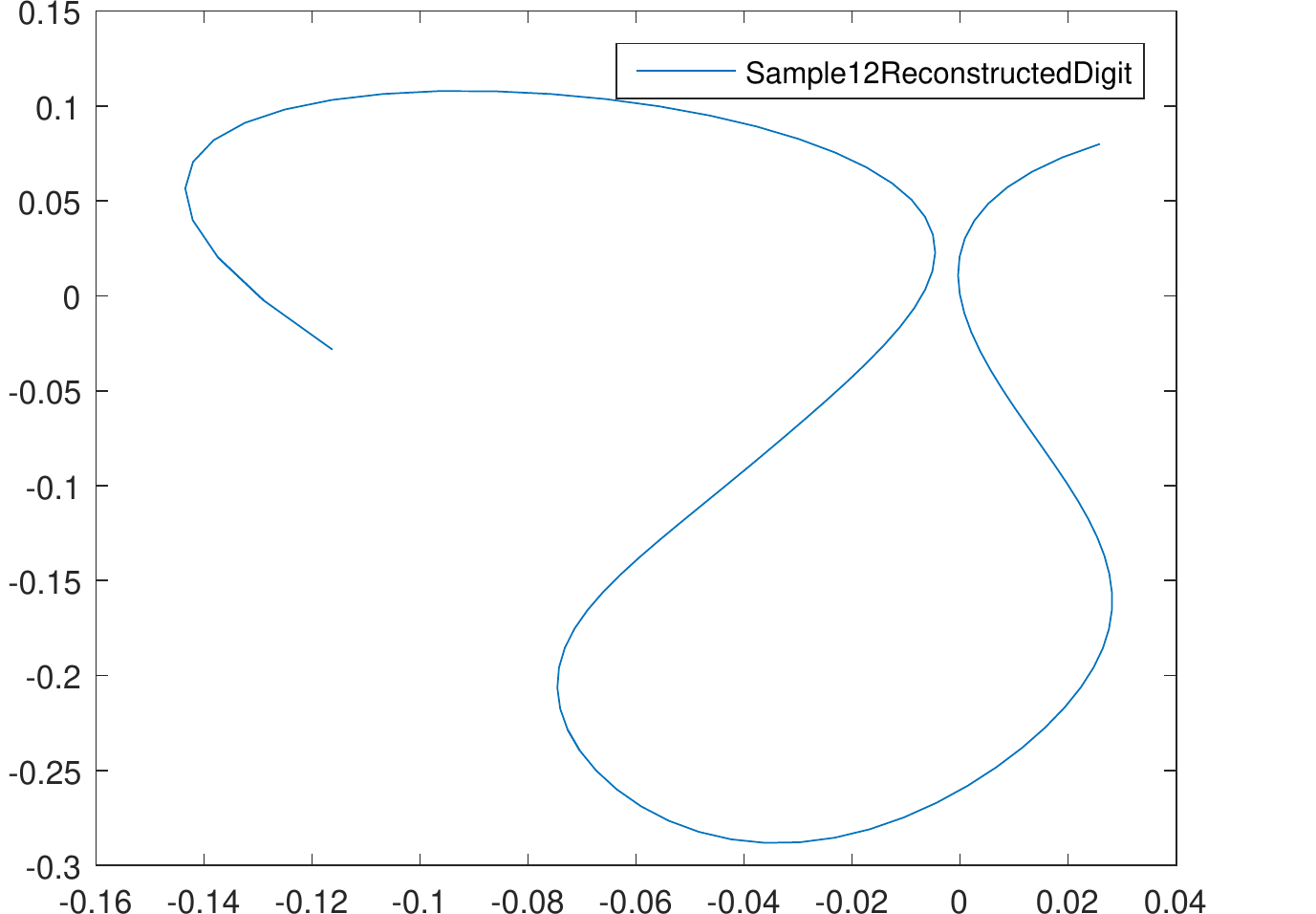}
\caption{Sample 12, reconstructed digit}
\end{subfigure}

\begin{subfigure}{.5\textwidth}
\centering
\includegraphics[trim={4cm 10cm 3cm 10cm},clip,width=0.9\linewidth]{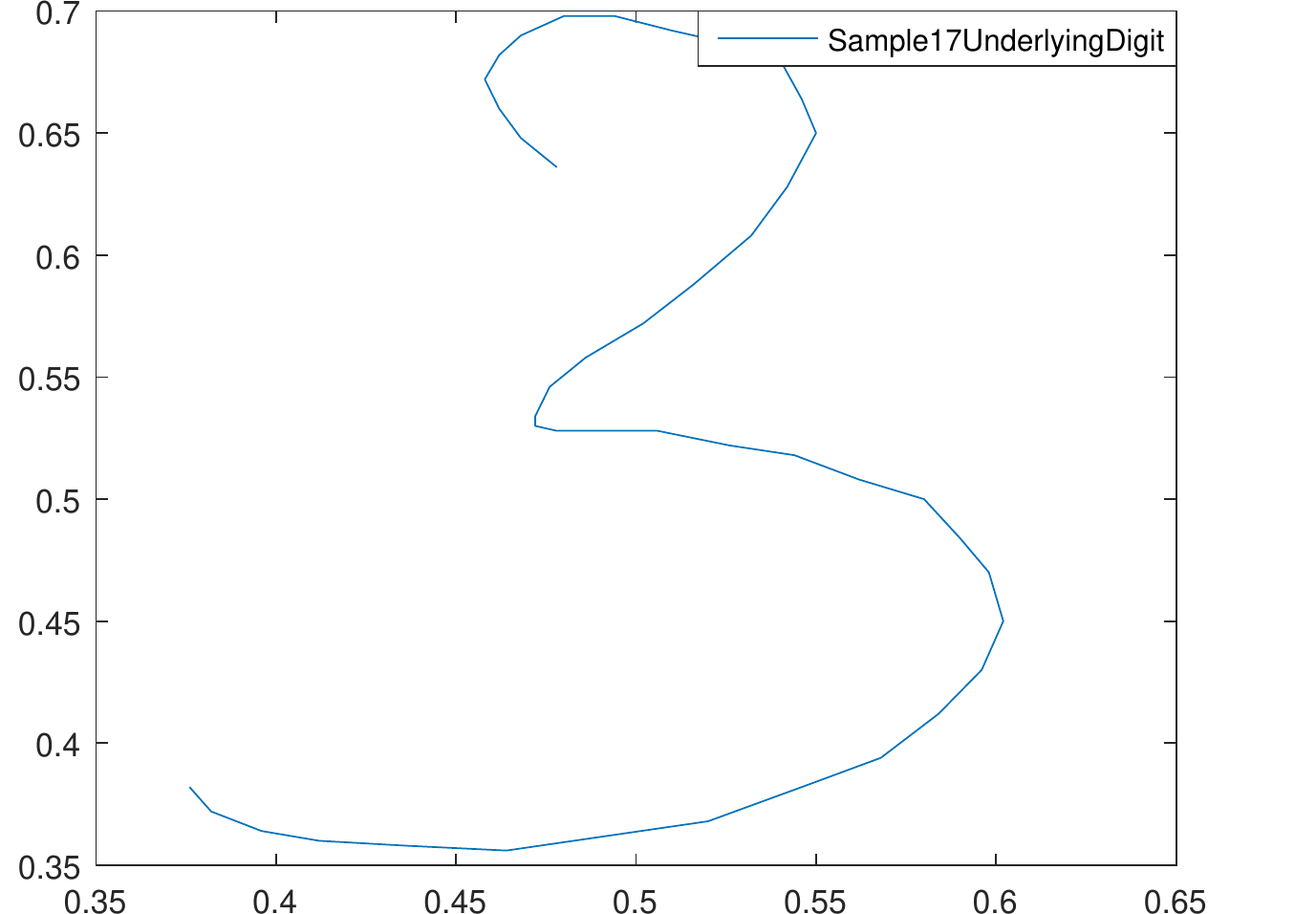}
\caption{Sample 17, underlying digit}
\end{subfigure}%
\begin{subfigure}{.5\textwidth}
\centering
\includegraphics[trim={4cm 10cm 3cm 10cm},clip,width=0.9\linewidth]{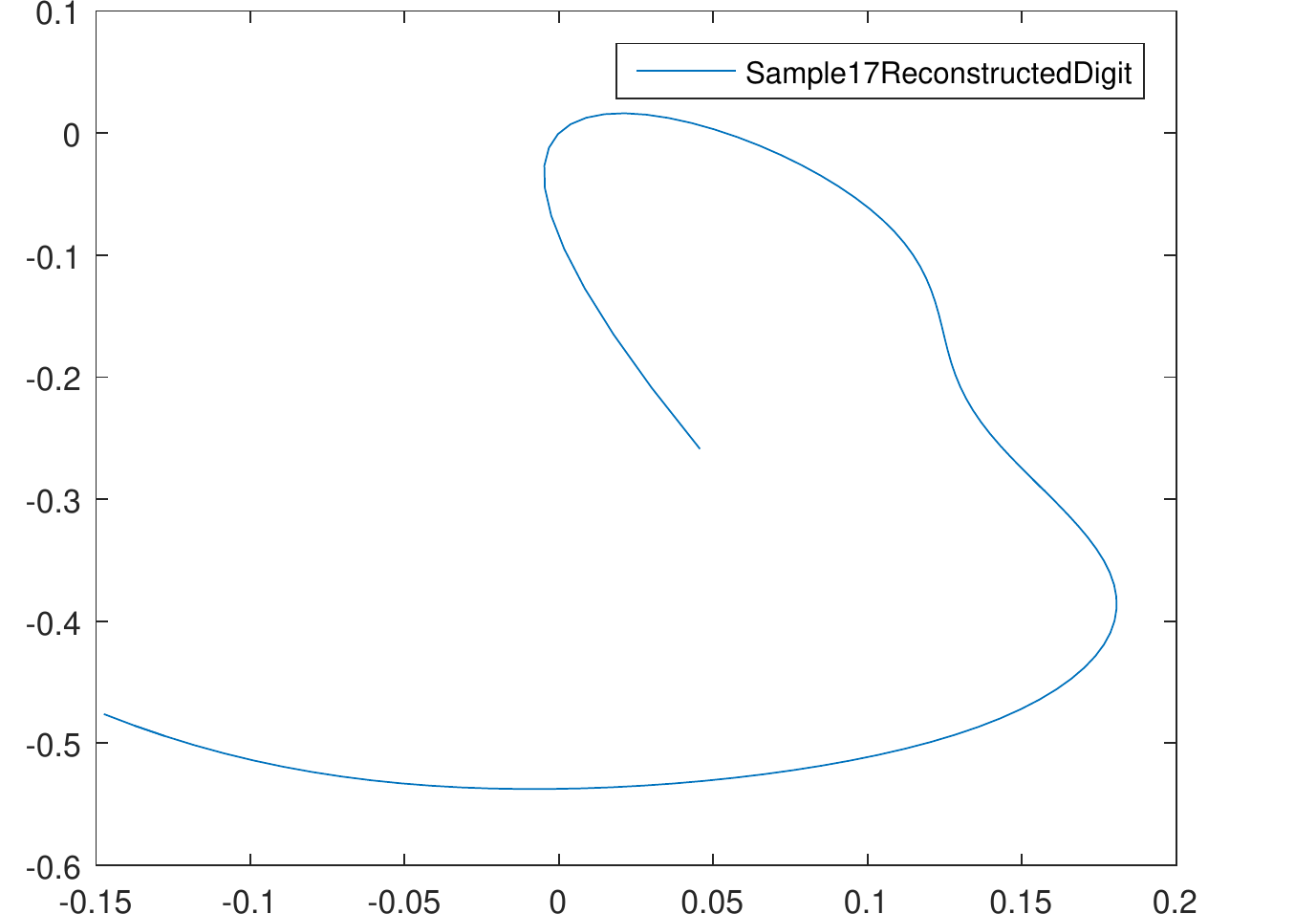}
\caption{Sample 17, reconstructed digit}
\end{subfigure}

\begin{subfigure}{.5\textwidth}
\centering 
\includegraphics[trim={4cm 10cm 3cm 10cm},clip,width=0.9\linewidth]{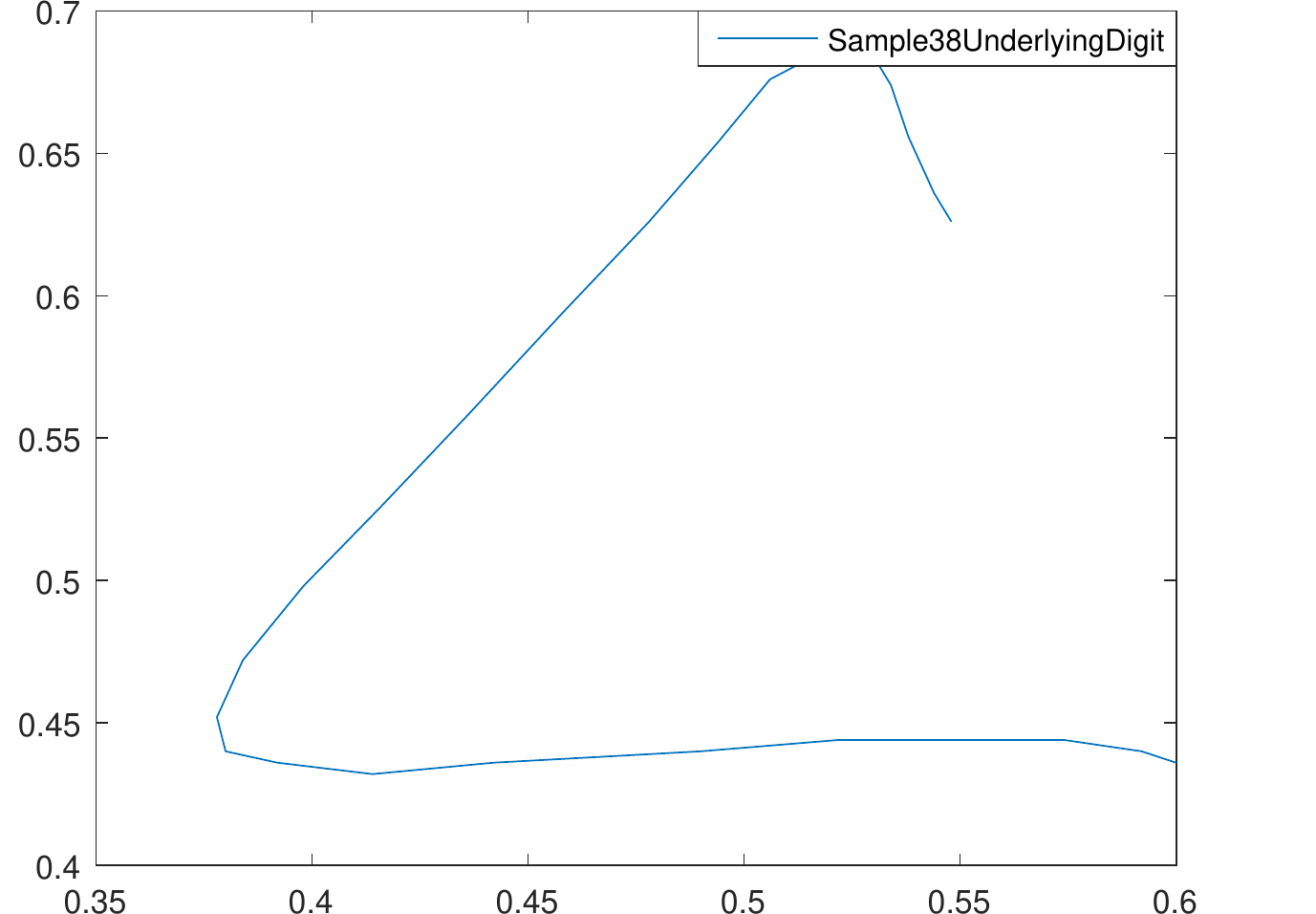}
\caption{Sample 38, underlying digit}
\end{subfigure}%
\begin{subfigure}{.5\textwidth}
\centering
\includegraphics[trim={4cm 10cm 3cm 10cm},clip,width=0.9\linewidth]{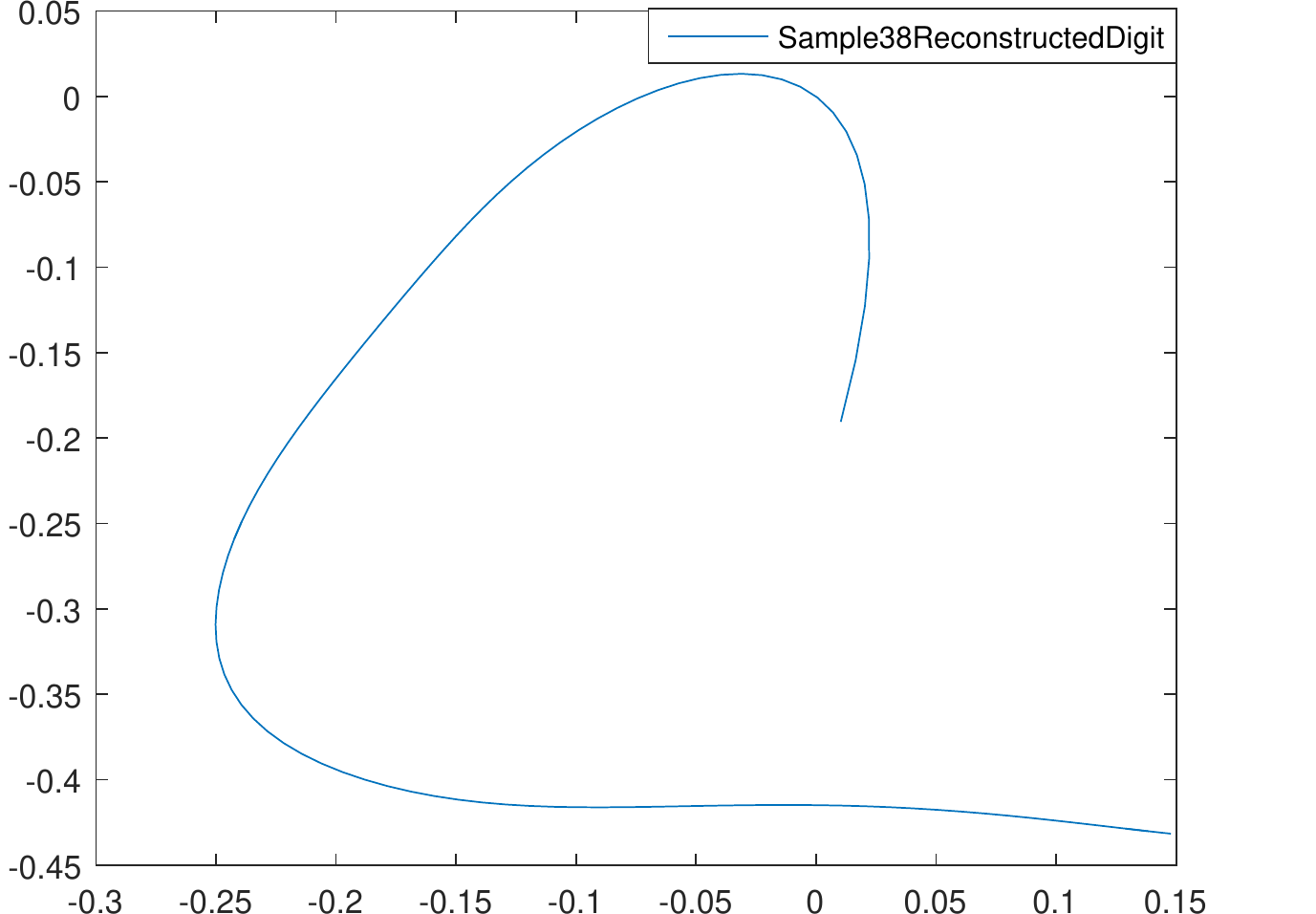}
\caption{Sample 38, reconstructed digit}
\end{subfigure}
\caption{Reconstruction of digits from the data set \cite{Dua:2017} using signature level $9$ and $10$}
\label{robustdataset-1}
\end{figure}
\clearpage

\begin{figure}[t]
\begin{subfigure}{.5\textwidth}
\centering
\includegraphics[trim={4cm 10cm 3cm 10cm},clip,width=0.9\linewidth]{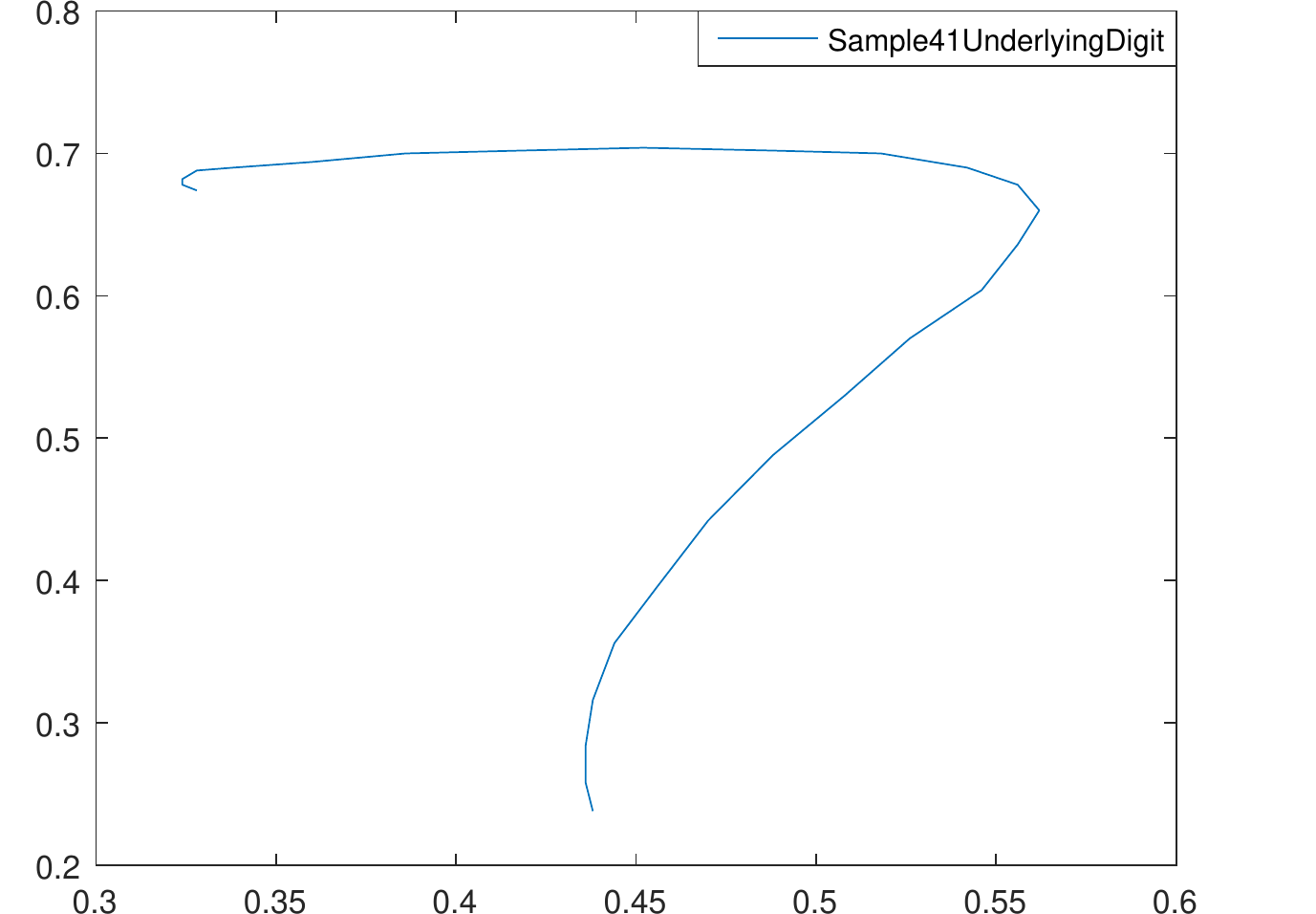}
\caption{Sample 41, underlying digit}
\end{subfigure}%
\begin{subfigure}{.5\textwidth}
\centering
\includegraphics[trim={4cm 10cm 3cm 10cm},clip,width=0.9\linewidth]{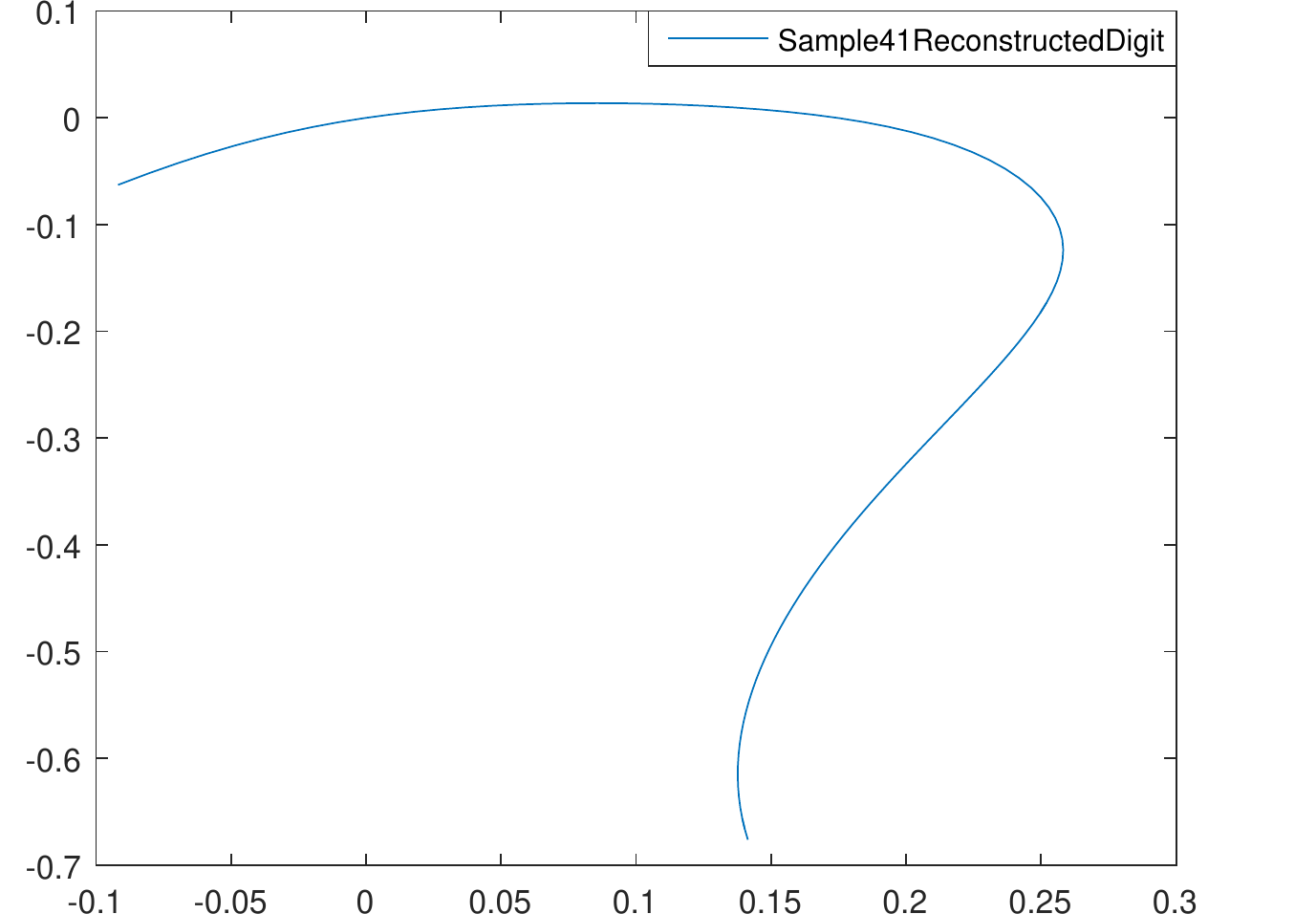}
\caption{Sample 41, reconstructed digit}
\end{subfigure}

\begin{subfigure}{.5\textwidth}
\centering
\includegraphics[trim={4cm 10cm 3cm 10cm},clip,width=0.9\linewidth]{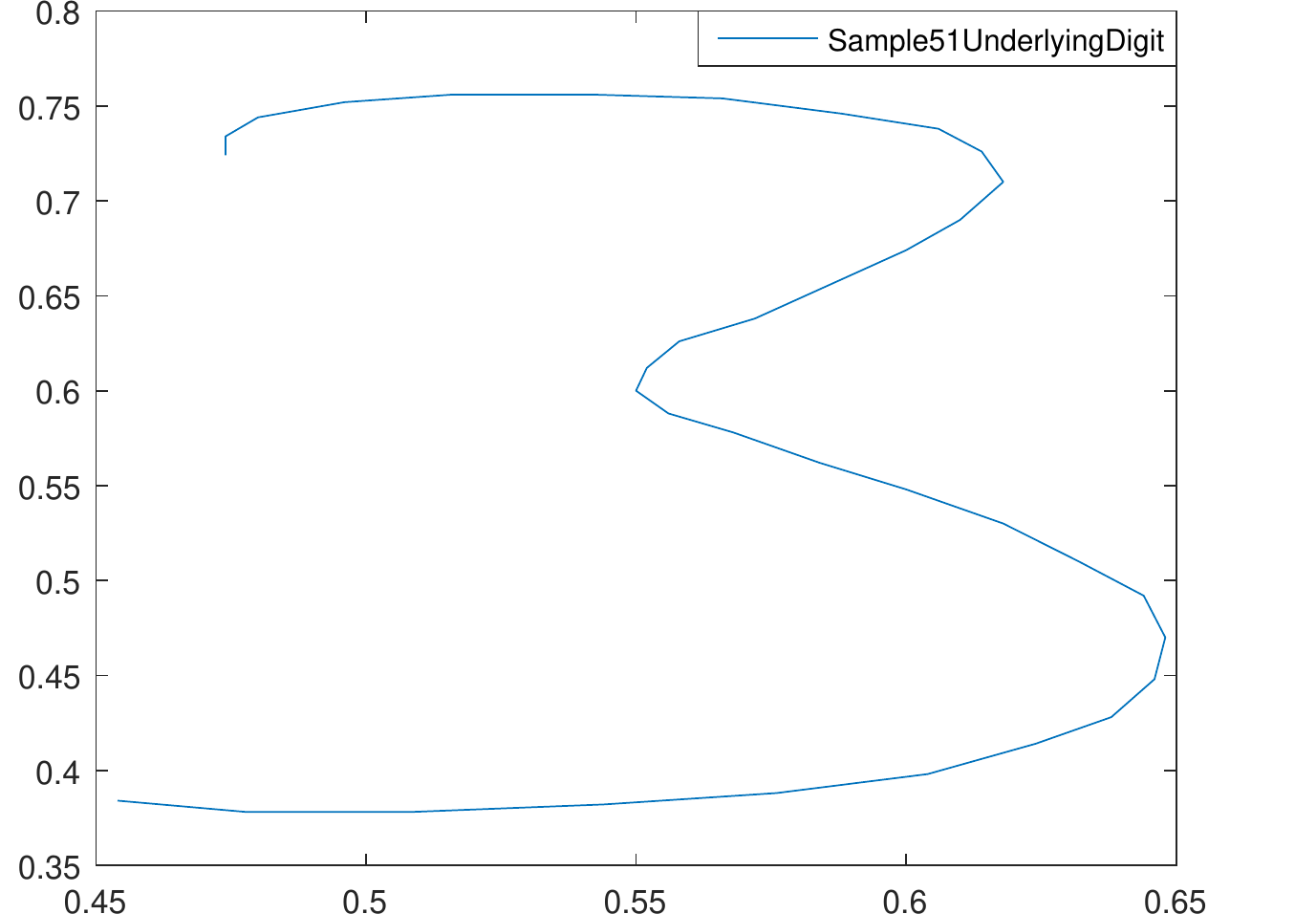}
\caption{Sample 51, underlying digit}
\end{subfigure}%
\begin{subfigure}{.5\textwidth}
\centering
\includegraphics[trim={4cm 10cm 3cm 10cm},clip,width=0.9\linewidth]{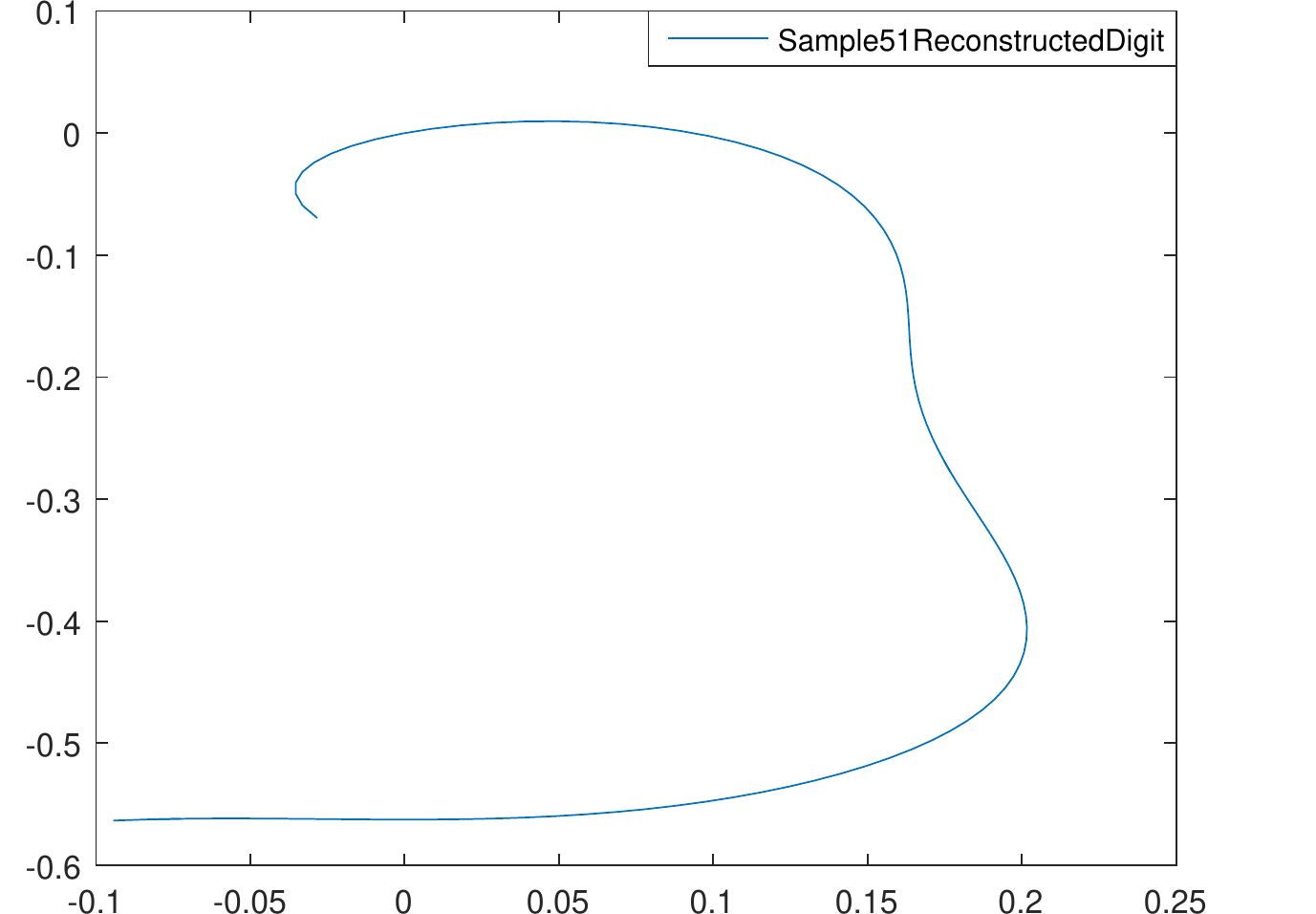}
\caption{Sample 51, reconstructed digit}
\end{subfigure}

\begin{subfigure}{.5\textwidth}
\centering
\includegraphics[trim={4cm 10cm 3cm 10cm},clip,width=0.9\linewidth]{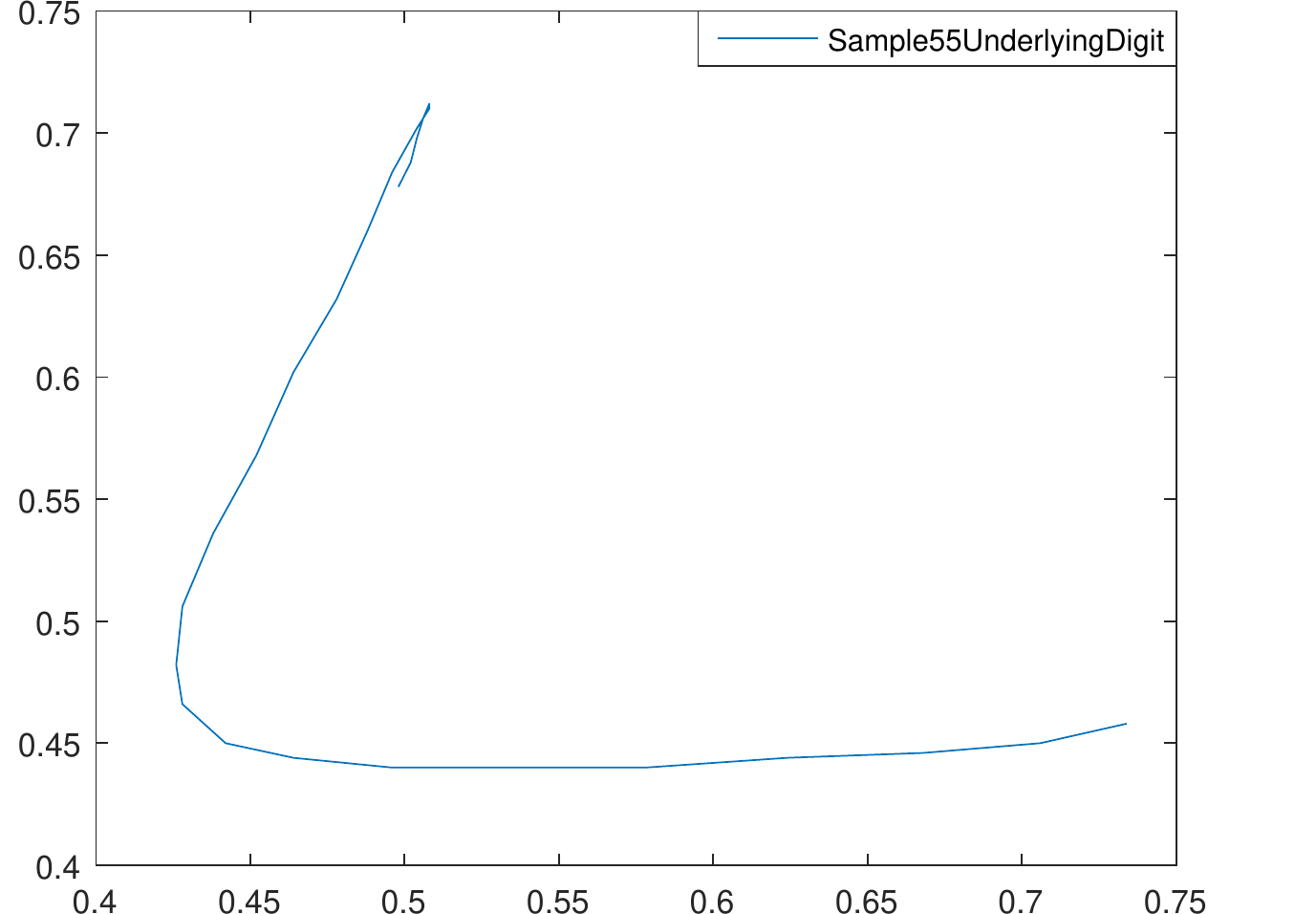}
\caption{Sample 55, underlying digit}
\end{subfigure}%
\begin{subfigure}{.5\textwidth}
\centering
\includegraphics[trim={4cm 10cm 3cm 10cm},clip,width=0.9\linewidth]{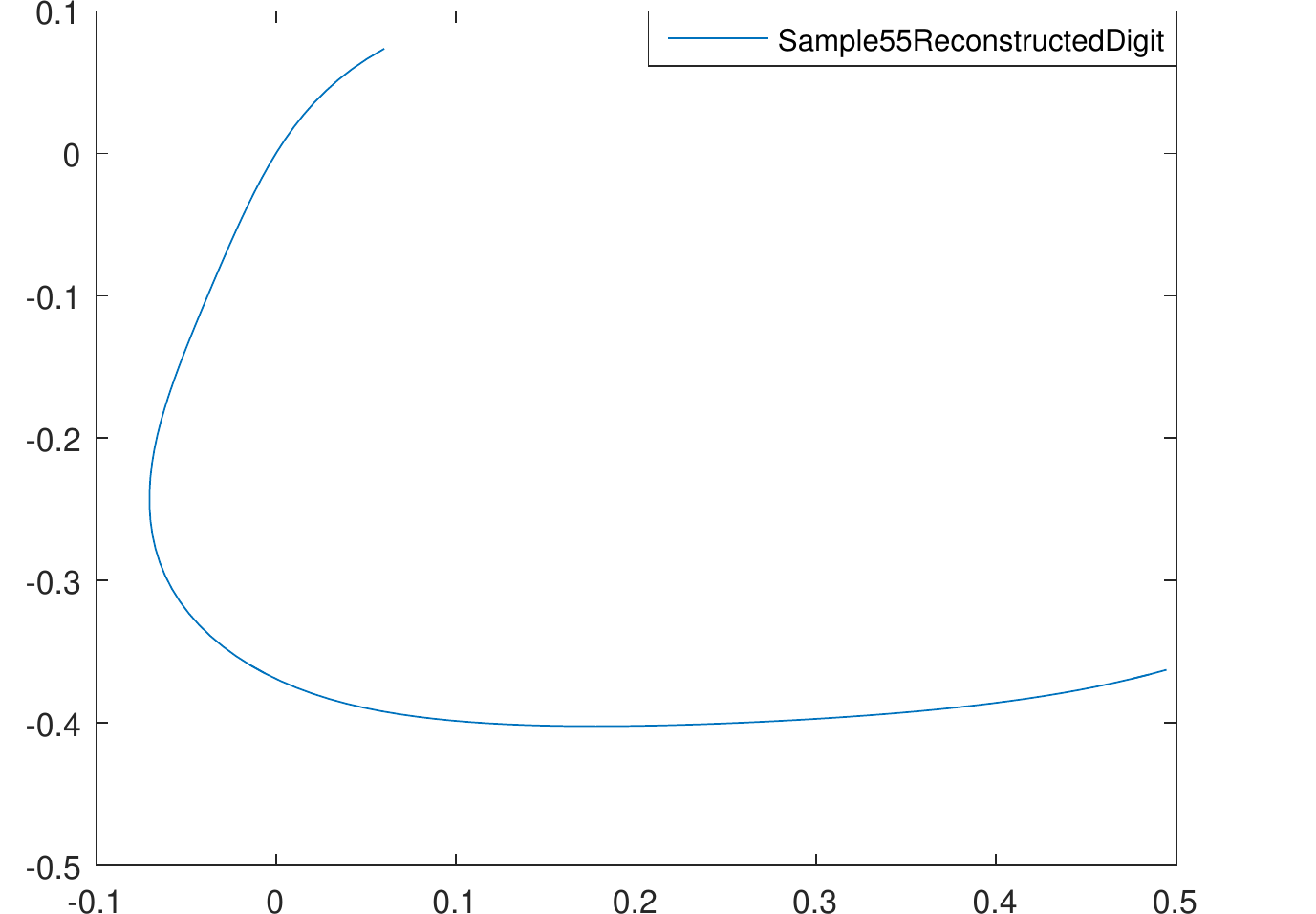}
\caption{Sample 55, reconstructed digit}
\end{subfigure}

\begin{subfigure}{.5\textwidth}
\centering
\includegraphics[trim={4cm 10cm 3cm 10cm},clip,width=0.9\linewidth]{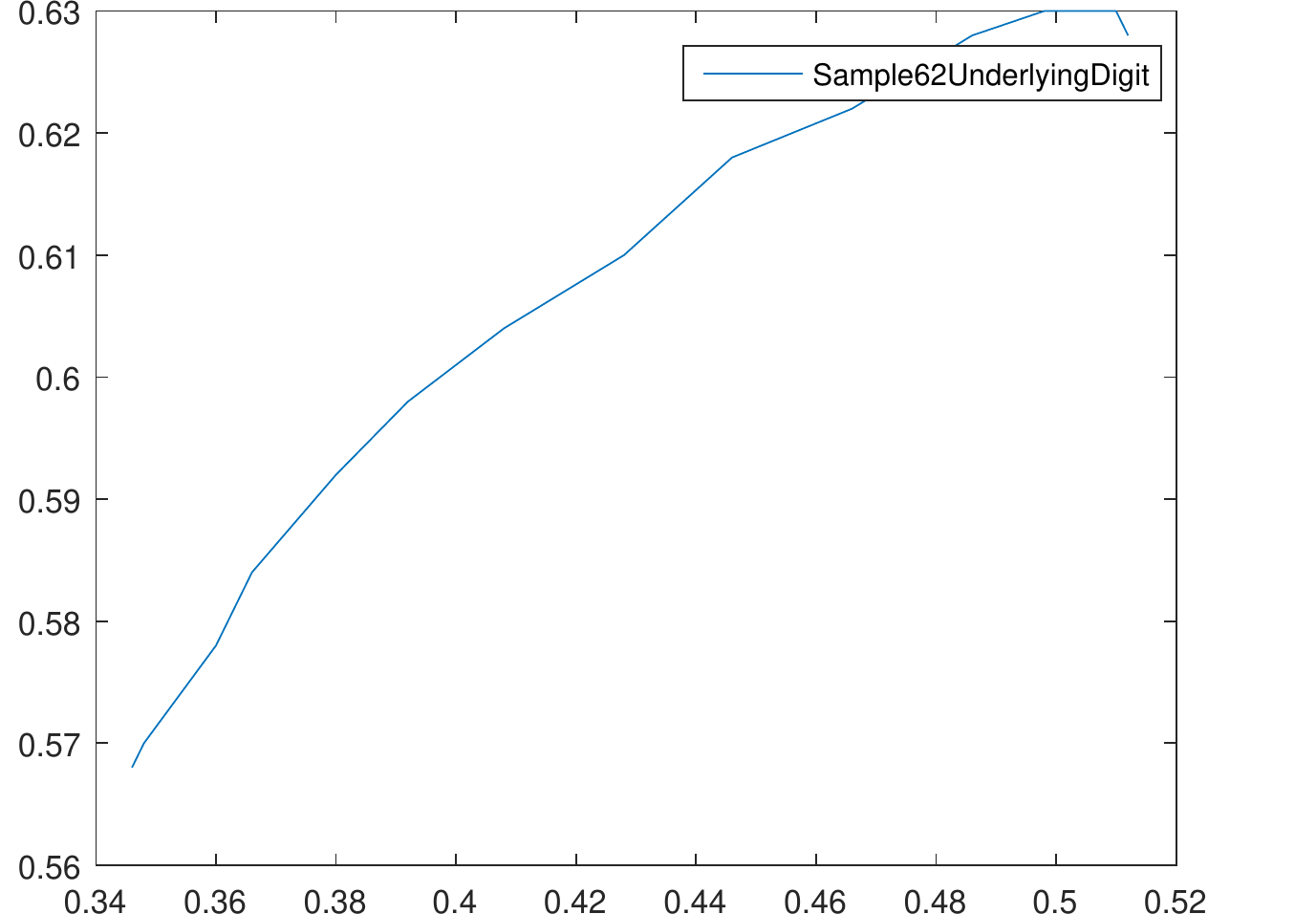}
\caption{Sample 62, underlying digit}
\end{subfigure}%
\begin{subfigure}{.5\textwidth}
\centering
\includegraphics[trim={4cm 10cm 3cm 10cm},clip,width=0.9\linewidth]{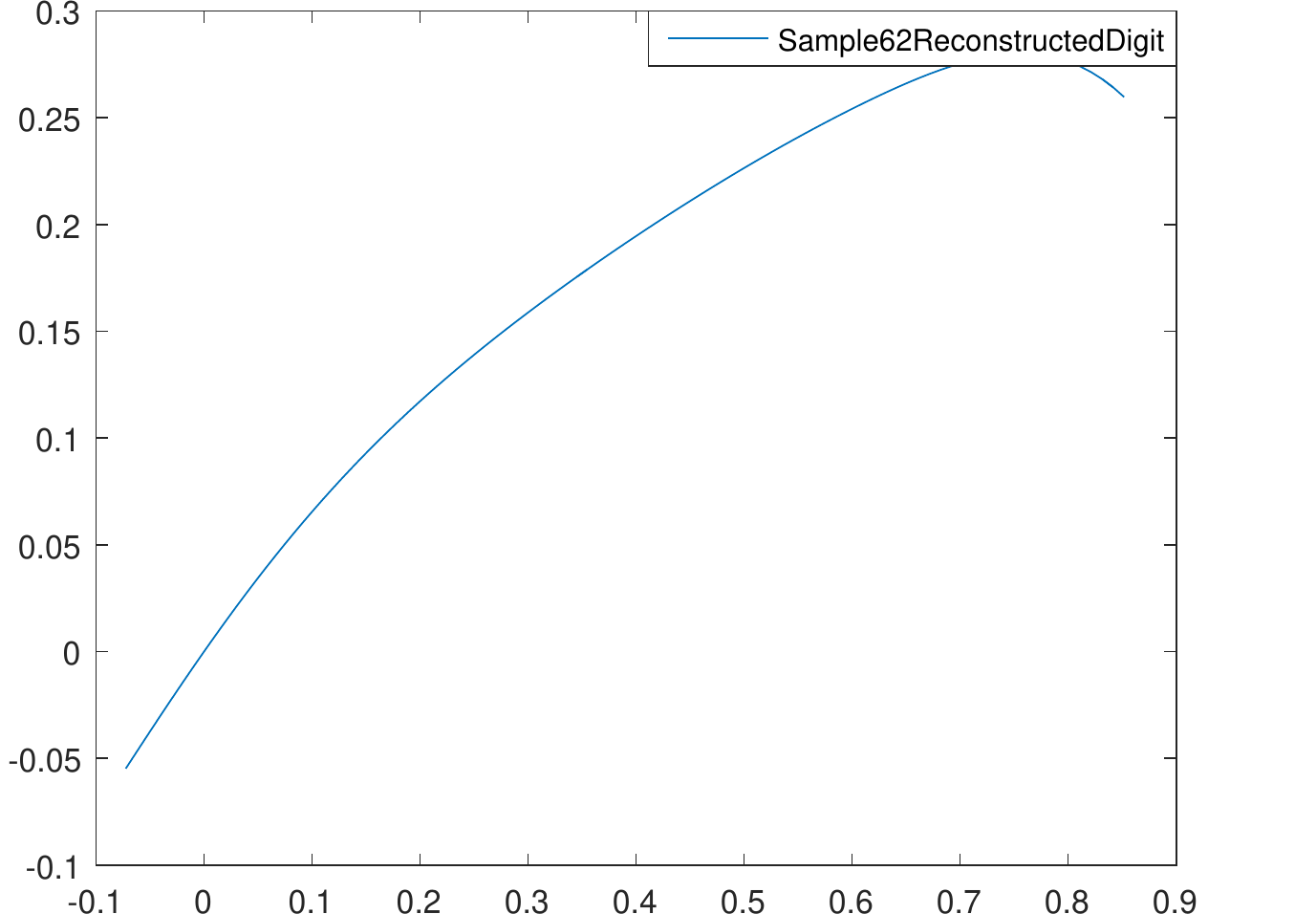}
\caption{Sample 62, reconstructed digit}
\end{subfigure}
\caption{Reconstruction of digits from the data set \cite{Dua:2017} using signature level $9$ and $10$}
\label{robustdataset-2}
\end{figure}
\clearpage

\begin{figure}[t]
\begin{subfigure}{.5\textwidth}
\centering
\includegraphics[trim={4cm 10cm 3cm 10cm},clip,width=0.9\linewidth]{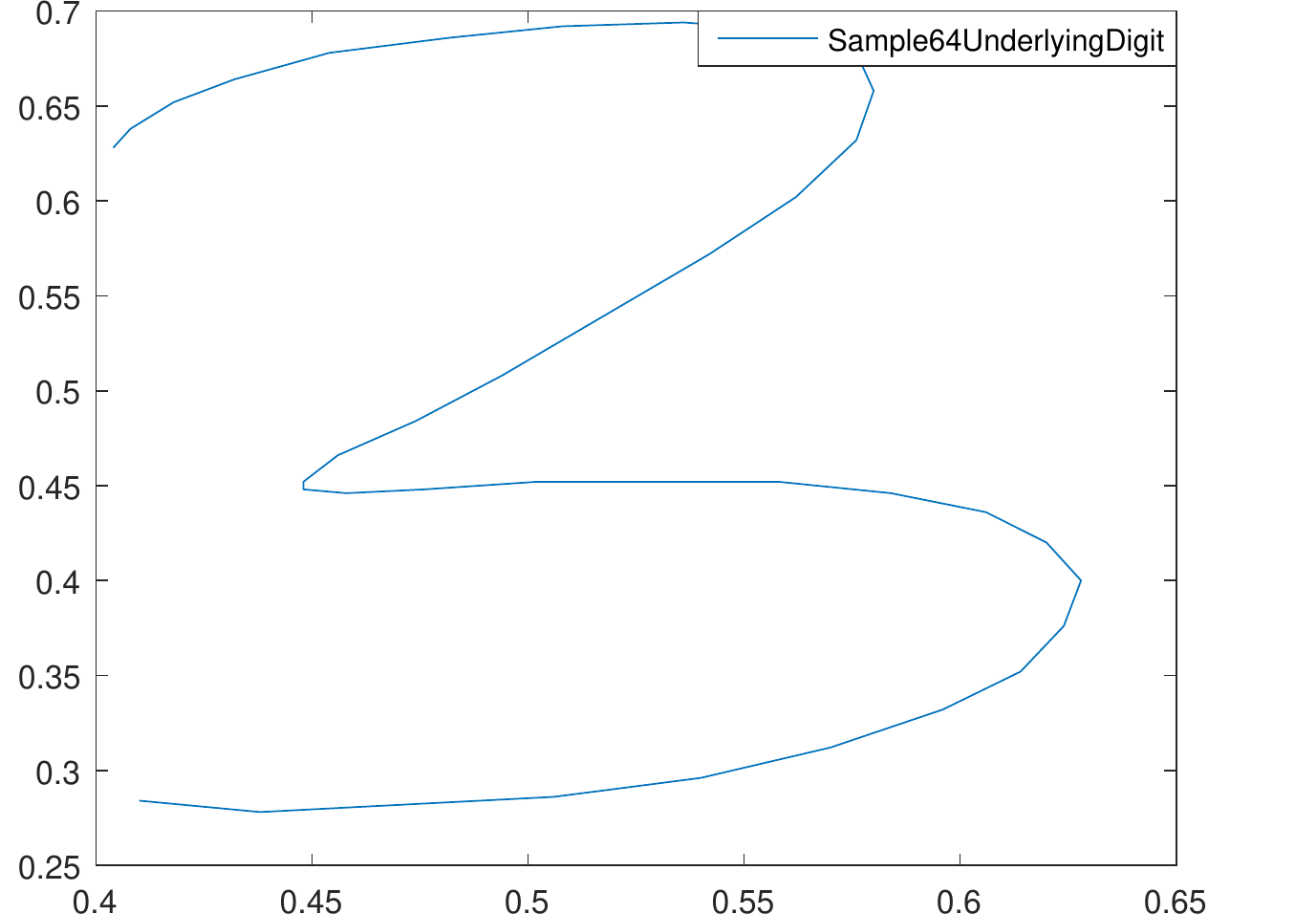}
\caption{Sample 64, underlying digit}
\end{subfigure}%
\begin{subfigure}{.5\textwidth}
\centering
\includegraphics[trim={4cm 10cm 3cm 10cm},clip,width=0.9\linewidth]{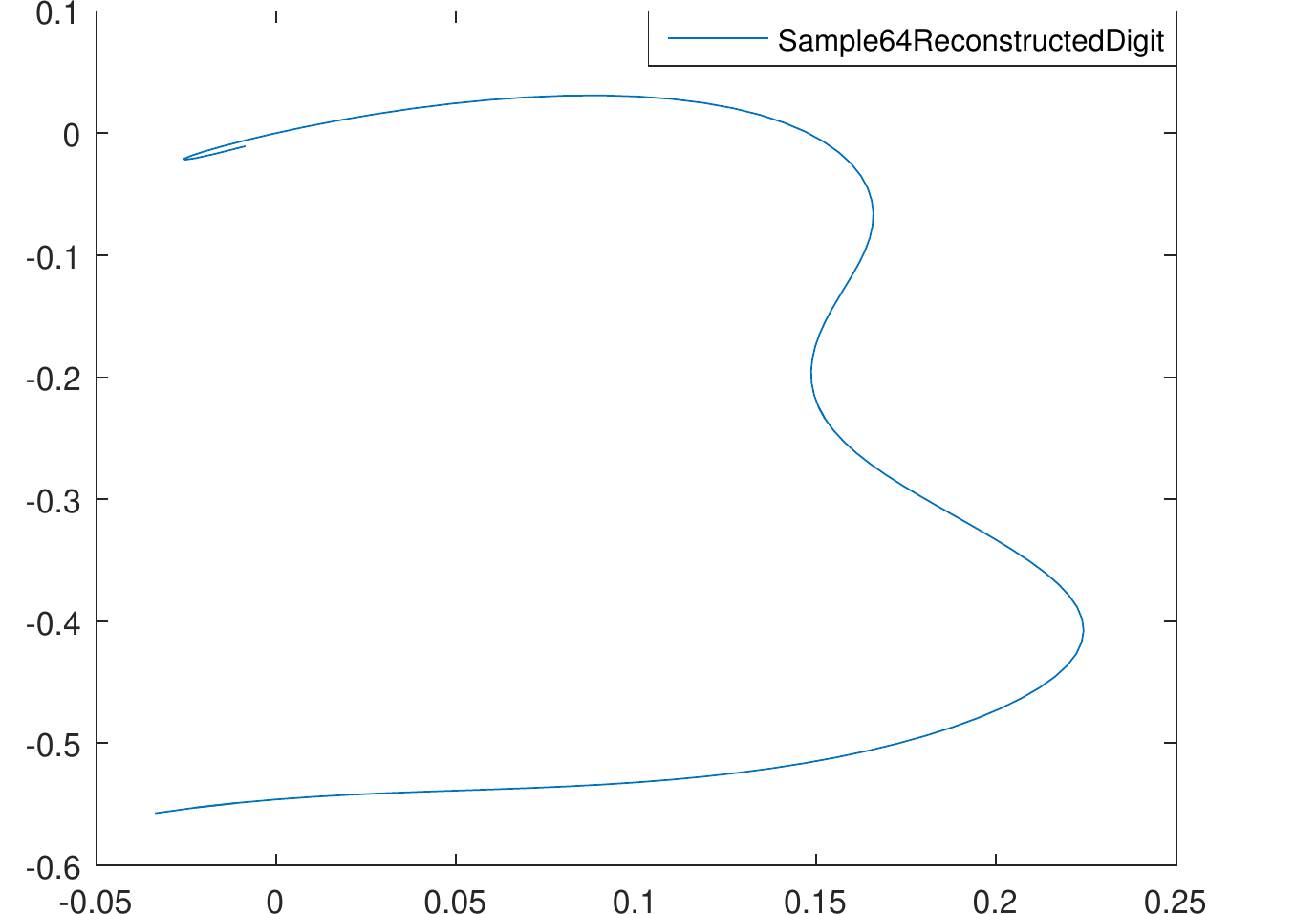}
\caption{Sample 64, reconstructed digit}
\end{subfigure}

\begin{subfigure}{.5\textwidth}
\centering
\includegraphics[trim={4cm 10cm 3cm 10cm},clip,width=0.9\linewidth]{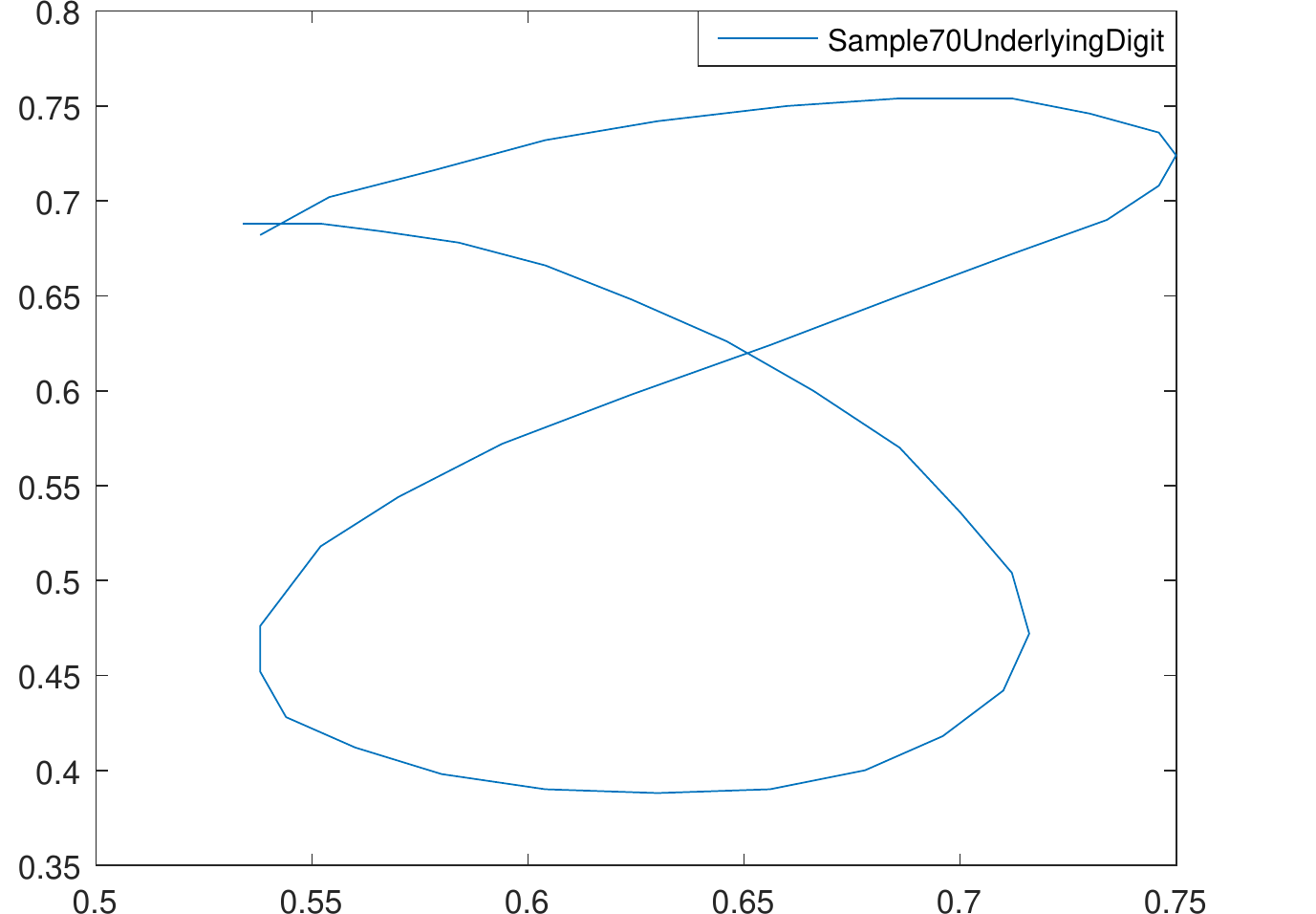}
\caption{Sample 70, underlying digit}
\end{subfigure}%
\begin{subfigure}{.5\textwidth}
\centering
\includegraphics[trim={4cm 10cm 3cm 10cm},clip,width=0.9\linewidth]{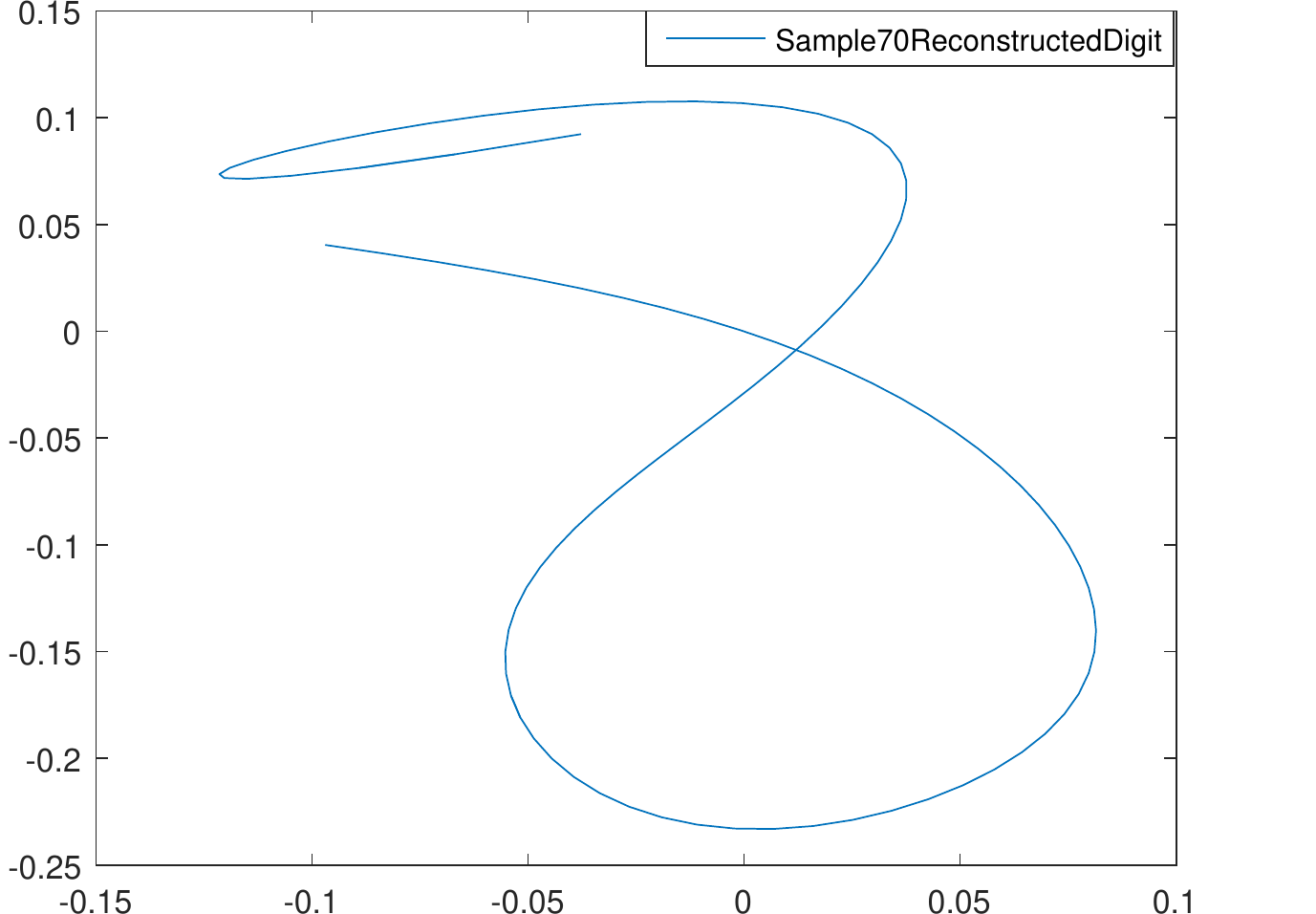}
\caption{Sample 70, reconstructed digit}
\end{subfigure}

\begin{subfigure}{.5\textwidth}
\centering
\includegraphics[trim={4cm 10cm 3cm 10cm},clip,width=0.9\linewidth]{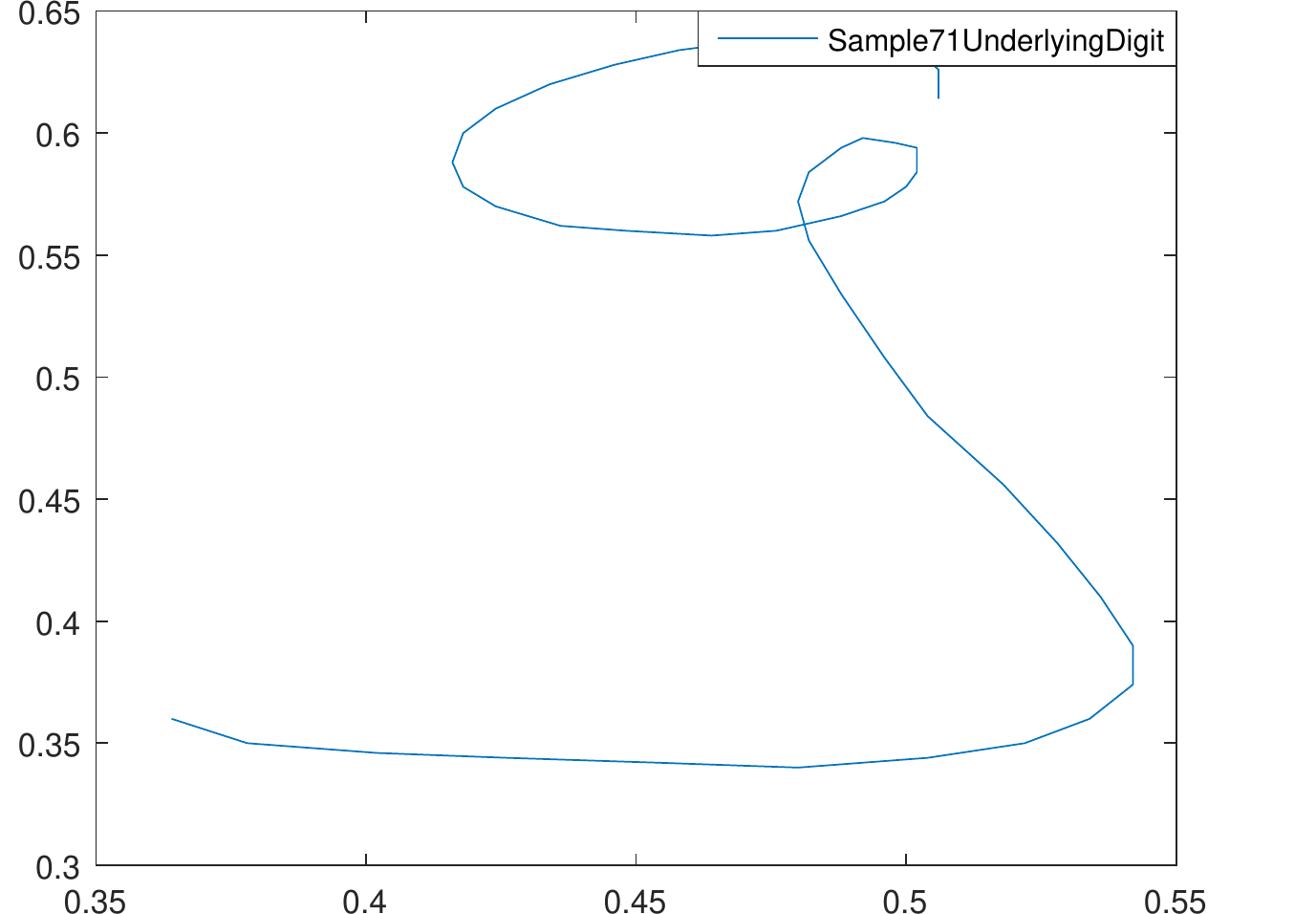}
\caption{Sample 71, underlying digit}
\end{subfigure}%
\begin{subfigure}{.5\textwidth}
\centering
\includegraphics[trim={4cm 10cm 3cm 10cm},clip,width=0.9\linewidth]{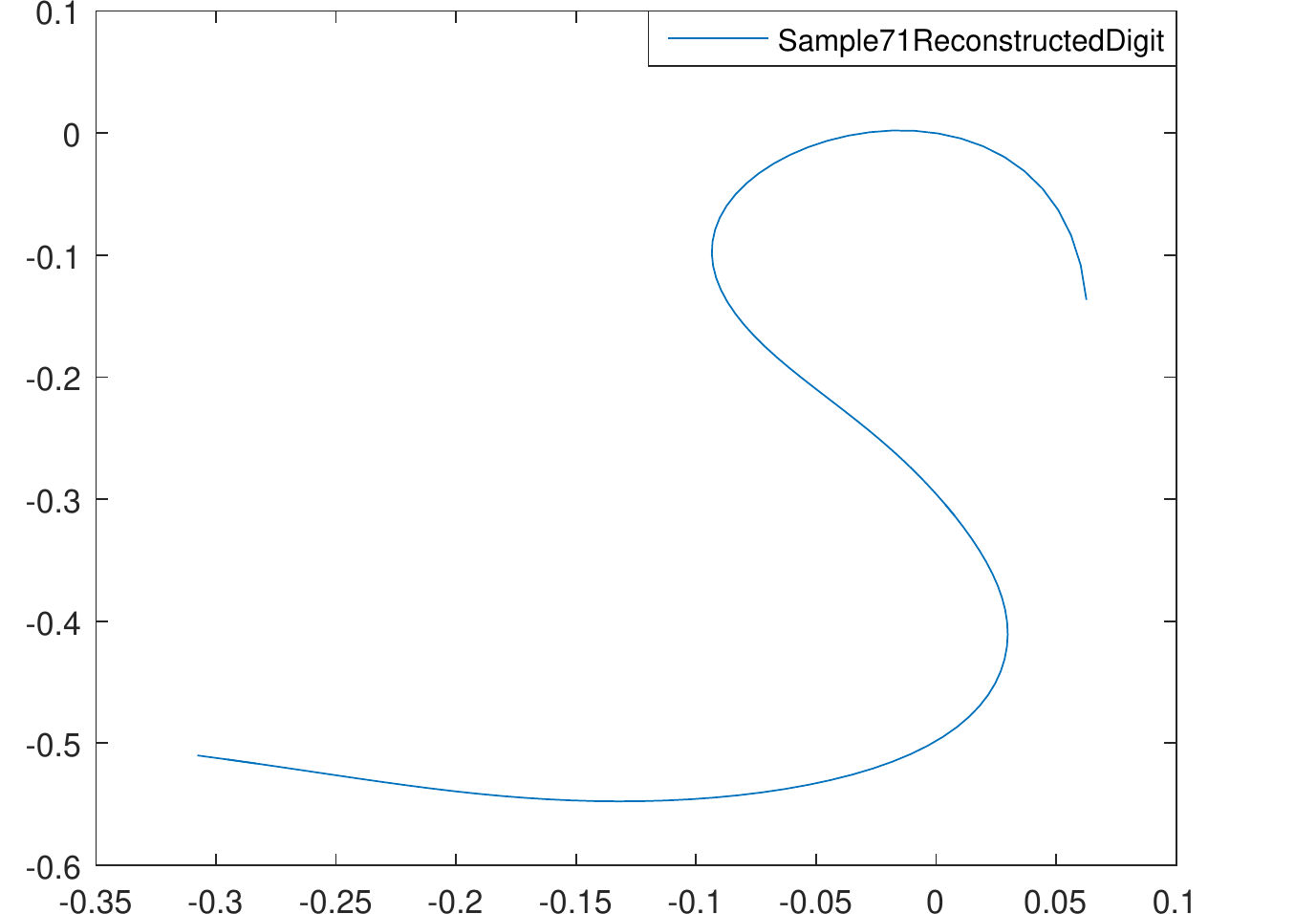}
\caption{Sample 71, reconstructed digit}
\end{subfigure}

\begin{subfigure}{.5\textwidth}
\centering
\includegraphics[trim={4cm 10cm 3cm 10cm},clip,width=0.9\linewidth]{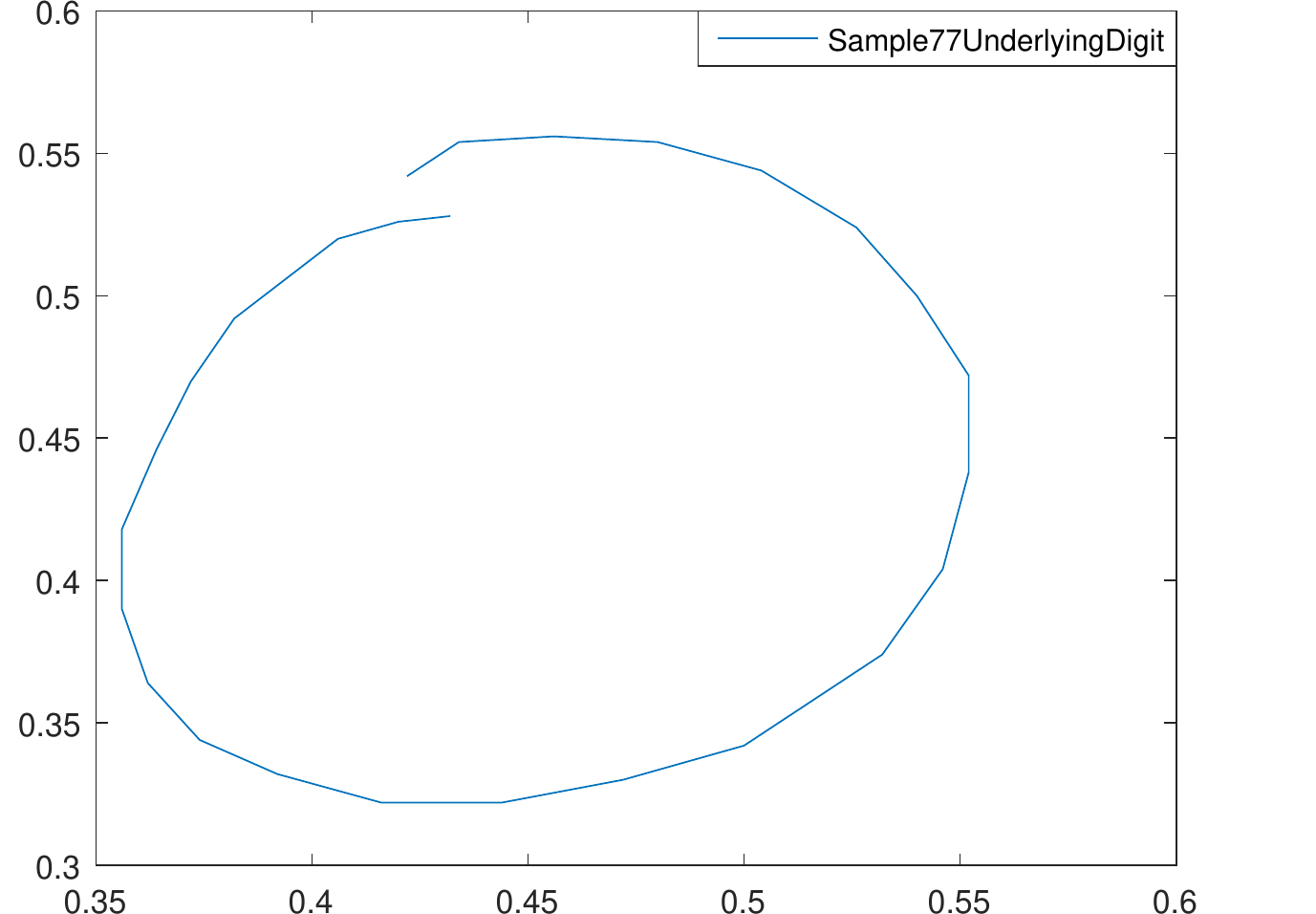}
\caption{Sample 77, underlying digit}
\end{subfigure}%
\begin{subfigure}{.5\textwidth}
\centering
\includegraphics[trim={4cm 10cm 3cm 10cm},clip,width=0.9\linewidth]{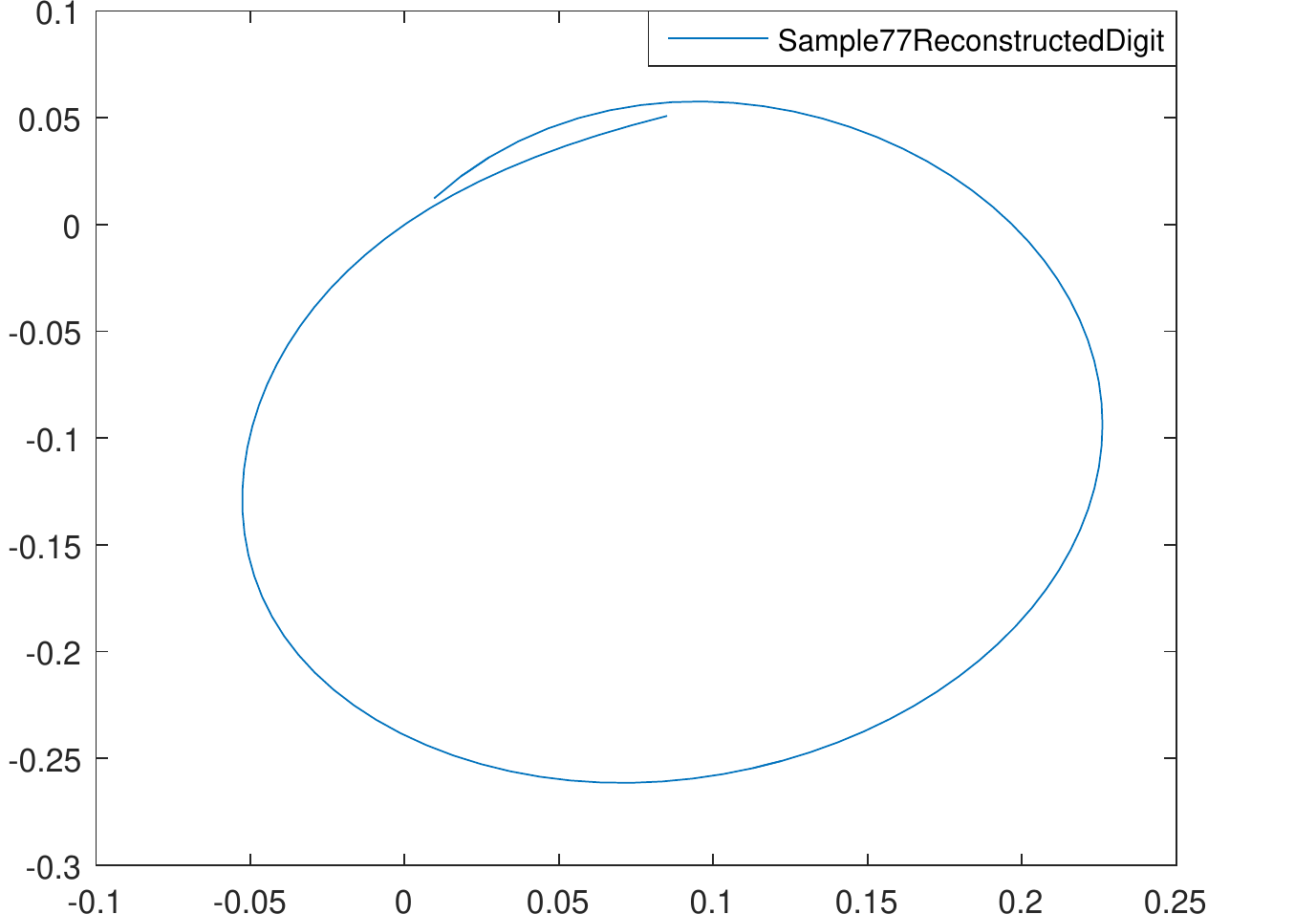}
\caption{Sample 77, reconstructed digit}
\end{subfigure}
\caption{Reconstruction of digits from the data set \cite{Dua:2017} using signature level $9$ and $10$}
\label{robustdataset-3}
\end{figure}
\clearpage

\begin{figure}[t]
\begin{subfigure}{.5\textwidth}
\centering
\includegraphics[trim={4cm 10cm 3cm 10cm},clip,width=0.9\linewidth]{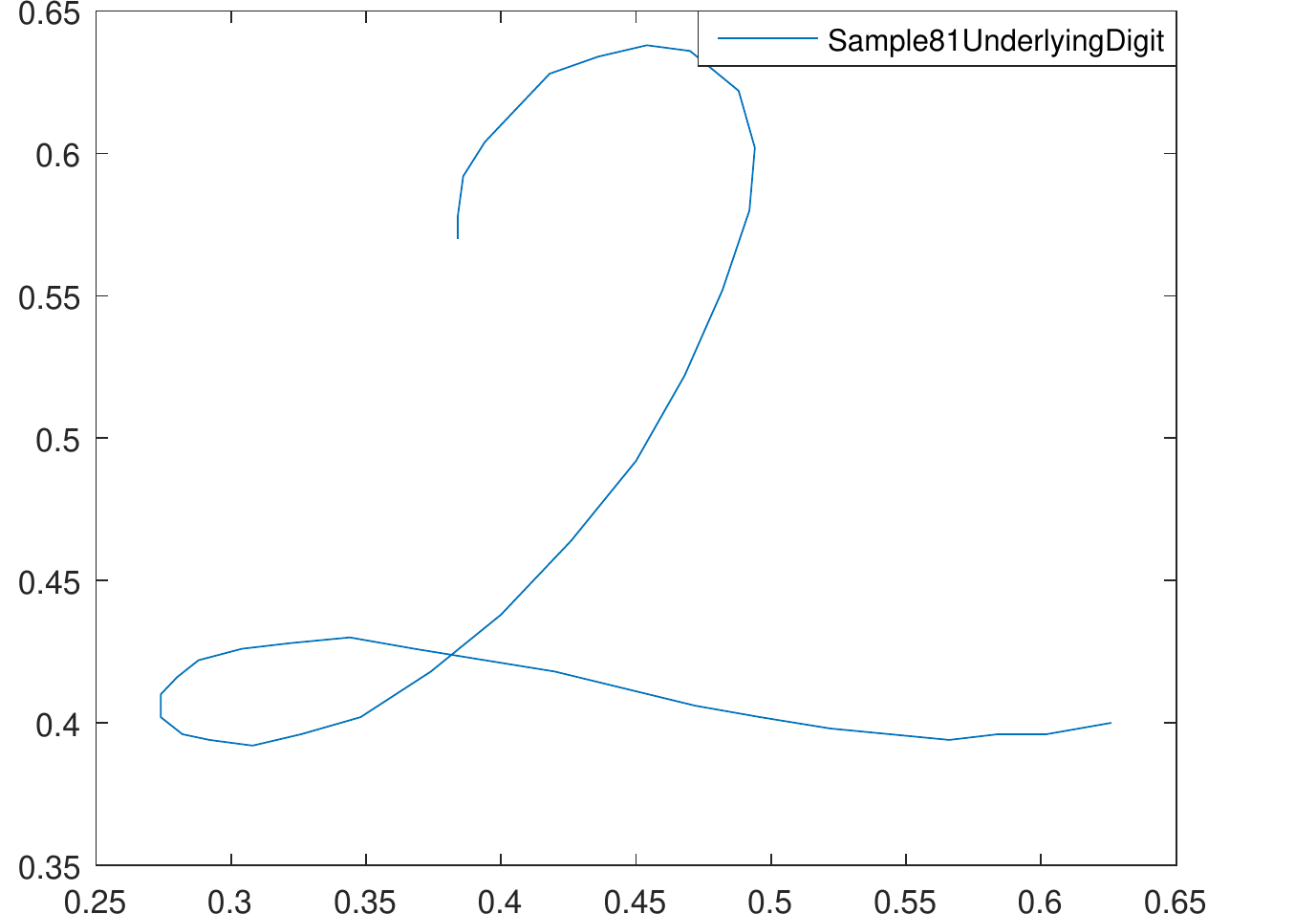}
\caption{Sample 81, underlying digit}
\end{subfigure}%
\begin{subfigure}{.5\textwidth}
\centering
\includegraphics[trim={4cm 10cm 3cm 10cm},clip,width=0.9\linewidth]{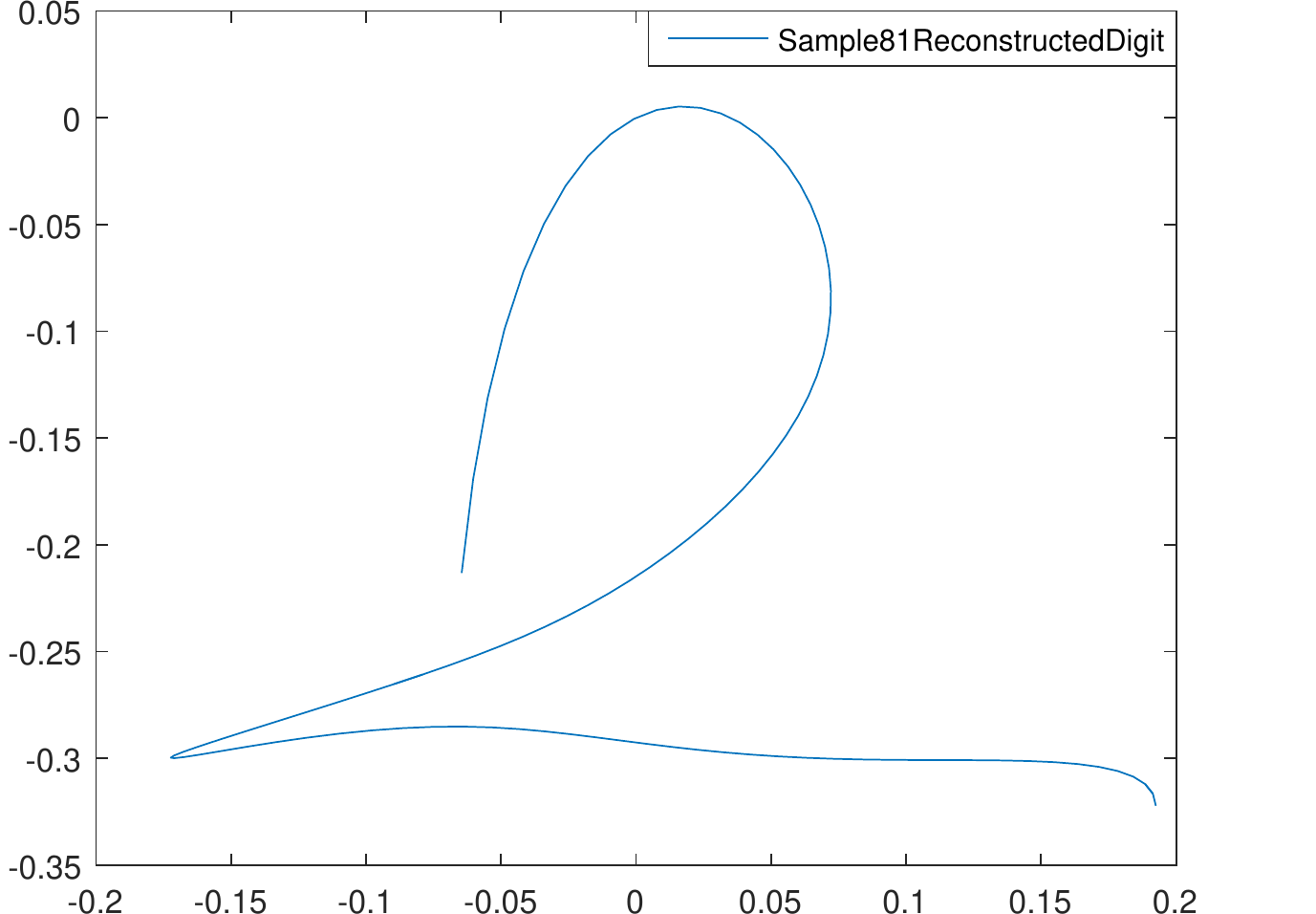}
\caption{Sample 81, reconstructed digit}
\end{subfigure}

\begin{subfigure}{.5\textwidth}
\centering
\includegraphics[trim={4cm 10cm 3cm 10cm},clip,width=0.9\linewidth]{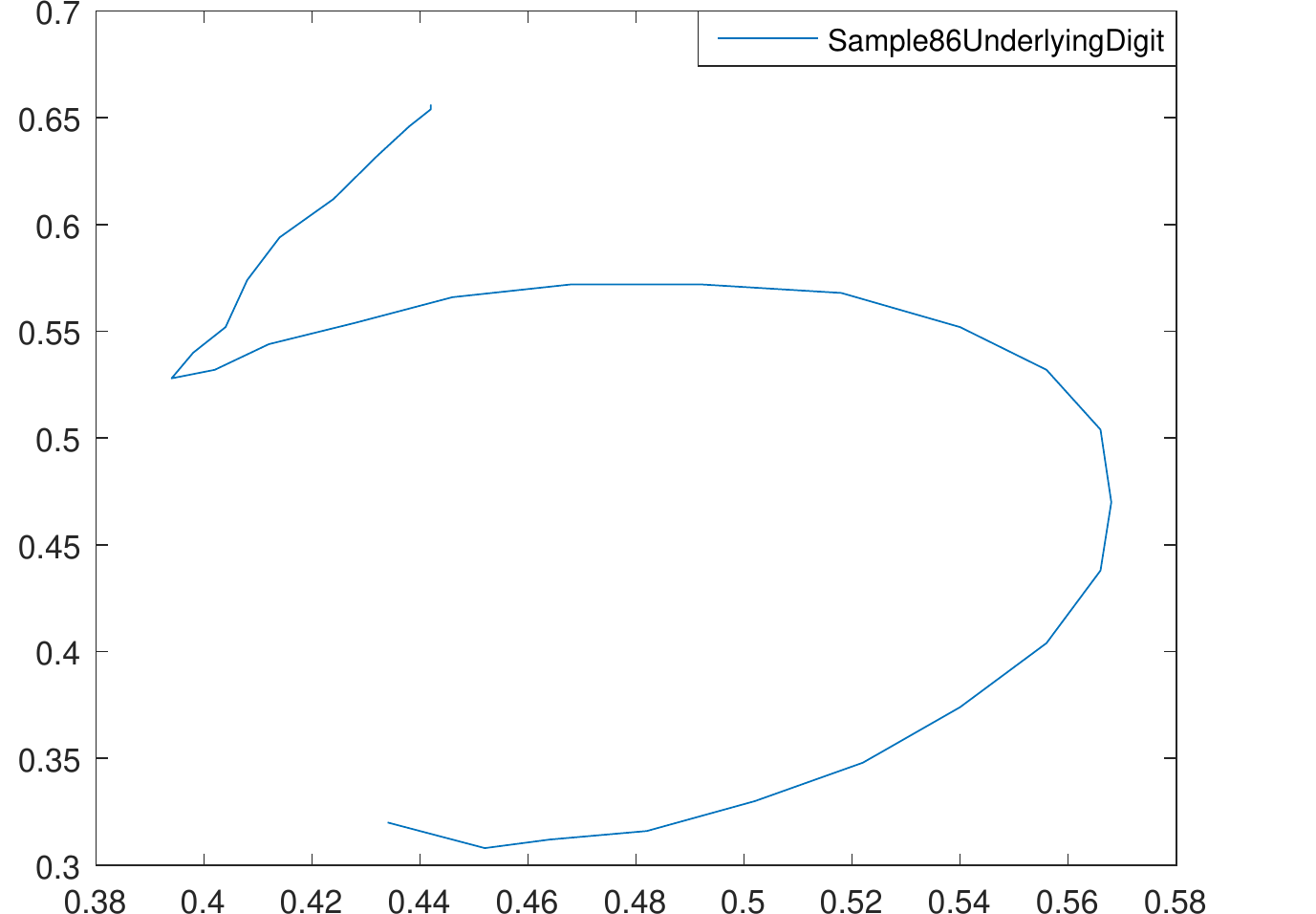}
\caption{Sample 86, underlying digit}
\end{subfigure}%
\begin{subfigure}{.5\textwidth}
\centering
\includegraphics[trim={4cm 10cm 3cm 10cm},clip,width=0.9\linewidth]{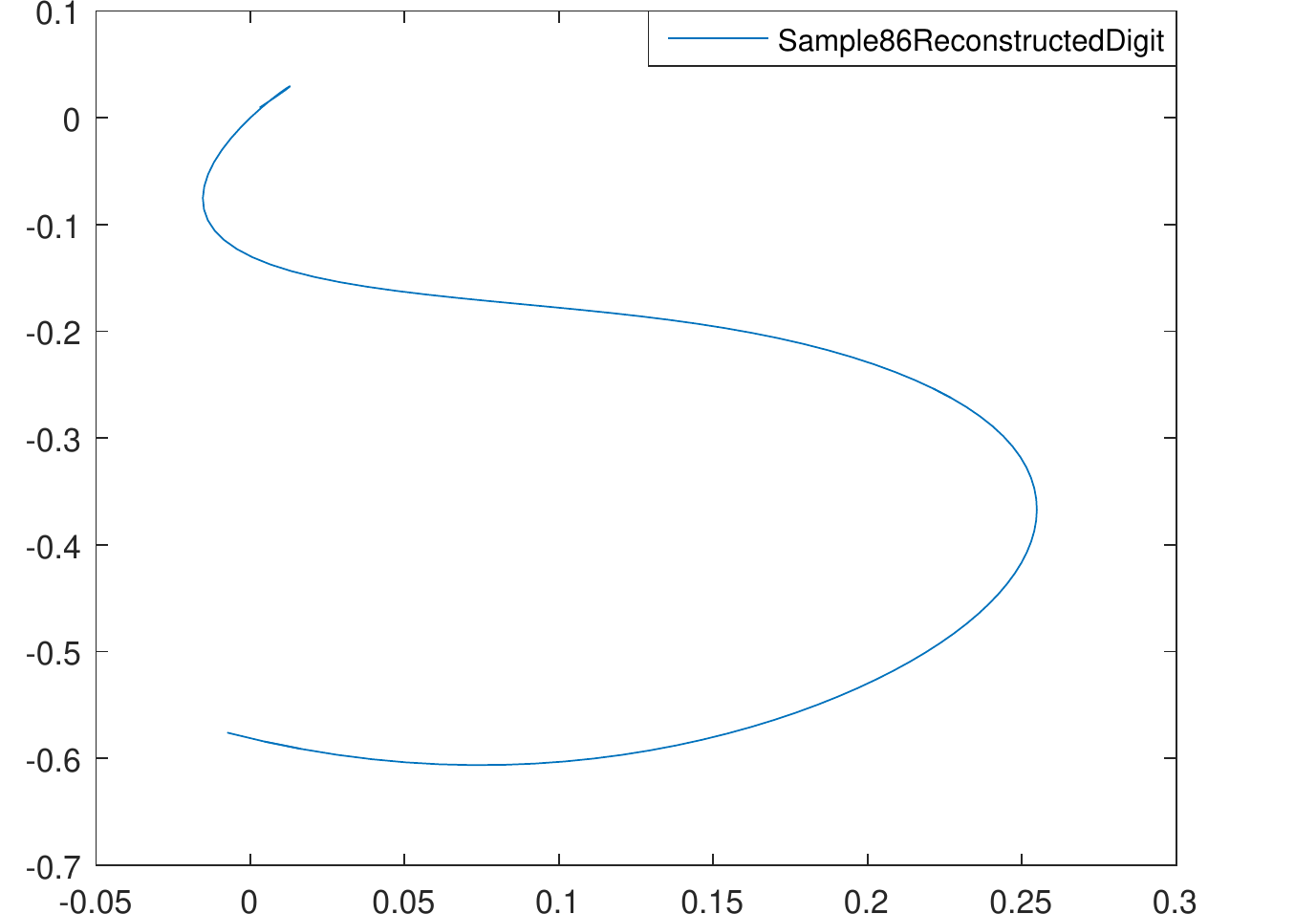}
\caption{Sample 86, reconstructed digit}
\end{subfigure}

\begin{subfigure}{.5\textwidth}
\centering
\includegraphics[trim={4cm 10cm 3cm 10cm},clip,width=0.9\linewidth]{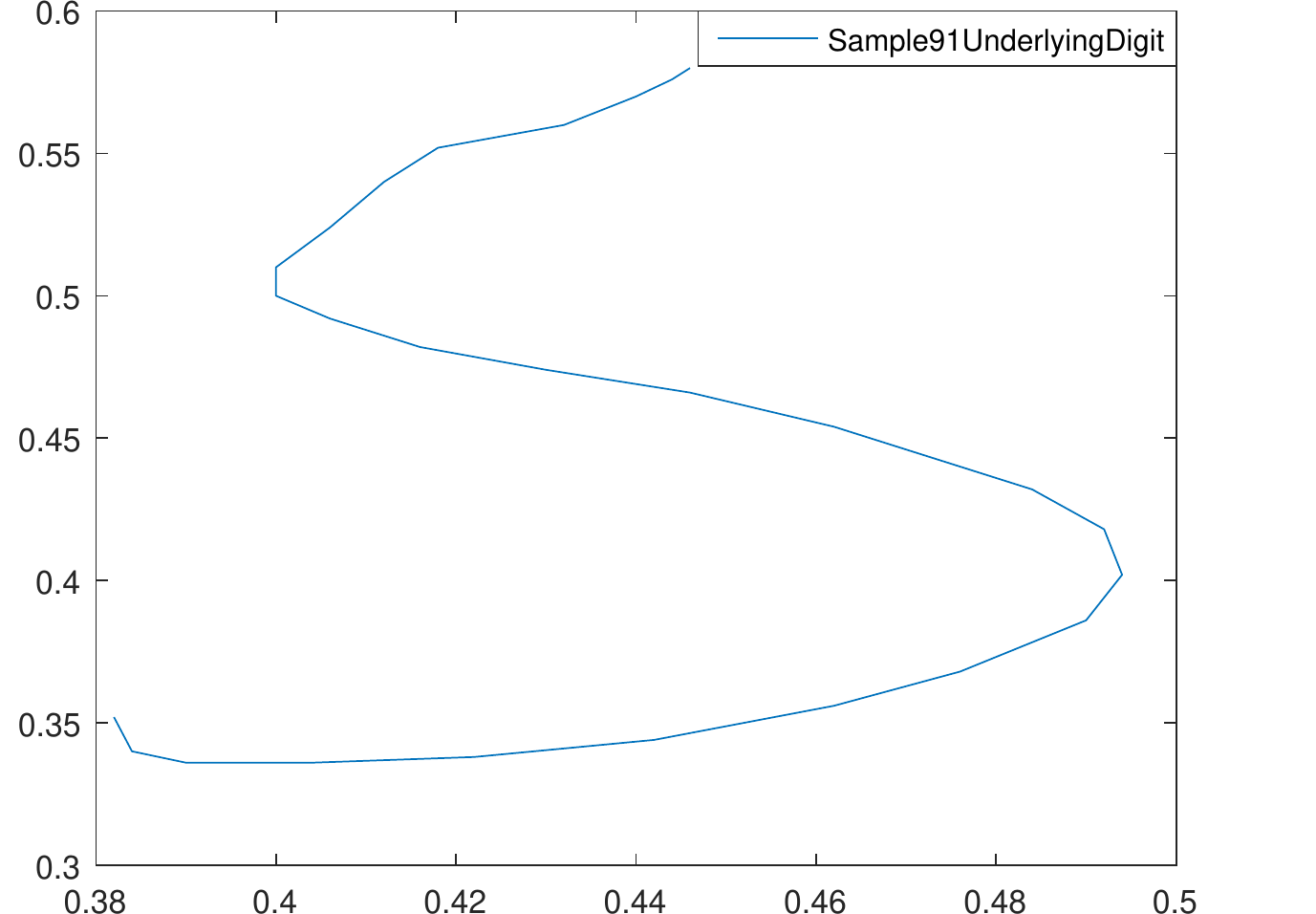}
\caption{Sample 91, underlying digit}
\end{subfigure}%
\begin{subfigure}{.5\textwidth}
\centering
\includegraphics[trim={4cm 10cm 3cm 10cm},clip,width=0.9\linewidth]{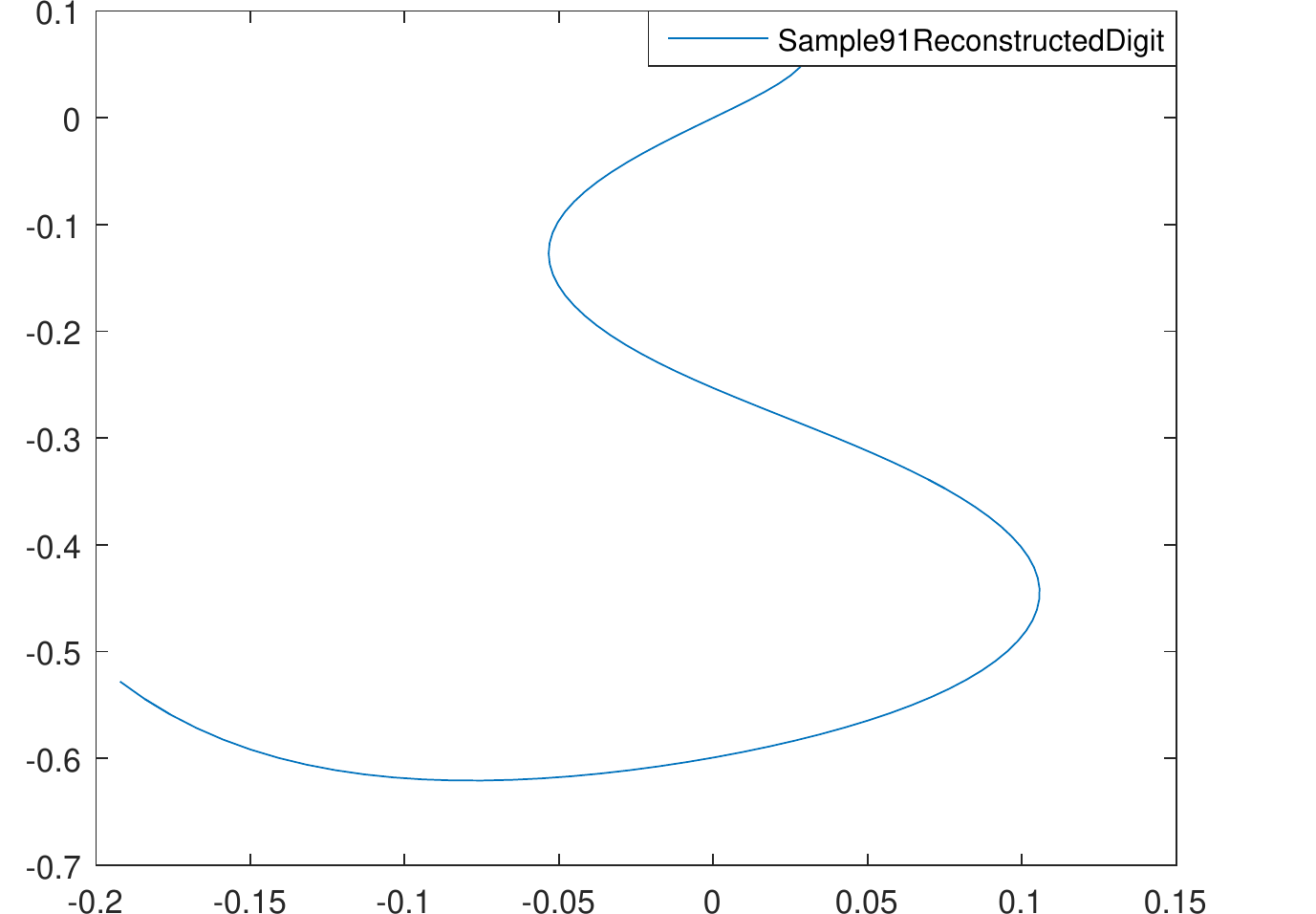}
\caption{Sample 91, reconstructed digit}
\end{subfigure}

\begin{subfigure}{.5\textwidth}
\centering
\includegraphics[trim={4cm 10cm 3cm 10cm},clip,width=0.9\linewidth]{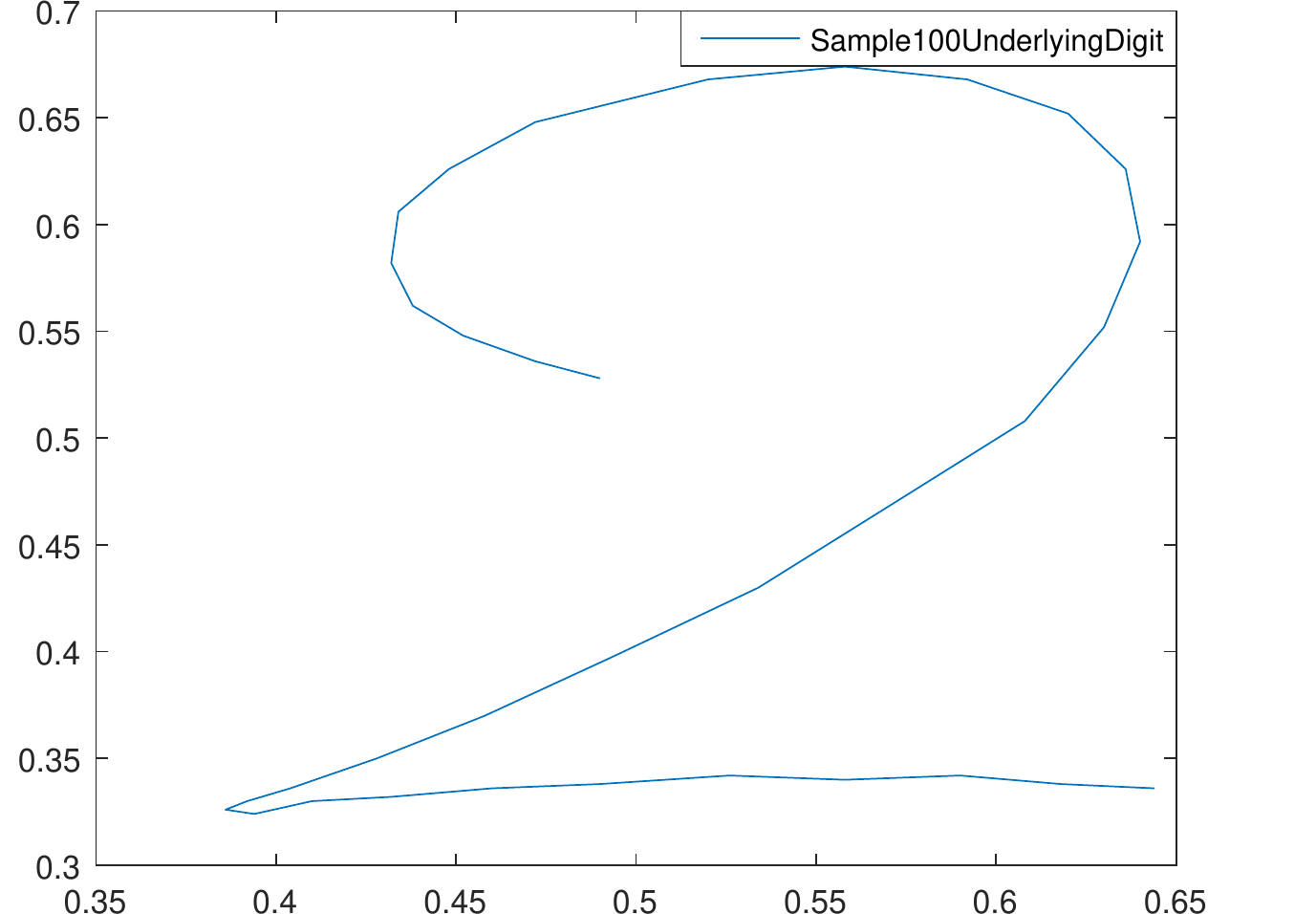}
\caption{Sample 100, underlying digit}
\end{subfigure}%
\begin{subfigure}{.5\textwidth}
\centering
\includegraphics[trim={4cm 10cm 4cm 10cm},clip,width=0.9\linewidth]{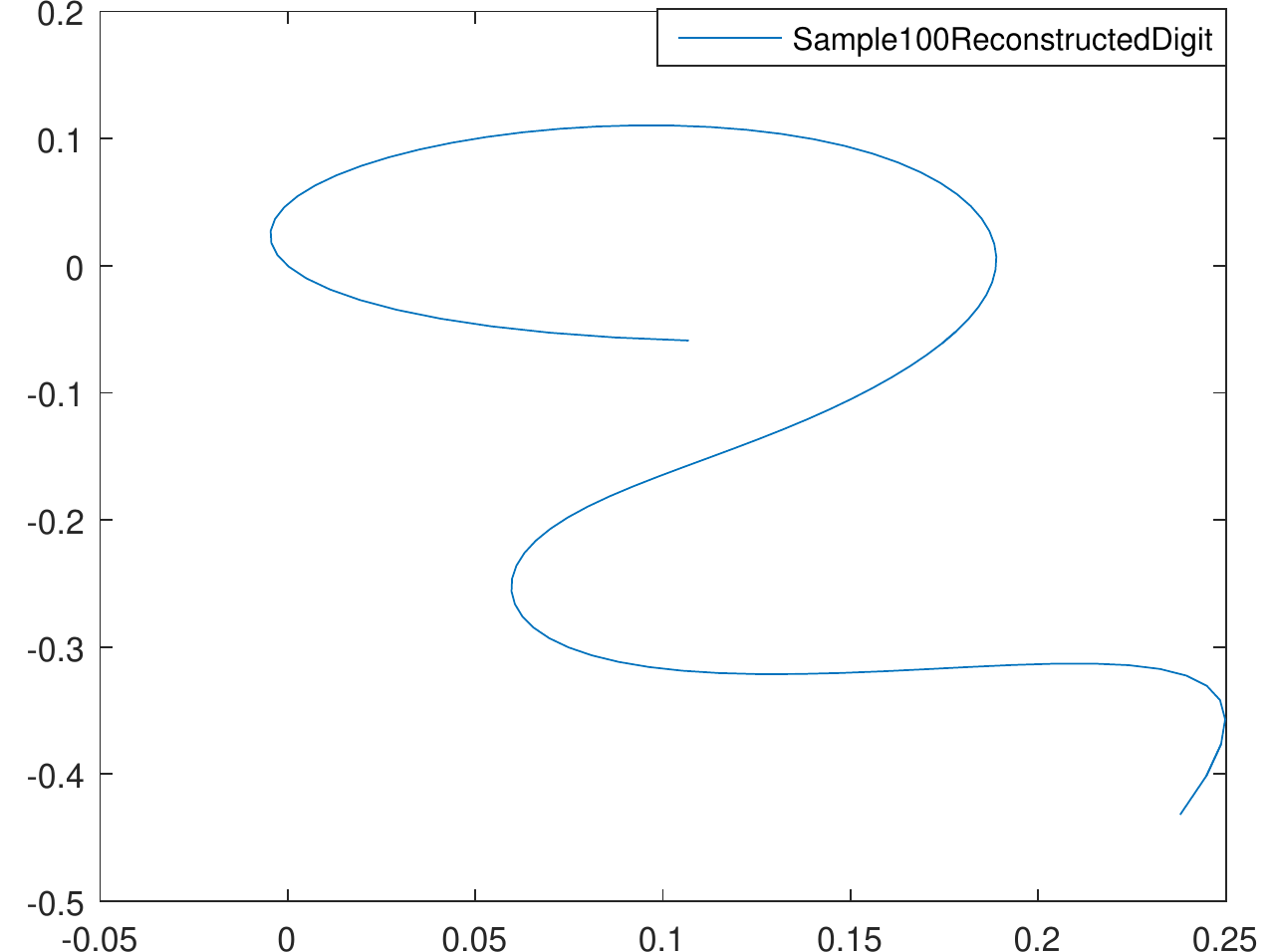}
\caption{Sample 100, reconstructed digit}
\end{subfigure}
\caption{Reconstruction of digits from the data set \cite{Dua:2017} using signature level $9$ and $10$}
\label{robustdataset-4}
\end{figure}
\clearpage

\begin{figure}[t]
\begin{subfigure}{.5\textwidth}
\centering
\includegraphics[trim={4cm 10cm 3cm 10cm},clip,width=0.9\linewidth]{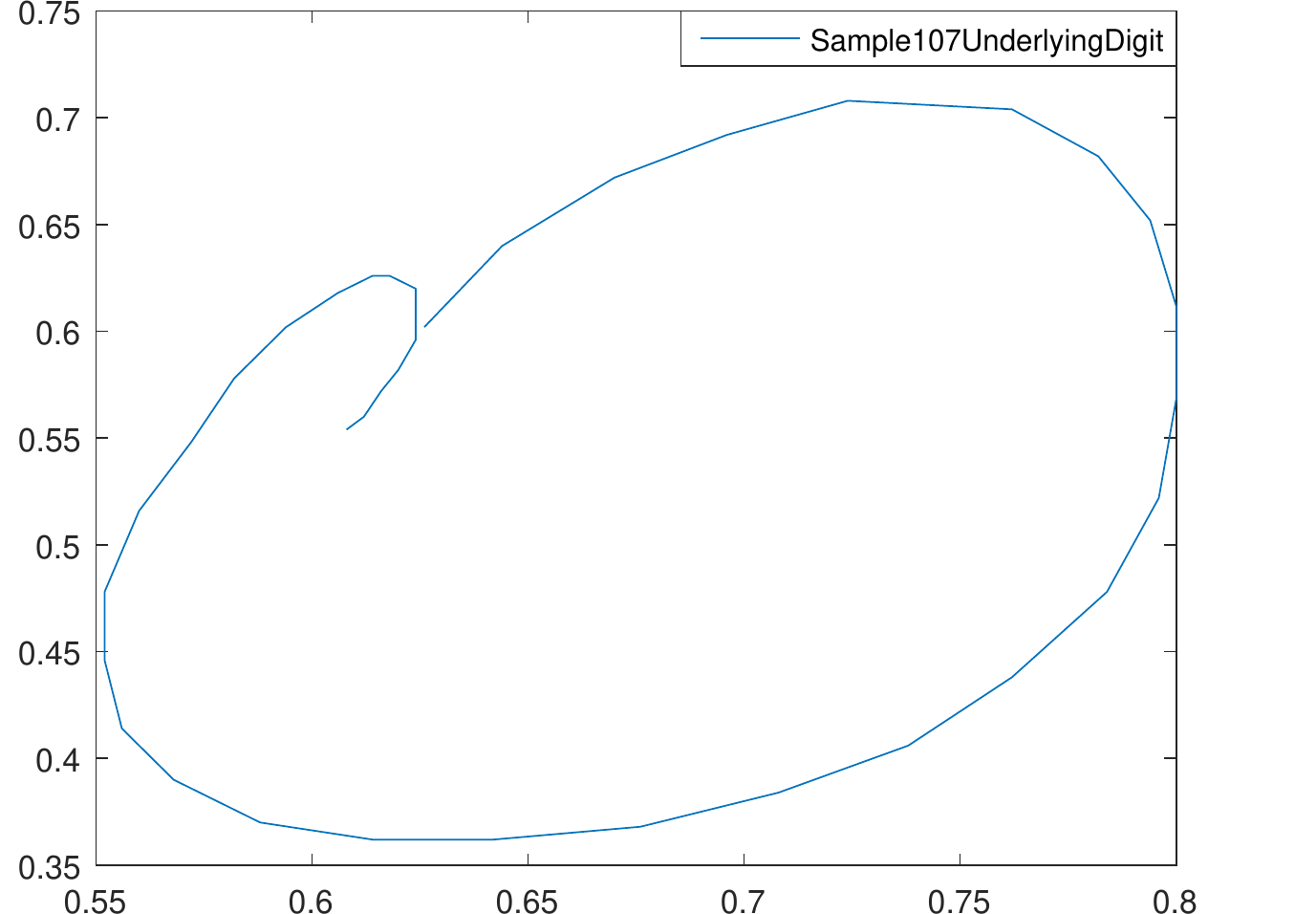}
\caption{Sample 107, underlying digit}
\end{subfigure}%
\begin{subfigure}{.5\textwidth}
\centering
\includegraphics[trim={4cm 10cm 3cm 10cm},clip,width=0.9\linewidth]{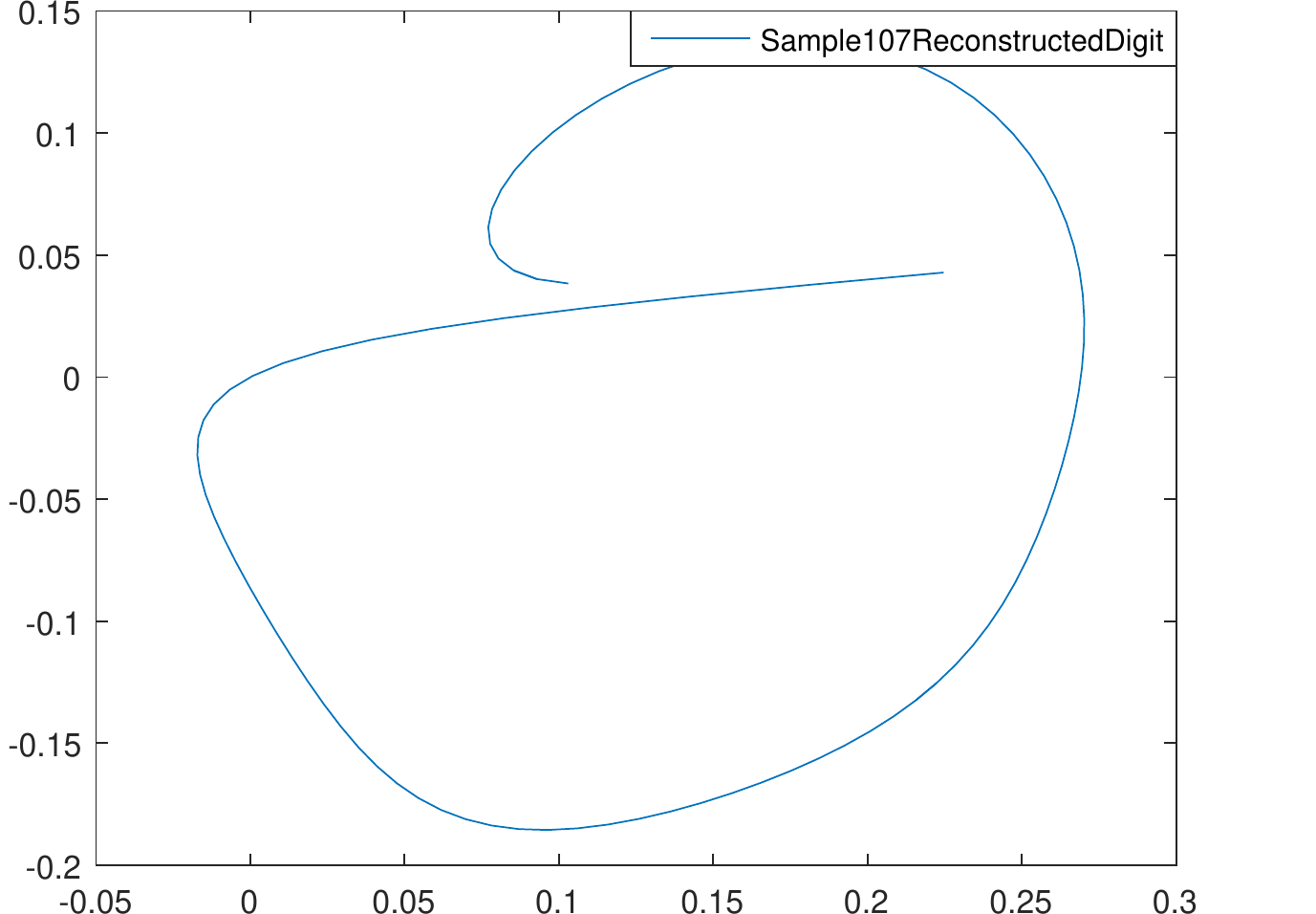}
\caption{Sample 107, reconstructed digit}
\end{subfigure}

\begin{subfigure}{.5\textwidth}
\centering
\includegraphics[trim={4cm 10cm 3cm 10cm},clip,width=0.9\linewidth]{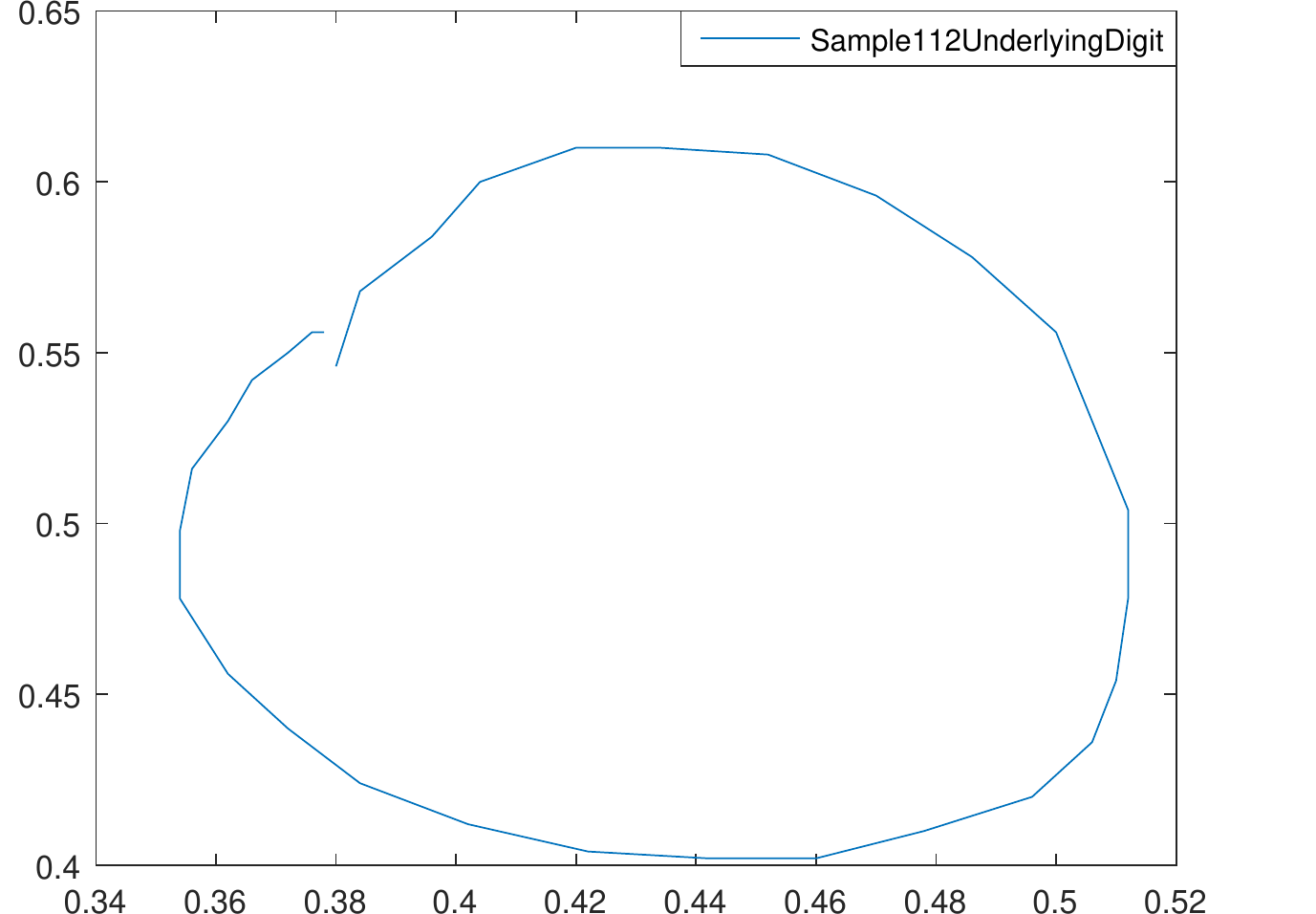}
\caption{Sample 112, underlying digit}
\end{subfigure}%
\begin{subfigure}{.5\textwidth}
\centering
\includegraphics[trim={4cm 10cm 3cm 10cm},clip,width=0.9\linewidth]{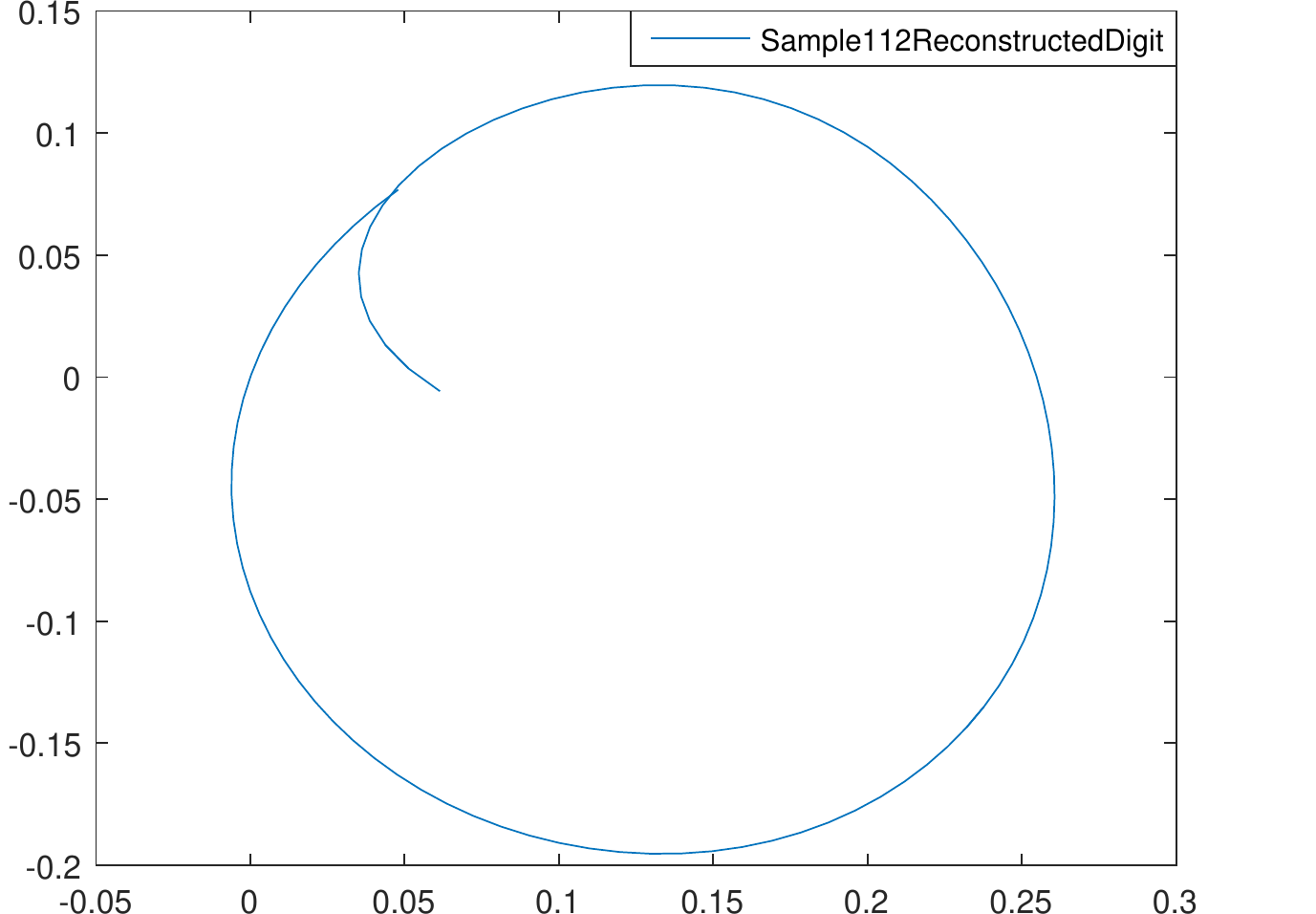}
\caption{Sample 112, reconstructed digit}
\end{subfigure}

\begin{subfigure}{.5\textwidth}
\centering
\includegraphics[trim={4cm 10cm 3cm 10cm},clip,width=0.9\linewidth]{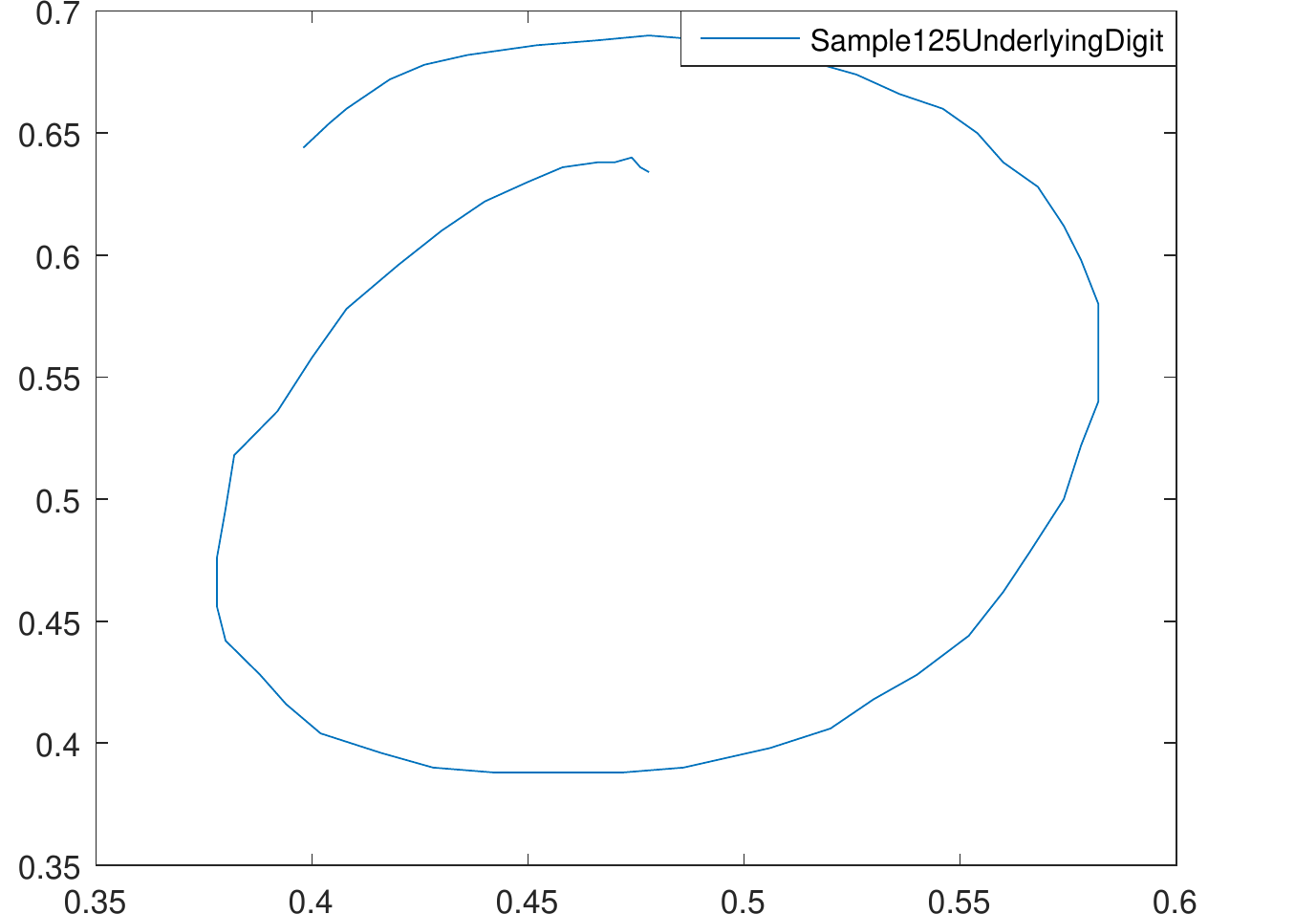}
\caption{Sample 125, underlying digit}
\end{subfigure}%
\begin{subfigure}{.5\textwidth}
\centering
\includegraphics[trim={4cm 10cm 3cm 10cm},clip,width=0.9\linewidth]{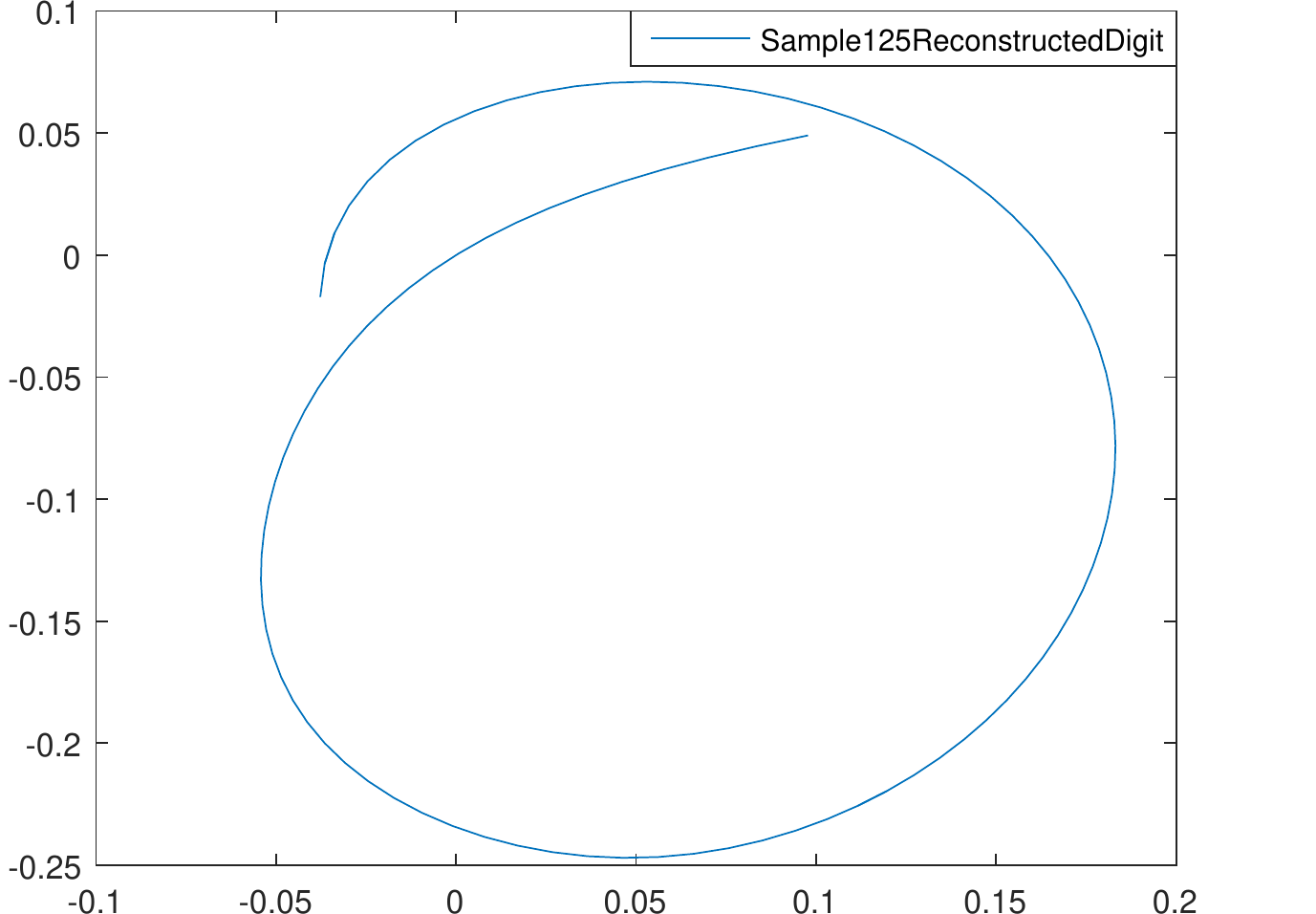}
\caption{Sample 125, reconstructed digit}
\end{subfigure}

\begin{subfigure}{.5\textwidth}
\centering
\includegraphics[trim={4cm 10cm 3cm 10cm},clip,width=0.9\linewidth]{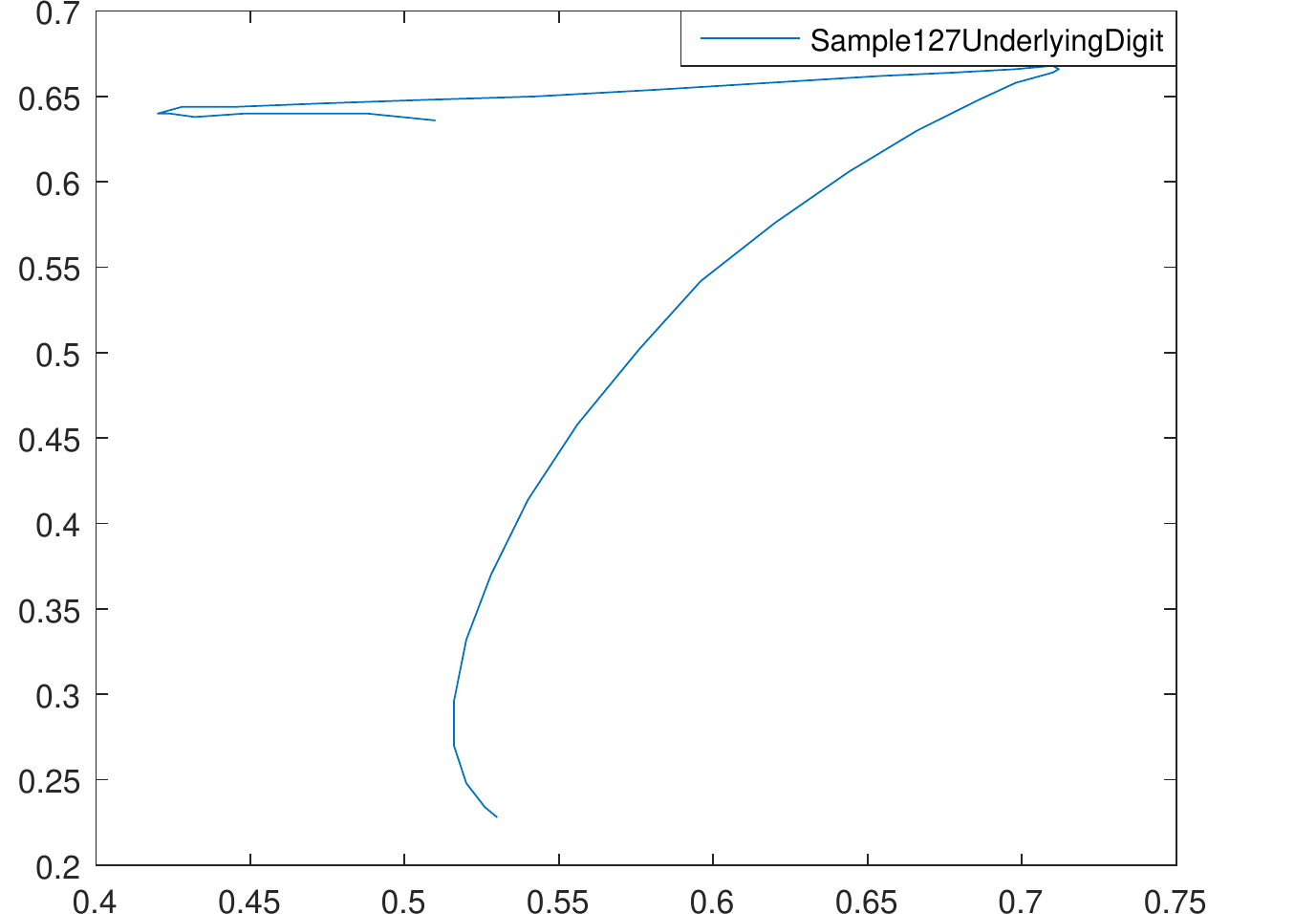}
\caption{Sample 127, underlying digit}
\end{subfigure}%
\begin{subfigure}{.5\textwidth}
\centering
\includegraphics[trim={4cm 10cm 3cm 10cm},clip,width=0.9\linewidth]{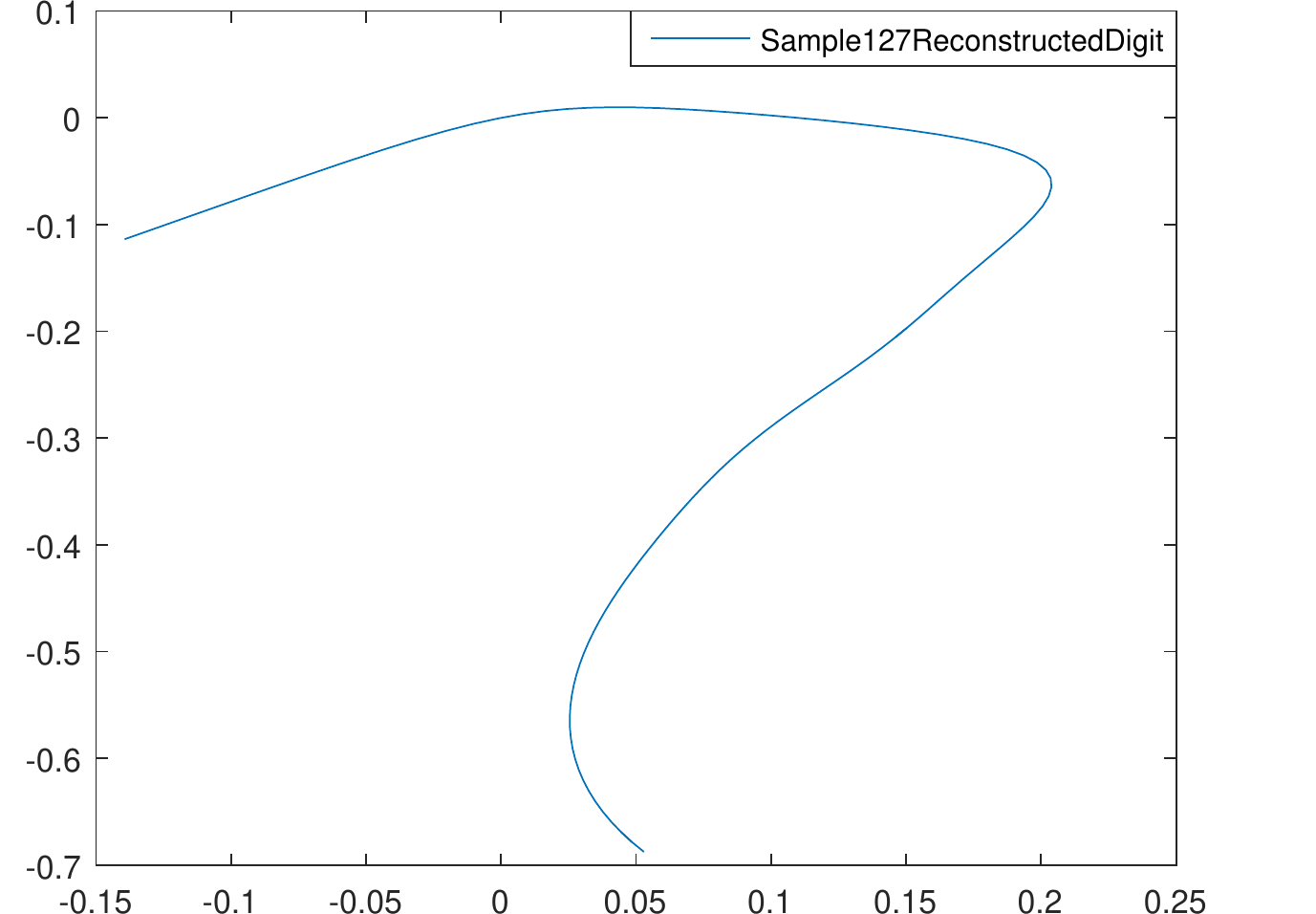}
\caption{Sample 127, reconstructed digit}
\end{subfigure}
\caption{Reconstruction of digits from the data set \cite{Dua:2017} using signature level $9$ and $10$}
\label{robustdataset-5}
\end{figure}

\begin{rmk}
From a computational point of view, in general if we want to use the insertion method to invert the signature, we need a nonlinear optimisation solver. However most of such solvers require a good initial guess. Hence when doing computation, we need to keep in mind that such factors may affect the results.\newline
If we compare the insertion method with the symmetrisation method described in \cite{chang2017signature}, from computation we saw that the insertion method is better in terms of efficiency, but the symmetrisation method gives more accurate approximation results.
\end{rmk}
\section{Conclusions}
In this article we have developed a practical algorithm for inverting the signature of a path by inserting elements into terms of the signature and comparing with other terms in the signature, and we have demonstrated computational results for inverting the signature of a piecewise linear path. In essence, the insertion algorithm depends on the relation
\begin{align*}
\left\lVert x-y\right\rVert=\frac{\left\lVert I_{p,n}(x)-I_{p,n}(y)\right\rVert}{\left\lVert \bar{S}_n\right\rVert}.
\end{align*}
Therefore, there is a possibility that the insertion method can be extended for inverting the signature of a more complicated path if $\left\lVert I_{p,n}(x)-I_{p,n}(y)\right\rVert$ decays faster than the norm of the normalised signature, $\left\lVert \bar{S}_n\right\rVert$.\newline
Moreover, we can see from the analysis that understanding the decay of the signature can be very helpful for signature inversion, therefore finding a lower bound for the terms in the signature of a path has its impacts on inverting the signature.\newline
In conclusion, the insertion method described in this article has potential in inverting the signature of a more general path, which is an interesting topic to study.
\bibliography{insertionref}        

\begin{thebibliography}{10}

\bibitem{Anderson:1999:LUG:323215}
E.~Anderson, Z.~Bai, C.~Bischof, L.~S. Blackford, J.~Demmel, Jack~J. Dongarra,
  J.~Du~Croz, S.~Hammarling, A.~Greenbaum, A.~McKenney, and D.~Sorensen.
\newblock {\em {LAPACK} Users' Guide (Third Ed.)}.
\newblock Society for Industrial and Applied Mathematics, Philadelphia, PA,
  USA, 1999.

\bibitem{boedihardjo2018non}
Horatio Boedihardjo and Xi~Geng.
\newblock A non-vanishing property for the signature of a path.
\newblock {\em arXiv preprint arXiv:1808.05903}, 2018.

\bibitem{libalgebra}
Stephen Buckley, Djalil Chafai, Lajos Gyurko, Arend Janssen, and Terry Lyons.
\newblock {Libalgebra C++ Package, Computational Rough Paths}.
\newblock \url{https://sourceforge.net/projects/coropa/}.

\bibitem{canonne2017}
Cl\'ement Canonne.
\newblock {A short note on Poisson tail bounds}.
\newblock
  \url{http://www.cs.columbia.edu/~ccanonne/files/misc/2017-poissonconcentration.pdf},
  2017.

\bibitem{chang2017signature}
Jiawei Chang, Nick Duffield, Hao Ni, Weijun Xu, et~al.
\newblock Signature inversion for monotone paths.
\newblock {\em Electronic Communications in Probability}, 22, 2017.

\bibitem{chang2018super}
Jiawei Chang, Terry Lyons, and Hao Ni.
\newblock Super-multiplicativity and a lower bound for the decay of the
  signature of a path of finite length.
\newblock {\em Comptes Rendus Mathematique}, 2018.

\bibitem{chen1958integration}
Kuo-sai Chen.
\newblock Integration of paths--a faithful representation of paths by
  non-commutative formal power series.
\newblock {\em Transactions of the American Mathematical Society},
  89(2):395--407, 1958.

\bibitem{chen1957integration}
Kuo-Tsai Chen.
\newblock {Integration of paths, geometric invariants and a generalized
  Baker-Hausdorff formula}.
\newblock {\em Annals of Mathematics}, pages 163--178, 1957.

\bibitem{chen1977iterated}
Kuo-Tsai Chen.
\newblock Iterated path integrals.
\newblock {\em Bulletin of the American Mathematical Society}, 83(5):831--879,
  1977.

\bibitem{Dua:2017}
Dua Dheeru and Efi Karra~Taniskidou.
\newblock {UCI Machine Learning Repository}.
\newblock University of California, Irvine, School of Information and Computer
  Sciences, 2017.
\newblock \url{http://archive.ics.uci.edu/ml}.

\bibitem{geng2017reconstruction}
Xi~Geng.
\newblock Reconstruction for the signature of a rough path.
\newblock {\em Proceedings of the London Mathematical Society},
  114(3):495--526, 2017.

\bibitem{hambly2010uniqueness}
Ben Hambly and Terry Lyons.
\newblock Uniqueness for the signature of a path of bounded variation and the
  reduced path group.
\newblock {\em Annals of Mathematics}, pages 109--167, 2010.

\bibitem{hoeffding1963probability}
Wassily Hoeffding.
\newblock Probability inequalities for sums of bounded random variables.
\newblock {\em Journal of the American statistical association},
  58(301):13--30, 1963.

\bibitem{lyons2007differential}
Terry~J Lyons, Michael Caruana, and Thierry L{\'e}vy.
\newblock {\em Differential equations driven by rough paths}.
\newblock Springer, 2007.

\bibitem{lyons2017hyperbolic}
Terry~J Lyons and Weijun Xu.
\newblock Hyperbolic development and inversion of signature.
\newblock {\em Journal of Functional Analysis}, 272(7):2933--2955, 2017.

\bibitem{lyons2014inverting}
Terry~J Lyons and Weijun Xu.
\newblock Inverting the signature of a path.
\newblock {\em Journal of the European Mathematical Society}, 20(7):1655--1687,
  2018.

\bibitem{pfeffer2018learning}
Max Pfeffer, Anna Seigal, and Bernd Sturmfels.
\newblock Learning paths from signature tensors.
\newblock {\em arXiv preprint arXiv:1809.01588}, 2018.

\end{thebibliography}
\bibliographystyle{plain} 
\end{document}